\documentclass[12pt]{article}
\usepackage[utf8]{inputenc}
\usepackage{graphicx}
\graphicspath{ {images/} }
\usepackage{caption}
\usepackage{subcaption}
\usepackage{float}
\usepackage[width=160mm,top=25mm,bottom=25mm,bindingoffset=6mm]{geometry}
\usepackage{fancyhdr}
\usepackage[style=numeric,sorting=nty, defernumbers=false]{biblatex}
\usepackage{amsmath,amsthm,amsfonts,amssymb,amscd}
\usepackage{bbm}
\usepackage{lastpage}
\usepackage{enumerate}
\usepackage{mathrsfs}
\usepackage{xcolor}
\usepackage{listings}
\usepackage{array}
\usepackage[T1]{fontenc}
\usepackage{tikz-cd}
\usepackage{titlesec}
\usepackage{tikz}
\usepackage{tensor}
\usepackage[toc,page]{appendix}
\usepackage{datetime}
\usepackage{verbatim}
\usepackage{mathtools}
\usepackage{imakeidx}
\usepackage[export]{adjustbox}
\usepackage{setspace}
\makeindex[intoc]
\usepackage{hyperref}
\hypersetup{%
  colorlinks=true,
  linkcolor=blue,
  citecolor=magenta,
  linkbordercolor={0 0 1}
}

\DeclareBibliographyCategory{cited}
\AtEveryCitekey{\addtocategory{cited}{\thefield{entrykey}}}

\def\Xint#1{\mathchoice
{\XXint\displaystyle\textstyle{#1}}%
{\XXint\textstyle\scriptstyle{#1}}%
{\XXint\scriptstyle\scriptscriptstyle{#1}}%
{\XXint\scriptscriptstyle\scriptscriptstyle{#1}}%
\!\int}
\def\XXint#1#2#3{{\setbox0=\hbox{$#1{#2#3}{\int}$ }
\vcenter{\hbox{$#2#3$ }}\kern-.6\wd0}}

\def\dashint{\Xint-}

\DeclareFontFamily{U}{wncy}{}
    \DeclareFontShape{U}{wncy}{m}{n}{<->wncyr10}{}
    \DeclareSymbolFont{mcy}{U}{wncy}{m}{n}
    \DeclareMathSymbol{\Sh}{\mathord}{mcy}{"58}

\titleformat{\section}[hang]{\Large\normalfont\bfseries\raggedright}
{\thesection.}{1ex}{}[]
\titleformat{\subsection}[hang]
{\large\normalfont\bfseries\raggedright}
{\thesubsection.}{1ex}{}[]

\newdateformat{monthyeardate}{%
\monthname[\THEMONTH], \THEYEAR}

\begin{document}

\theoremstyle{plain}
\newtheorem{prop}{Proposition}[section]
\theoremstyle{plain}
\newtheorem{conj}{Conjecture}[section]
\theoremstyle{plain}
\newtheorem*{conj*}{Conjecture}
\theoremstyle{plain}
\newtheorem{thm}[prop]{Theorem}
\theoremstyle{plain}
\newtheorem{cor}[prop]{Corollary}
\theoremstyle{plain}
\newtheorem{lemma}[prop]{Lemma}
\theoremstyle{definition}
\newtheorem{defn/}[prop]{Definition}

\newenvironment{defn}
{\renewcommand{\qedsymbol}{$\diamond$}%
	\pushQED{\qed}\begin{defn/}}
	{\popQED\end{defn/}}

\theoremstyle{remark}
\newtheorem{remark/}[prop]{Remark}

\newenvironment{remark}
{\renewcommand{\qedsymbol}{$\diamond$}%
	\pushQED{\qed}\begin{remark/}}
	{\popQED\end{remark/}}

\theoremstyle{remark}
\newtheorem{example/}[prop]{Example}

\newenvironment{example}
{\renewcommand{\qedsymbol}{$\diamond$}%
	\pushQED{\qed}\begin{example/}}
	{\popQED\end{example/}}

\newenvironment{claim}[1]{\par\noindent\underline{Claim:}\space#1}{}
\newenvironment{claimproof}[1]{\par\noindent\underline{Proof:}\space#1}{\leavevmode\unskip\penalty9999 \hbox{}\nobreak\hfill\quad\hbox{$\centerdot$}}

\newcommand{\N}{\mathbb{N}}
\newcommand{\C}{\mathbb{C}}
\newcommand{\R}{\mathbb{R}}
\newcommand{\Z}{\mathbb{Z}}
\newcommand{\Q}{\mathbb{Q}}
\newcommand{\bbP}{\mathbb{P}}
\newcommand{\csigma}{\mathfrak{S}}
\newcommand{\1}{\mathbbm{1}}
\newcommand{\bb}[1]{\mathbb{#1}}
\newcommand{\mcal}[1]{\mathcal{#1}}
\newcommand{\T}{\mathbb{T}}
\newcommand{\A}{\mathcal{A}}
\renewcommand{\S}{\mathcal{S}}
\renewcommand{\div}{\operatorname{div}}
\newcommand{\supp}{\operatorname{supp}}
\newcommand{\ma}{\text{max}}
\newcommand{\D}{\mcal{D}}
\newcommand{\Des}{\text{Des}}
\newcommand{\I}{\mcal{I}}
\newcommand{\pv}{\text{ p.v.}}
\renewcommand{\Re}{\operatorname{Re}}

\begin{titlepage}
    \begin{center}

        \vspace*{2cm}

        \Huge
        \textbf{An Introduction to Pointwise Sparse Domination}

        \vspace{3cm}

        \includegraphics[width=0.4\textwidth]{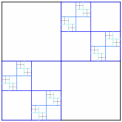}

        \vspace{3cm}

        \Large
        \text{Rodrigo Duarte}

        \vspace{1cm}

        \monthyeardate\today

    \end{center}
\end{titlepage}

\pagenumbering{roman} 
\pagestyle{plain}
\tableofcontents

\clearpage
\section*{Introduction}
\addcontentsline{toc}{section}{Introduction}
\pagenumbering{arabic}

The goal of this expository paper is to give a self-contained introduction to sparse domination. This is a method relying on techniques from 
dyadic Harmonic Analysis which has received a lot of attention in recent years. Essentially, it allows for a unified approach to proving weighted norm inequalities 
for a large variety of operators. In this work, we will introduce the basic ideas of dyadic Harmonic Analysis, which we use 
to build up to the main result we discuss on pointwise sparse domination, which is the Lerner-Ombrosi theorem. We also give applications of this theorem to some 
families of operators, mainly relating to singular integral operators. The text has been structured so as to motivate 
the introduction of new ideas through the lens of solving specific problems in Harmonic Analysis. Before starting the study of dyadic methods, we begin by 
motivating the need for weighted estimates. 

One classical problem in Harmonic Analysis is that of understanding the pointwise convergence behavior of Fourier series. A celebrated theorem of Lennart Carleson
\footnote{See \cite{Carleson}.} states that for a function \(f\in L^2(\T)\) we have 
\[ 
    \sum_{|n| \leq N}\hat{f}(n)e^{2\pi i n x} \xrightarrow[N\rightarrow +\infty]{} f(x),\text{ for a.e. }x\in \T,
\] 
where we use the convention of inserting the factor of \(2\pi\) in the exponential when defining the Fourier transform. That is,
\[ 
    \hat{f}(n) = \int_{\T}e^{-2\pi i nx}f(x)dx 
\] 
and \(\T = \R / \Z\). As usual, results on pointwise almost everywhere convergence of Fourier series rely on estimates for the associated maximal operator. 
Using transference arguments, 
the problem can be stated on the real line instead of the torus.\footnote{See for instance Theorem 4.3.12 in \cite{Cgrafakos}.} The pointwise convergence 
result then follows from the estimate
\[ 
    \|Cf\|_{L^{2, \infty}(\R)}\lesssim \|f\|_{L^2(\R)},\text{ for all }f \in L^2(\R),
\] 
where \(C\) is the Carleson operator,\index{Carleson operator} defined by
\[ 
    Cf(x) = \sup_{R > 0}\left|\int_{-R}^R \hat{f}(\xi)e^{2\pi i \xi x}d\xi \right|. 
\] 
This estimate was later improved by Hunt to a general Lebesgue exponent.\footnote{See \cite{Hunt}.} More precisely, if \(f\in L^p(\R)\), where \(1 < p < \infty\), then 
\[ 
    \|Cf\|_{L^p(\R)}\lesssim \|f\|_{L^p(\R)},\text{ for all }f\in L^p(\R). 
\] 
In turn, this estimate implies some pointwise convergence results. Firstly, if \(f\in L^p(\R)\) with \(1 \leq p \leq 2\), then \(\hat{f}\in L^{p'}\) by the 
Hausdorff-Young inequality, which implies that \(\hat{f}\) is locally integrable. In turn, this means that the partial Fourier integrals 
\[ 
    \int_{|\xi|\leq R}\hat{f}(\xi)e^{2\pi i \xi x}d\xi 
\] 
are well-defined. Using the Carleson-Hunt theorem one obtains
\[ 
    \int_{|\xi|\leq R}\hat{f}(\xi)e^{2\pi i \xi x}d\xi \xrightarrow[R\rightarrow +\infty]{} f(x),\text{ for a.e. }x\in \R,
\] 
for any \(f\in L^p(\R)\), where \(1 < p \leq 2\). There are several natural questions one might ask about variants of this problem. One example is to think about 
what happens near \(p = 1\) and there are several results proving pointwise almost everywhere convergence for spaces which lie between \(L^p(\T), p > 1\), and 
\(L^1(\T)\).\footnote{See for instance Antonov's theorem \cite{antonov}.} However, another natural question would be to ask what happens in higher dimensions. 
The situation in higher dimensions becomes very interesting 
because there are several ways to define the partial Fourier integrals. One way is to consider the integrals over cubes
\[ 
    \int_{[-R, R]^d}\hat{f}(\xi)e^{2\pi i \xi \cdot x}d\xi. 
\] 
In this case, one can reduce the problem to several applications of the one-dimensional result, thereby showing that 
\[ 
    \int_{[-R, R]^d}\hat{f}(\xi)e^{2\pi i \xi \cdot x}d\xi \xrightarrow[R\rightarrow +\infty]{}f(x), \text{ for a.e. }x\in \R^d,
\] 
where \(f\in L^p(\R^d)\) with \(1 < p \leq 2\).\footnote{For a proof see Theorems 4.3.16 and 4.3.12 in \cite{Cgrafakos}.} However, another natural approach would be to consider the spherical partial Fourier integrals
\[ 
    S_Rf(x) := \int_{|\xi|\leq R}\hat{f}(\xi)e^{2\pi i \xi \cdot x}d\xi. 
\] 
This case is much more subtle than when integrating over cubes. Intuitively, the fact that we are projecting the frequencies onto a region with 
a curved boundary allows one to use the infinite tangent directions to construct Kakeya type sets which yield counter-examples to boundedness properties. 
An important example is Fefferman's ball multiplier theorem (see \cite{fefferman}), which states that the operator 
\[ 
    f \mapsto (\chi_{B(0, 1)}\hat{f})^\vee 
\] 
is bounded on \(L^p(\R^d)\) with \(d > 1\) if and only if \(p = 2\). To prove pointwise convergence of the spherical Fourier integrals we are interested 
in obtaining estimates for the spherical Carleson operator
\[ 
    C_Sf(x) = \sup_{R > 0}\left|S_Rf(x)\right|.
\] 
The ball multiplier theorem then shows that \(C_S\) is not bounded on \(L^p(\R^d)\) with \(p \neq 2\). This leaves open the case when \(p = 2\).

\begin{conj*}[Conjecture 9.19 from \cite{lacey}]
    Let \(d > 1\). Then,
    \[ 
        \|C_Sf\|_{L^{2,\infty}(\R^d)}\lesssim \|f\|_{L^2(\R^d)},\text{ for all }f\in L^2(\R^d). 
    \] 
\end{conj*}

This conjecture is likely to be very difficult, but it is interesting to consider some simplifications of the general setup. One idea is to think about 
functions which have symmetries relevant to the problem at hand. Given that we are using spherical Fourier integrals it is natural to consider the case of 
radial functions. Suppose then that \(f\) is a radial function, with radial projection \(f_0\), that is, \(f(x) = f_0(|x|)\). 
Since the Fourier transform of a radial function is radial, and also given that \(S_Rf\) 
is defined as a projection in frequency to a ball, we see that \(S_Rf(x)\) is itself a radial function. In turn, this means that \(C_Sf\) will be radial. 
Let's denote its radial projection by \(Tf_0(\rho)\). We expect \(T\) to be related to the Carleson operator in one dimension, however, the exact relationship is unclear at 
this point. We now have
\[ 
    \begin{split}
        \|C_Sf\|_{L^p(\R^d)} &= \left(\int_{\R^d}|Tf_0(|x|)|^pdx\right)^{1/p} \\ 
                             &\sim \left(\int_0^{+\infty}|Tf_0(\rho)|^p\rho^{d-1}d\rho\right)^{1/p} \\ 
                             &= \|Tf_0\|_{L^p(\R^+, \rho^{d-1}d\rho)}.
    \end{split}
\]
This way, we see that the boundedness of the spherical Carleson operator applied to radial functions follows from weighted norm inequalities for a certain 
one-dimensional operator \(T\), and the weights appear quite naturally in this problem. This line of reasoning was followed by Prestini, who proved in 
\cite{prestini} that \(C_S\) is bounded on \(L^p\) for radial functions when 
\[ 
    \frac{2d}{d+1} < p < \frac{2d}{d-1}. 
\]
This gives one reason to be interested in weighted estimates, but these show up 
in other contexts as well. For example, in quantum mechanics, the square of the absolute 
value of the wave function determines the probability density for the position of a particle, while its Fourier transform is 
associated to the probability density of the momentum.
In this context, it is natural to consider \(L^2\) functions \(f\) such that both the position and momentum operator applied to \(f\) map to \(L^2\). 
But this is the same as asking that \(f \in H^1 \cap L^2(|x|^2)\). Sometimes higher moments are of interest, and in those cases, one might instead consider
more general power law weights \(|x|^\gamma\). It is also useful to consider weights when we compare \(L^p\) spaces with different exponents. 
If we are working in a space with finite measure then we have nested \(L^p\) spaces. However, if the space has infinite measure, this is no longer true. 
Despite this, we can still compare different exponents by introducing weights. For example, if \(p < q\),
\[
\|f\|_{L^p(\R^d)}\leq \|f\|_{L^q(w)}\left(\int_{\R^d}\frac{1}{w(x)^\frac{p}{q-p}}dx\right)^\frac{q-p}{pq}.
\]
To do this of course we cannot consider power law weights for then the right-hand side always blows up. 
In this context, it is often useful to modify power law weights to remove the non-smoothness at the origin. 
Usually one then considers the inhomogeneous weights \(w(x) = \langle x \rangle^\gamma\), 
where \(\langle x \rangle = (1 + |x|^2)^{1/2}\). There are also some other situations that might warrant the study of some specific families of weights besides 
the examples above. For example, Xavier Cabre and Xavier Ros-Oton, studied in \cite{Cabre1} the 
regularity of stable solutions to reaction-diffusion problems in bounded domains with symmetry of double revolution. Motivated by this problem, 
they were led to the consideration of Sobolev inequalities with monomial weights \(w(x) = |x_1|^{\gamma_1}\dots |x_d|^{\gamma_d}\) 
(see \cite{Cabre2}).

Weighted norm inequalities received a considerable amount of attention in the 1970s after Muckenhoupt's work on the Hardy-Littlewood maximal operator (see 
\cite{muckenhoupt}).
These results were mainly one-weight qualitative weighted norm inequalities. By qualitative we mean that the exact dependence of the implicit constant on the 
weight was unknown. More recently, motivated by problems in quasi-conformal theory, several researchers became interested in obtaining quantitative weighted 
estimates. For some time, it was conjectured that a Calderón-Zygmund operator \(T\) should satisfy the estimate
\[
\|Tf\|_{L^2(w)}\lesssim [w]_{A_2}\|f\|_{L^2(w)},
\]
where the implicit constant does not depend on the weight. This became known as the \(A_2\) conjecture and was solved by Hytönen (\cite{Hytonen}) in 2012. 
Shortly after, Andrei Lerner published a simpler proof of the same result which introduced many of the ideas that led to the method of sparse 
domination (see \cite{LernerA2}). These methods have now been found to have a wide range of applicability, and many operators can be controlled by sparse 
operators. For example, there are sparse domination results for pseudodifferential operators, Bochner-Riesz multipliers, singular oscillatory integrals, 
rough singular integrals, variational Carleson operators, maximal spherical functions, the Hilbert transform along curves and many others. For more 
information about recent sparse domination results see \cite{MariaSparseRevolution} and references therein. One of the most general results on pointwise 
sparse domination is the theorem of Lerner and Ombrosi in \cite{LernerOmbrosi}. Our main goal is to present the proof of this theorem in a self-contained 
and motivated way. 

Although the method of sparse domination was originally motivated by quantitative weighted estimates, it is also useful for obtaining two-weight norm 
inequalities. In this work, we focus less on the exact dependence of the implicit constants on the weights, seeing the method of sparse domination as a useful 
approach to provide unified proofs for two-weight norm inequalities for a vast array of operators. After proving the Lerner-Ombrosi theorem in Section 
\ref{sec: general sparse domination}, we 
will also apply the theorem to some concrete operators.

\clearpage
\section{Dyadic Lattices}\label{sec: dyadic lattices}
\pagestyle{fancy}

In this section, we will explore the fundamental ideas of dyadic Harmonic Analysis.
Dyadic methods have had a long history in Harmonic Analysis and many important results rely on some kind of dyadic decomposition. 
In general terms, this refers to the idea of partitioning space into regions that fit neatly into one another capturing scales with geometric ratios, 
usually given by powers of two. One example of this is the idea of using dyadic intervals \(I = [2^{-k}m, 2^{-k}(m+1)[\) to decompose a function as 
\[
f = \sum_{I}\langle f , h_I\rangle h_I,
\]
where \(h_I\) is the Haar function associated to \(I\). The Haar system is a complete orthonormal system in \(L^2(\R)\) with properties which are 
much easier to understand than for non-dyadic decompositions. One may think of dyadic methods as providing a model of the real 
line where we trade the advantages of having a continuum for the advantages of having an easy-to-work-with algebraic structure. However, in the process of 
stripping away the notion of a continuum, the dyadic model of the line does not go so far as a complete discretization of the line, like \(\Z\). The problem 
with using \(\Z\) as a model for the line is the lack of small scales. This way, the dyadic model finds a careful balance between \(\Z\) and \(\R\). On the one 
hand, the dyadic model has an algebraic structure that makes its use simple, but on the other hand, it is still able to capture all the relevant scales that 
are present in the continuum of the line. These are, of course, very imprecise statements, and we will shortly build on these intuitions to show exactly how 
to use dyadic methods to prove many non-obvious statements. 

Dyadic methods are useful in several different ways. One approach is to use dyadic 
structures to decompose the domain space of a function, using this decomposition to identify key elements in the behavior of the function. This is 
for example the approach used in Calderón-Zygmund theory, where dyadic cubes are fundamental in defining the good and the bad part of a function. 
Another approach is to decompose the frequency space in dyadic blocks, which is used for example in Littlewood-Paley theory. Yet another option is 
to employ a dyadic decomposition of phase space, in such a way as to respect the uncertainty principle, which is the case in time-frequency analysis. 
Another common approach is to consider a dyadic version of an operator one is interested in studying to serve as a toy model for this operator's behavior. 
It is usually the case that dyadic variants are better behaved and easier to deal with than their continuous counterparts and therefore they serve as good 
toy models. However, for some operators, they are more useful than that. Perhaps surprisingly, in some cases, it is possible to prove results about 
continuous operators from their dyadic counterparts. An example is the result of Petermichl on the dyadic representability of the Hilbert transform 
(\cite{Petermichl}). Intuitively, the Hilbert transform can be recovered from averages of dyadic shifts of the dyadic Hilbert transform, from which 
one can deduce properties of the Hilbert transform from those of the dyadic Hilbert transform.

In this section, our main goal is to introduce the most important ideas of dyadic Harmonic Analysis, 
and we start by motivating the notion of a dyadic lattice by investigating the \(L^p\) boundedness properties of the 
fractional maximal operator.

\subsection{The basic definitions}\label{sec: the basic definitions}

Given \(f\in L^1_\text{loc}(\R^d)\) we define the fractional maximal\index{fractional maximal operator} operator as
\[
M_\alpha f(x) = \sup_{Q \ni x} |Q|^\frac{\alpha}{d}\dashint_Q|f| = \sup_{Q\ni x}|Q|^\frac{\alpha}{d}\frac{1}{|Q|}\int_Q|f|,
\]
where the supremum runs over all cubes with sides parallel to the coordinate axes. In fact, in this text, all cubes will be assumed to have the form
\[
Q = [x_1, x_1 + l[ \times \dots \times [x_d, x_d + l[\ =: Q(x, l),
\]
where \(x = (x_1, \dots, x_d) = x(Q)\) is called the corner of \(Q\), and \(l = l(Q)>0\) is the sidelength. When \(\alpha = 0\) we get the Hardy-Littlewood maximal operator\index{Hardy-Littlewood maximal operator} which we denote simply by \(Mf\) instead of \(M_0f\).

Our immediate goal is to obtain the \(L^p\) boundedness properties of the fractional maximal operator. This can be done by using Vitali-type covering lemmas, however, we wish to proceed here in a different way so as to motivate the introduction of dyadic lattices. Before proving strong type estimates we begin by investigating the weak type \((p, q)\) inequality
\[
|\{x\in \R^d: M_\alpha f(x) > \lambda \}|\lesssim \lambda^{-q}\|f\|_{L^p}^q.
\]
Suppose that \(f\) is not identically zero, and without loss of generality assume also that \(f\) is non-negative. Now let \(x\in \R^d\) be such that \(M_\alpha f(x) > \lambda\), for some \(\lambda > 0\). This means that we can find some cube \(Q_x \ni x\) such that 
\[
|Q_x|^\frac{\alpha}{d}\dashint_{Q_x}f > \lambda.
\]
If we put \(E_\lambda = \{x\in \R^d: M_\alpha f(x) > \lambda\}\) we see that
\[
E_\lambda = \bigcup_{x\in E_\lambda} Q_x.
\]
In general, the cubes \(Q_x\) will have a lot of overlap and some of these cubes will be redundant, for instance, if they are completely contained in another cube from the family. Suppose that, after removing all the redundant cubes we are left with a countable family \(\{Q_j\}\) of disjoint cubes such that
\[
E_\lambda = \bigcup_{j}Q_j\text{ and }|Q_j|^\frac{\alpha}{d} \dashint_{Q_j}f > \lambda.
\]
Then, 
\[
\begin{split}
  \lambda^q|E_\lambda| & = \sum_j \lambda^q|Q_j|\leq \sum_j |Q_j|\left(|Q_j|^\frac{\alpha}{d}\dashint_{Q_j}f\right)^q \leq \sum_j |Q_j|^{1+\frac{\alpha}{d}q -q+\frac{q}{p'}}\left(\int_{Q_j}f^p\right)^\frac{q}{p}.
\end{split}
\]
This suggests we should take \(p, q, \alpha\) so that
\[
1+\frac{\alpha}{d}q-q+\frac{q}{p'} = 0 \iff \frac{1}{q}+\frac{\alpha}{d} = \frac{1}{p}.
\]
Assuming this condition and also that \(p\leq q\), we get
\[
\lambda^q|E_\lambda| \leq \sum_j \left(\int_{Q_j}f^p\right)^\frac{q}{p}\leq \left(\sum_j \int_{Q_j}f^p\right)^\frac{q}{p} \leq \|f\|_{L^p}^q,
\]
which yields the weak type \((p,q)\) estimate. 

Of course, this argument depends crucially on the fact that the non-redundant cubes are disjoint. However, this is not necessarily the case since the cubes can overlap in many complicated ways. So, for this argument to work, we would have to consider, instead of all cubes, only cubes belonging to some family \(\D\) with the following property: any two cubes from \(\D\) are either disjoint, or one is contained in the other. We will refer to this property as the `disjointness' property.\index{disjointness property} Another way of understanding the disjointness property is by considering the order \(Q \leq P\) given by inclusion \(Q\subseteq P\). Then, the disjointness property states that two cubes are comparable if and only if they intersect. If the supremum is taken over cubes coming from such a family \(\D\), then, after removing the redundant cubes, we are left only with maximal cubes, and therefore all these cubes would be pairwise disjoint allowing us to finish the argument. The question now becomes how to obtain such a family \(\D\). Note that, besides the requirement that this family should satisfy the disjointness property, we also want these cubes to be able to approximate well any given cube. In other words, the cubes from \(\D\) should be able to approximate any scale and any location in space. 

To construct such a family we start with any given cube, \(B\), which we call the base. If we want to approximate every location at the scale of \(B\), a simple way to do this is to just translate \(B\) to form a partition of \(\R^d\). In other words, we add to \(\D\) all the cubes of the form
\[
B + l(B)m,\ m\in \Z^d.
\]
Now let's think about cubes with scales smaller than \(B\). In order to preserve the disjointness property, any cube with a scale smaller than \(B\) that intersects \(B\) must be completely contained in it, and the same has to be true for any other cube of the family inside \(B\). To do this, we first form a partition of \(B\) using the cubes
\[
x(B) + \frac{1}{2}(B-x(B)) + \frac{l(B)}{2}u,\ u\in \{0, 1\}^d.
\]
These cubes are called the dyadic children\index{dyadic children} of \(B\), or the first order descendants\index{first order dyadic descendants} of \(B\), and we denote the collection of these \(2^d\) cubes by \(\Des_1(B)\) (see Figure \ref{fig: dyadic children}).

\begin{figure}[ht] 
\centering
\includegraphics[width=0.3\textwidth]{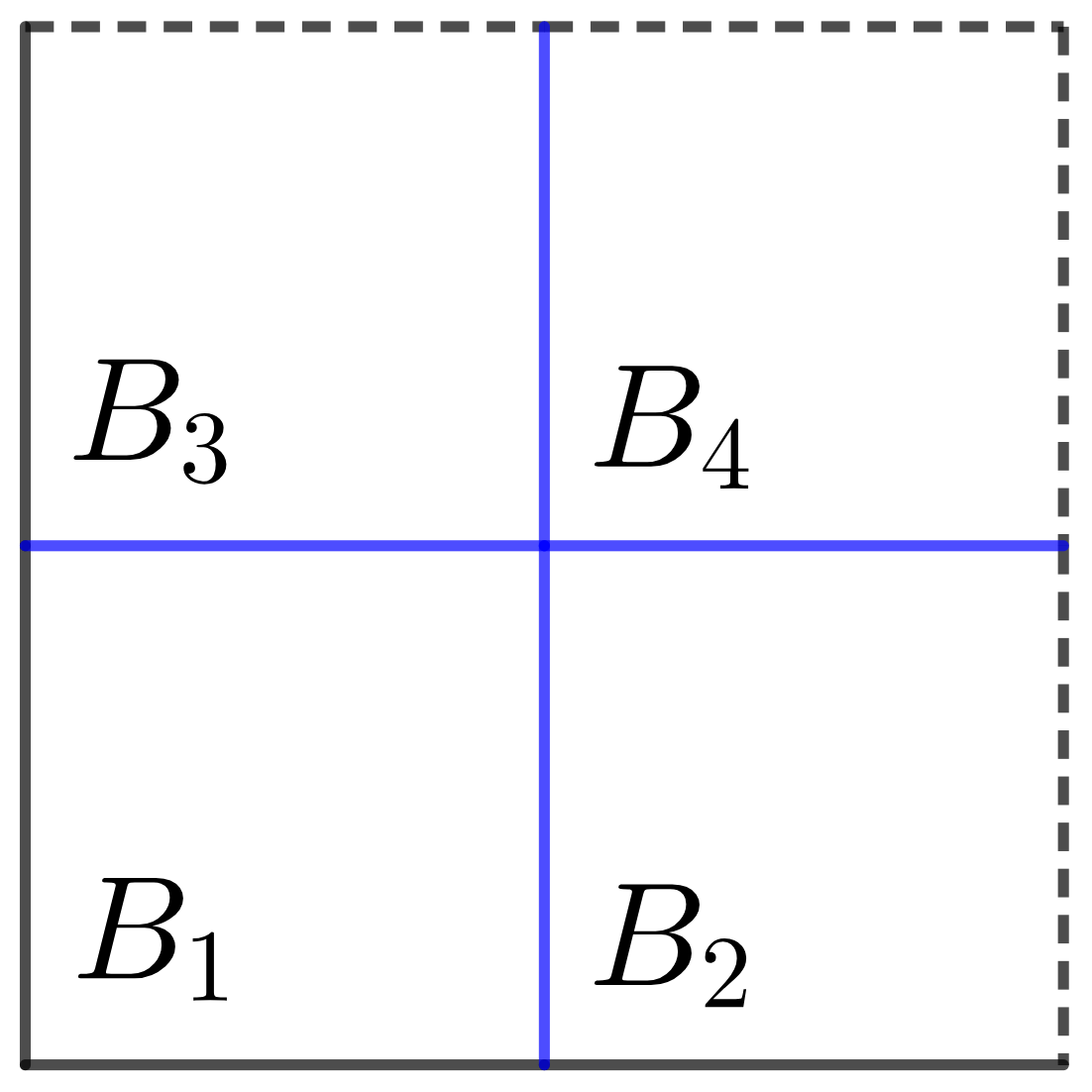}
\caption{The cube \(B\) with its dyadic children, \(B_1, B_2, B_3\) and \(B_4\).}\label{fig: dyadic children}
\end{figure}

We now add these cubes to \(\D\) and note that we maintain the property that either two cubes are comparable or they are disjoint. This way we are able to capture scales of the order of \(l(B)/2\). To approximate even smaller scales we simply keep going by adding the dyadic children of all cubes in \(\Des_1(B)\), thereby obtaining the family of second order descendants of \(B\), \(\Des_2(B)\). Proceeding in this way we get the family of all descendants\index{dyadic descendants} of \(B\),
\[
\Des(B) = \bigcup_{k\geq 0} \Des_k(B),
\]
where the cubes in \(\Des_k(B)\) are the \(k\)-th order descendants of \(B\), which have sidelength \(2^{-k} l(B)\).\footnote{Here we consider \(\Des_0(B) = \{B\}\).} At this point, we can approximate all small scales inside \(B\), but we still need to approximate small scales outside \(B\) as well, so we add to \(\D\) all the cubes in \(\Des(Q)\) for all cubes \(Q\in \D\) congruent to \(B\). This gives a structure that can be naturally understood in terms of generations. All the cubes congruent to \(B\) are said to be in generation zero, i.e.
\[
\D_0 = \{B + l(B)m : m\in \Z^d\}.
\]
Then, the \(k\)-th generation\index{dyadic generation} is 
\[
\D_k = \{Q: Q\in \Des_k(P)\text{ for some }P\in \D_0\},\ k\geq 1.
\]
The only thing we are missing in this construction is the addition of cubes with large scales. This introduces some element of choice into the construction. The idea is that we want to add a cube \(B_1\) such that \(B \in \Des_1(B_1)\). However, there are \(2^d\) possible cubes we could choose for \(B_1\). Indeed, any cube of the form
\[
x(B) + 2(B-x(B)) - l(B)u,\ u\in \{0, 1\}^d
\]
would work (see Figure \ref{fig: choices for parent}). 

\begin{figure}[ht]
    \centering
    \includegraphics[width=0.6\textwidth]{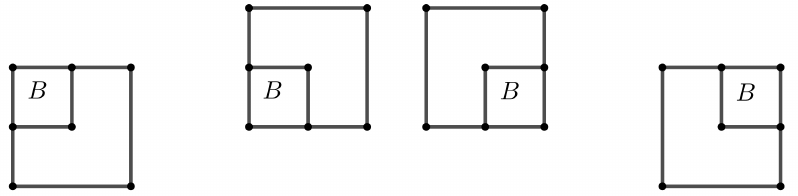}
    \caption{The cube \(B\) together with the possible choices for its parent.}
    \label{fig: choices for parent}
\end{figure}

Suppose we make some choice \(B_1\) associated to some \(u^{(1)}\in \{0, 1\}^d\). Then, \(B_1\) is called the parent\index{dyadic parent} of \(B\), or the first order ancestor\index{dyadic ancestor} of \(B\) and we denote this by \(B_1 = A_1(B)\). We then add \(B_1\) to \(\D\) together with all the translates \(B_1 + l(B_1)m,\ m\in \Z^d\). This forms the generation of order \(-1\), \(\D_{-1}\). We then make some choice for the parent of \(B_1\), say \(B_2\) associated with \(u^{(2)}\in \{0, 1\}^d\), we add it to the family with its translates, forming the generation of order \(-2\), \(\D_{-2}\). Continuing in this way, we obtain a sequence of cubes \(B_1, B_2, \dots, B_n,\dots\), such that \(B_n \in \Des_1(B_{n+1})\), each associated to some \(u^{(n)} \in \{0, 1\}^d\). At each step we get the corresponding generation \(\D_{-n} = \{B_n + l(B_n)m: m\in \Z^d\}\). One easily proves by induction that
\[
B_n = 2^n B - (2^n - 1)x(B) - l(B) \sum_{j=1}^n 2^{j-1}u^{(j)}.
\]
We finally obtain our family of cubes
\[
\D = \bigcup_{k \in \Z} \D_k.
\]
We call a family of cubes obtained this way a \textit{dyadic lattice}\index{dyadic lattice}. In general, a dyadic lattice can be defined from a base cube \(B\) and a choice of ancestors of \(B\), \(u\in (\{0,1\}^d)^{\N_1}\).\footnote{To avoid confusion we use the notation \(\N_k = \{k, k+1, k+2,\dots\}\).} If it is not clear from the context what these choices are we may write \(\D(B, u)\). In summary, we give the following definition:
\begin{defn}\label{definition of dyadic lattice}
Given a cube \(B\) and \(u\in (\{0,1\}^d)^{\N_1}\), \(u = (u^{(1)}, u^{(2)},\dots)\), the dyadic lattice with base \(B\) and ancestor sequence \(u\) is 
\[
\D(B, u) = \bigcup_{\substack{n\geq 1 \\ m\in \Z^d}}\Des(B_n(B,u) + 2^nl(B)m),
\]
where 
\[
B_n(B, u) = 2^n B - (2^n - 1)x(B) - l(B) \sum_{j=1}^n 2^{j-1}u^{(j)}.
\]
\end{defn}

Sometimes we may refer to these families as dyadic grids, instead of dyadic lattices. Any dyadic lattice \(\D(B, u)\) has a number of useful properties:
\begin{itemize}
    \item The cubes in \(\D_k\) are all congruent, have a sidelength of \(2^{-k}l(B)\), and form a partition of \(\R^d\);
    \item Any cube \(Q\in \D\) has exactly \(2^d\) first-order descendants, and exactly one first-order ancestor which we denote by \(A_1(Q)\). Note that the property of having exactly one ancestor is true also for \(A_1(Q)\) which gives a unique second-order ancestor for \(Q\), \(A_2(Q)\). This way we obtain a well-defined sequence of ancestors of \(Q\): \(A_1(Q), A_2(Q), A_3(Q), \dots\);
    \item If \(P, Q \in \D\), then \(P\cap Q \in \{\varnothing, P, Q\}\) (the disjointness property). Alternatively, we could say that for all \(P, Q\in \D\), either they are comparable or they are disjoint.
\end{itemize}
Definition \ref{definition of dyadic lattice} is easy to understand, but it is hard to work with in general. It is simpler to characterize a dyadic lattice by some of its fundamental properties.
\begin{prop}\label{characterization of dyadic lattices}
    A collection of cubes \(\D\) is a dyadic lattice if and only if the following are true:
    \begin{enumerate}
        \item We can split \(\D = \cup_{k\in \Z}\D_k\) such that \(\D_k = \{Q + l(Q)m: m\in \Z^d\}\), for some \(Q\in \D_k\);
        \item For every \(k\), there exists \(Q\in \D_k\) such that \(x(Q) + (1/2)(Q-x(Q)) \in \D_{k+1}\).
    \end{enumerate}
\end{prop}
\begin{proof}
    It is clear that, if \(\D\) is a dyadic lattice, then it satisfies these properties, so let's focus on the other direction. Suppose that \(\D\) is a collection of cubes satisfying properties 1 and 2. First note that if the representation in property 1, \(\D_k = \{Q + l(Q)m: m\in \Z^d\}\) works with \(Q\), then it must also work with any other cube in \(\D_k\). Now pick any cube \(B\in \D_0\), and call this the base. Then \(\D_0\) is exactly the generation of order zero with respect to \(B\). By the second property there is some cube \(Q\in \D_0\) such that the corner child \(x(Q)+(1/2)(Q-x(Q))\in \D_1\). Therefore, the cubes in \(\D_1\) are exactly of the form
    \[
    x(Q)+\frac{1}{2}(Q-x(Q)) + \frac{1}{2}l(Q)m,\ m\in \Z^d.
    \]
    Note also that we know there is some \(m^1\in \Z^d\) such that \(Q = B + l(B)m^1\). So we can write the cubes in \(\D_1\) as
    \[
    x(B) + \frac{1}{2}(B-x(B)) + \frac{1}{2}l(B)m,\ m\in \Z^d.
    \]
    Therefore, the cubes in \(\D_1\) are exactly the dyadic children of all the cubes in \(\D_0\). Therefore \(\D_1\) corresponds to the first dyadic generation with respect to \(B\). Of course, the same argument can be made for \(\D_2\) and so on, so \(\D_k\) for \(k\geq 0\) is formed exactly by the cubes in the \(k\)-th dyadic generation with respect to \(B\). 

    Now, by property 2, we know there is some \(Q\in \D_{-1}\) so that \(x(Q) + (1/2)(Q-x(Q))\in \D_0\). But we know the form of all cubes in \(\D_0\), which means that
    \[
    x(Q) + \frac{1}{2}(Q-x(Q)) = B + l(B)m,\text{ for some }m\in \Z^d.
    \]
    This equality holds if and only if the sidelengths and the corners match, so this is the same as 
    \[
    x(Q) = x(B) + l(B)m\text{ and }l(Q) = 2l(B).
    \]
    Now we want to see that there is a cube in \(\D_{-1}\) which is actually a parent of \(B\). Since all the cubes in \(\D_{-1}\) can be written as
    \[
    Q + l(Q)n,\ n\in \Z^d,
    \]
    this means that we want to find \(n\in \Z^d\) and \(u\in \{0, 1\}^d\) such that
    \[
    Q + l(Q)n = 2B - x(B) - l(B)u.
    \]
    Note that both cubes have the same sidelength, so this equality holds if and only if the corners match, i.e.
    \[
    x(Q) + 2l(B)n = x(B) - l(B)u.
    \]
    This is the same as
    \[
    x(B) + l(B)m + 2l(B)n = x(B) - l(B)u,
    \]
    and therefore it suffices to take \(n\in \Z^d\) and \(u\in \{0,1\}^d\) so that \(u = -(m + 2n)\). Of course, this is always possible because we can choose \(u\in \{0, 1\}^d\) such that \(m+u \in (2\Z)^d\). This shows the existence of a \(u^{(1)}\in \{0, 1\}^d\) such that
    \[
    B_1 = 2B-x(B)-l(B)u^{(1)}
    \]
    is a parent of \(B\) in \(\D_{-1}\). Thus, the cubes in \(\D_{-1}\) form exactly a generation of order \(-1\) with respect to \(B\) and the choice \(u^{(1)}\). Again, the same argument can be carried through for \(\D_{-2}\), and so on. We then obtain a sequence of choices \(u \in (\{0, 1\}^d)^{\N_1}\) such that \(\D = \D(B, u)\), which proves the result.
\end{proof}
One of the most common dyadic lattices is the standard dyadic lattice\index{standard dyadic lattice} defined by
\[
\D^\text{st} = \D([0,1[^d, u=(\vec{0},\vec{0},\dots))
\]
In other words, it is the dyadic lattice obtained starting from the base \(B = [0,1[^d\) and taking the sequence of ancestors to be \(B_1 = [0, 2[^d, B_2 = [0, 2^2[^d,\dots\), which gives
\[
\D^\text{st} = \{2^{-k}([0, 1[^d + m): m\in \Z^d, k\in \Z\}.
\]
The standard dyadic lattice is well-suited for most purposes. However, 
it does have a small technical problem which we will want to avoid at times. 
The problem is that there are cubes that are not contained in any dyadic cube from the standard dyadic lattice, 
which is the case for example for \(Q = [-1/2,1/2[^d\). This occurs because the standard dyadic lattice has more than one quadrant. 
To properly define this concept we first introduce the idea of a dyadic tower, which will be useful in several contexts. 
Since any given generation of dyadic cubes, \(\D_k\), forms a partition of \(\R^d\), this means that, 
given a point \(x\in \R^d\) there is a unique cube \(Q_k(x) \in \D_k\) that contains \(x\). 
The dyadic tower\index{dyadic tower over a point} over \(x\) is simply the set of all such cubes, \(T_x = \{Q_k(x): k\in \Z\}\). 
We call the set 
\[
H_x = \bigcup_{k\in \Z} Q_k(x)
\]
the quadrant\index{quadrant of a dyadic lattice} that contains \(x\). It is simple to see that if \(H_x \cap H_y \neq \varnothing\), then \(H_x = H_y\). So the set of quadrants forms again a partition of \(\R^d\). The inability to find a cube in the lattice that contains a given cube comes from the existence of multiple quadrants.
\begin{prop}\label{one quadrant lattices}
A dyadic lattice \(\D(B, u)\) has exactly one quadrant if and only if, for all compact sets \(K\) we can find some cube \(Q\in \D\) such that \(K\subseteq Q\).
\end{prop}
\begin{proof}
First, suppose that \(\D\) has one quadrant, and let \(x\in B\). Since \(\D\) has one quadrant we know that \(H_x = \R^d\), but \(x\in B\) so 
\[
\R^d = H_x = \bigcup_{k\in \Z}Q_k(x) = \bigcup_{n\geq 1} B_n.
\]
Now, given a compact set \(K\), the sets \(\{\text{Int}(B_n)\}_n\) form an open cover of \(K\), and therefore it has a finite subcover. But the cubes \(B_n\) form an increasing sequence, \(B_1\subseteq B_2\subseteq \dots\), which implies the existence of some \(n_0\) such that \(K\subseteq B_{n_0}\). 

For the other direction assume that for all \(K\), there exists \(Q\in \D\) such that \(K\subseteq Q\).
Fix \(x\in \R^d\), and for each \(N\in \N_1\) let \(K_N\) be the closed ball with center in \(x\) and radius \(N\). By assumption, there is some \(Q_N\in \D\) so that \(K_N \subseteq Q_N\). Now, all the cubes \(Q_N\) contain \(x\), and therefore they are nested. Without loss of generality, we may assume that \(Q_1\subseteq Q_2\subseteq \dots\). But then \(\{Q_N\}\) is a sequence of ancestors of \(Q_1\) with sidelength increasing without bound. Thus,
\[
H_x = \bigcup_{k\in \Z}Q_k(x) = \bigcup_{N\in \N_1}Q_N = \R^d.
\]
This shows that \(\D\) has exactly one quadrant. 
\end{proof}
It is not hard to construct dyadic lattices that have exactly one quadrant, since this condition is equivalent to having
\[
\bigcup_{n\geq 1}B_n = \R^d.
\]
For example, if we choose \(B = [0,1[^d\) and \(u\) such that \(u^{(1)} = (0,\dots, 0)\), \(u^{(2)} = (1,\dots, 1)\), \(u^{(3)} = (0,\dots,0)\), etc, then the dyadic lattice \(\D(B, u)\) has exactly one quadrant. It is also interesting to note that, if a dyadic lattice \(\D\) has only one quadrant, then it satisfies the following property: Any two cubes \(Q_1, Q_2\in \D\) have a common ancestor, i.e. there is some \(P\in \D\) so that \(Q_1, Q_2\in \Des(P)\). This is a simple consequence of Proposition \ref{one quadrant lattices} by taking \(K = \overline{Q_1\cup Q_2}\).

Now that we have seen the definition of dyadic lattices and their basic properties let us return to the matter at hand. We start by defining the dyadic fractional maximal operator.\index{dyadic fractional maximal operator} Given a dyadic lattice \(\D\) and \(f\in L^1_\text{loc}(\R^d)\) we define
\[
M_\alpha ^\D f(x) = \sup_{Q\in \D}\chi_Q(x) |Q|^\frac{\alpha}{d}\dashint_Q|f|.
\]
Now we can finally run the argument without any issues. In fact, we can do a bit better and let the maximal operator be defined with respect to a different measure. 
Suppose \(w\) is a locally integrable function such that \(0 < w(x) < \infty\) for a.e. \(x\in \R^d\). 
We call such a function a \textit{weight}\index{weight}.\footnote{For 
a brief overview on important properties of weights see Appendix \ref{appendix: weighted spaces}.} 
We then define
\[
M_{\alpha, w}^\D f(x) = \sup_{Q\in \D}\chi_Q(x) w(Q)^{\frac{\alpha}{d}-1}\int_Q|f(y)|w(y) dy.
\]
Now we can run the same argument as before to obtain the following result.
\begin{prop}\label{weak type inequality for the fractional maximal operator}
    Let \(\D\) be a dyadic lattice and let \(w\) be a weight such that \(w(H) = +\infty\) for every quadrant \(H\) of \(\D\). Let \(1\leq p\leq q\leq \infty\), \(\alpha \in [0, d]\) be such that
    \[
    \frac{1}{q}+\frac{\alpha}{d} = \frac{1}{p}.
    \]
    Then,
    \[
    \|M_{\alpha,w}^\D f\|_{L^{q, \infty}(\R^d, w)}\leq \|f\|_{L^p(\R^d, w)},\ \forall f\in L^p(\R^d,w).
    \]
\end{prop}
\begin{proof}
    Let \(f\in L^p(w)\) and without loss of generality assume that \(f\geq 0\). First, note that if \(q = \infty\) then \(1/p = \alpha/d\), and \(1/p' = 1-\alpha/d\). So, if we simply use Hölder's inequality,
    \[
    \begin{split}
    w(Q)^{\frac{\alpha}{d}-1}\int_Qf(y)w(y)dy &= w(Q)^{\frac{\alpha}{d}-1}\int_Qf(y)w(y)^\frac{1}{p}w(y)^\frac{1}{p'}dy\\
    &\leq w(Q)^{-1/p'} \|fw^{1/p}\|_{L^p}w(Q)^{1/p'}\\
    &\leq \|f\|_{L^p(w)}.
    \end{split}
    \]
    This shows that
    \[
    \|M_{\alpha, w}^\D f\|_{L^\infty}\leq \|f\|_{L^p(w)}.
    \]
    So we may now assume that \(q<\infty\), which also implies that \(\alpha < d\) and \(p<\infty\). Given \(\lambda > 0\) put \(E_\lambda = \{x\in \R^d: M_{\alpha, w}^\D f(x) > \lambda\}\). Our goal is to show that 
    \[
    \lambda^q w(E_\lambda)\leq \|f\|_{L^p(w)}^q.
    \]
    Consider the family
    \[
    \mcal{F} = \{Q\in \D: w(Q)^{\frac{\alpha}{d}-1}\int_Q f w > \lambda\}.
    \]
    It is straightforward to see that
    \[
    E_\lambda = \bigcup_{Q\in \mcal{F}} Q.
    \]
    Now pick \(Q\in \mcal{F}\) and consider its sequence of ancestors \(A_1(Q), A_2(Q)\), etc. Since \(w(\cup_k A_k)=+\infty\) we see that
    \[
    w(A_k(Q))^{\frac{\alpha}{d}-1}\int_{A_k(Q)}f w \leq w(A_k(Q))^{-1/q}\|f\|_{L^p(w)}\xrightarrow[k\rightarrow +\infty]{} 0.
    \]
    This means that there exists a unique \(n\in \N_0\)\footnote{For \(n=0\) we put \(A_0(Q) = Q\).} such that
    \[
    w(A_n(Q))^{\frac{\alpha}{d}-1}\int_{A_n(Q)}f w>\lambda \text{ but }w(A_k(Q))^{\frac{\alpha}{d}-1}\int_{A_k(Q)}f w \leq \lambda, \forall k> n.
    \]
    Denote \(A_n(Q)\) by \(Q^\text{max}\) and consider the family of maximal cubes
    \[
    \mcal{F}^\text{max} = \{Q^*: Q^* = Q^\text{max} \text{ for some }Q\in \mcal{F}\}.
    \]
    These cubes are maximal in the sense that
    \[
    P \in \D \text{ and }Q^*\subsetneq P \implies w(P)^{\frac{\alpha}{d}-1}\int_{P}fw \leq \lambda.
    \]
    Now note that the maximal cubes still cover our set,
    \[
    E_\lambda = \bigcup_{Q\in \mcal{F}^\text{max}}Q,
    \]
    moreover, the cubes in \(\mcal{F}^\text{max}\) are pairwise disjoint because we are dealing with dyadic cubes. This means that we have a countable collection of pairwise disjoint cubes, \(\{Q_j\}_j\), such that 
    \[
    E_\lambda = \cup_j Q_j\text{ and }w(Q_j)^{\frac{\alpha}{d}-1}\int_{Q_j}f w > \lambda. 
    \]
    Therefore we have that
    \[
    \begin{split}
      \lambda^q w(E_\lambda) &= \sum_j \lambda^q w(Q_j) \\
      &\leq \sum_j w(Q_j)^{q\frac{\alpha}{d}-q+1}\left(\int_{Q_j}f w\right)^q \\
      &\leq \sum_j w(Q_j)^{q\frac{\alpha}{d}-q+1+\frac{q}{p'}}\left(\int_{Q_j}f^p w\right)^{q/p}\\
      &= \sum_j\left(\int_{Q_j}f^p w\right)^{q/p}.
    \end{split}
    \]
    At this point, it becomes important to use the assumption that \(p\leq q\). We obtain,
    \[
    \lambda^q w(E_\lambda) \leq \left(\sum_j \int_{Q_j}f^p w\right)^{q/p} \leq \|f\|_{L^p(w)}^q.
    \]
    
\end{proof}
\begin{remark}
    In this proposition, it was important to assume that \(p\leq q\) and \(1/q + \alpha/d = 1/p\). These conditions imply some 
    restrictions on what values \(p, q\) and \(\alpha\) can take. For example, the condition \(\alpha \in [0, d]\) is a consequence of 
    the other assumptions. In fact, 
    \[
    0 \leq \frac{1}{p}-\frac{1}{q} = \frac{\alpha}{d}\leq \frac{1}{q} + \frac{\alpha}{d} = \frac{1}{p} \leq 1,
    \]
    so \(\alpha \in [0, d]\). Moreover, 
    \[
    \frac{1}{p} = \frac{1}{q} + \frac{\alpha}{d} \geq \frac{\alpha}{d}, \text{ and }\frac{1}{q} + \frac{\alpha}{d} = \frac{1}{p}\leq 1,
    \]
    so \(p \leq d/\alpha\) and \(d/(d-\alpha) \leq q\). One way to think about this is that for each choice of \(\alpha, p\) 
    satisfying \(0\leq \alpha \leq d\) and \(1\leq p \leq d/\alpha\) we define
    \[
    q = \frac{pd}{d-\alpha p},
    \]
    which satisfies \(d/(d-\alpha) \leq q \leq \infty\).    
\end{remark}

As a consequence of the construction of cubes in the proof of Proposition \ref{weak type inequality for the fractional maximal operator}, 
we obtain the widely used Calderón-Zygmund decomposition.
\begin{prop}[Calderón-Zygmund decomposition]\label{CZ decomposition}\index{Calderón-Zygmund decomposition}
    Let \(f\in L^1(\R^d)\), \(\lambda > 0\), and let \(\D\) be a dyadic lattice. Then, there is a collection of disjoint dyadic cubes \(\{Q_j\}_j\subseteq \D\) such that:
    \begin{itemize}
        \item \(|f(x)|\leq \lambda\) for a.e. \(x \in \R^d \setminus \cup_j Q_j\);
        \item \(\lambda < \dashint_{Q_j}|f| \leq 2^d\lambda\);
        \item \(|\cup_j Q_j|\leq \lambda^{-1} \|f\|_{L^1}\).
    \end{itemize}
\end{prop}
\begin{proof}
    From the proof of Proposition \ref{weak type inequality for the fractional maximal operator}, we know that there are disjoint dyadic cubes \(\{Q_j\}_j\) such that 
    \[
    E_\lambda = \{x\in \R^d: M^\D f(x) > \lambda\} = \bigcup_j Q_j \text{ and }\dashint_{Q_j}|f| > \lambda,
    \]
    where \(M^\D\) is the dyadic Hardy-Littlewood operator, \(M_0^\D\). From the weak type estimate we conclude immediately the third property. Moreover, by the Lebesgue differentiation theorem, since for \(x \notin \cup_j Q_j\) we have \(M^\D f(x) \leq \lambda\), as we let the cubes approach \(x\), it follows that \(|f(x)|\leq \lambda\) for a.e. \(x\notin \cup_j Q_j\), which is the first property. It remains to show that \(\dashint_{Q_j}|f| \leq 2^d\lambda\). This is done by recalling that the cubes \(\{Q_j\}\) are maximal, i.e.
    \[
    \dashint_{A_1(Q_j)}|f|\leq \lambda.
    \]
    Therefore,
    \[
    \dashint_{Q_j}|f|\leq 2^d \dashint_{A_1(Q_j)}|f|\leq 2^d \lambda. 
    \]
\end{proof}
Note also that a simple modification of the argument provides us with a local version of this result.
\begin{prop}[Local Calderón-Zygmund decomposition]\label{local CZ decomposition}\index{local Calderón-Zygmund decomposition}
    Let \(Q\) be an arbitrary cube, \(\lambda> 0\), and suppose that \(f\in L^1(Q)\) is such that 
    \[
    \dashint_Q|f| \leq \lambda.
    \]
    Then, there is a disjoint collection of cubes \(\{Q_j\}_j \subseteq \emph{Des}(Q)\) such that:
    \begin{itemize}
        \item \(|f(x)|\leq \lambda\) for a.e. \(x \in Q \setminus \cup_j Q_j\);
        \item \(\lambda < \dashint_{Q_j}|f| \leq 2^d\lambda\);
        \item \(|\cup_j Q_j|\leq \lambda^{-1} \|f\|_{L^1(Q)}\).
    \end{itemize}
\end{prop}
\begin{remark}
    Note that in the local version the extra assumption that the average of \(|f|\) over \(Q\) is small needs to be added to guarantee the existence of the maximal cubes.
\end{remark}
Proposition \ref{weak type inequality for the fractional maximal operator} then naturally implies the corresponding strong type estimate by the off-diagonal Marcinkiewicz interpolation theorem (see Corollary 1.4.24 from \cite{Cgrafakos}).
\begin{prop}\label{strong type inequality for the fractional maximal operator}
    Let \(\D\) be a dyadic lattice and let \(w\) be a weight such that \(w(H) = +\infty\) for every quadrant \(H\) of \(\D\). Let \(1< p\leq q\leq \infty\), \(\alpha \in [0, d[\) be such that
    \[
    \frac{1}{q}+\frac{\alpha}{d} = \frac{1}{p}.
    \]
    Then,
    \[
    \|M_{\alpha, w}^\D f\|_{L^{q}(\R^d,w)}\lesssim_{\alpha, d, p, q} \|f\|_{L^p(\R^d,w)},\ \forall f\in L^p(\R^d, w).
    \]
\end{prop}

These results will be useful later on, but of course, we have still not shown the estimates work for the actual fractional maximal operator where we take the supremum over all cubes, not just dyadic cubes. In the next section, we investigate how to compare averages over dyadic cubes to averages over an arbitrary cube.

\subsection{The three lattice theorem}\label{sec: the three lattice theorem}

In order to use the results of the previous section to conclude boundedness properties for the fractional maximal operator, we need to be able to bound averages over an arbitrary cube by averages over dyadic cubes. More precisely we pose the following question: is it the case that, given a cube \(Q\), there exists some dyadic cube \(P\in \D\) so that \(Q\subseteq P\) and
\[
w(Q)^{\frac{\alpha}{d}-1}\int_Q fw \leq C w(P)^{\frac{\alpha}{d}-1}\int_{P}fw,
\]
where the constant \(C\) is independent of the  cubes?\footnote{Again we assume here that \(f\geq 0\).} To simplify the problem we begin by tackling the case \(w = 1\). In this case, it would be enough if we could show the existence of some constant \(C\) such that, given any cube \(Q\) we can find \(P\in \D\) with \(Q\subseteq P\) and \(|P|\leq C|Q|\). If this were possible it would follow that
\[
\frac{1}{|Q|^{1-\frac{\alpha}{d}}}\int_Q f \leq \left(\frac{|P|}{|Q|}\right)^{1-\frac{\alpha}{d}}\frac{1}{|P|^{1-\frac{\alpha}{d}}}\int_Pf \leq C^{1-\frac{\alpha}{d}}\frac{1}{|P|^{1-\frac{\alpha}{d}}}\int_P f,
\]
and this would in turn imply that
\[
M_\alpha f(x) \lesssim_{d,\alpha} M_\alpha^\D f(x).
\]
However, this won't be possible in general. First off, if \(\D\) has more than one quadrant, then it follows from Proposition \ref{one quadrant lattices} that there is some compact set \(K\) that is not contained in any cube from the lattice. But then, if \(Q\) contains \(K\), then it too cannot be contained in any cube from the lattice. Of course this particular obstruction can be easily overcome by insisting that \(\D\) should only have one quadrant. In this case, then given any cube \(Q\) we can always find some \(P\in \D\) that contains \(Q\). Moreover, we will even have a minimal cube \(P\) that satisfies that property in the following sense: given a cube \(Q\) there exists \(P\in \D\) so that \(Q\subseteq P\) and
\[
P_1\in \D\text{ and }Q\subseteq P_1 \implies P\subseteq P_1.
\]
Therefore the minimal cube \(P\) minimizes the ratio \(|P|/|Q|\) and is therefore the natural choice. However, even in this case, it is not true that the constant \(C\) exists. For example, consider the cubes \(Q, P\) in Figure \ref{fig: minimalP}.
\begin{figure}[ht]
    \centering
    \includegraphics[width=0.25\textwidth]{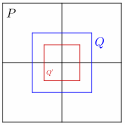}
    \caption{\(P\) is the minimal cube in \(\D\) containing \(Q\) or \(Q'\).}
    \label{fig: minimalP}
\end{figure}
Clearly, \(P\) is the minimal cube in \(\D\) that contains \(Q\). However, the same is true for any other cube \(Q'\) which has the same center as \(Q\) but a smaller sidelength. But then
\[
\frac{|P|}{|Q'|}\xrightarrow[|Q'|\rightarrow 0]{} +\infty
\]
showing that the constant \(C\) does not exist. 

This seems to be problematic for our goal of controlling \(M_\alpha\) pointwise by \(M_\alpha^\D\). Note that strictly speaking we don't actually need this pointwise control if our objective is simply to obtain \(L^p\) bounds for \(M_\alpha\). However, it will be very useful further ahead to have such pointwise control. The idea here is that, although we cannot control \(M_\alpha\) pointwise by the dyadic maximal operator, we can nevertheless control it by using several dyadic operators. The trick is to use three-fold dilations of dyadic cubes in the following way. Fix a dyadic lattice \(\D\) and let \(Q\) be some arbitrary cube. We can find a cube in \(\D\) that approximates \(Q\). Indeed let \(k\in \Z\) be such that 
\[
2^{-k}l(B) \leq l(Q) < 2^{-k + 1}l(B),
\]
and let \(P\) be the cube in \(\D_k\) that contains \(x(Q)\). Now we set 
\[
\tilde{P} = x(P) + 3(P - x(P))
\]
This cube is a three-fold expansion of \(P\) from its corner (see Figure \ref{fig: threefold expansion}).
\begin{figure}[ht]
    \centering
    \includegraphics[width=0.25\textwidth]{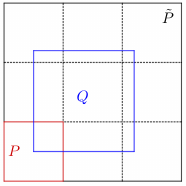}
    \caption{The threefold expansion of \(P\), \(\tilde{P}\), contains \(Q\).}
    \label{fig: threefold expansion}
\end{figure}
Since \(x(Q)\in P\) this means that \(x(P)_j\leq x(Q)_j < x(P)_j + l(P)\), for all \(j\), where \(x(Q)_j\) are the coordinates of \(x(Q)\). If we take a point \(y\in Q\), then
\[
x(P)_j \leq x(Q)_j \leq y_j < x(Q)_j + l(Q) < x(P)_j + l(P) + 2\times 2^{-k}l(B) = x(P)_j + 3 l(P),
\]
and so we see that \(Q\subseteq \tilde{P}\). Moreover, 
\[
\frac{|\tilde{P}|}{|Q|} = \frac{3^d |P|}{|Q|} = 3^d \left(\frac{2^{-k}l(B)}{l(Q)}\right)^d \leq 3^d.
\]
This way we see that \(\tilde{P}\) has the desired properties. The problem of course is that \(\tilde{P}\) is not in \(\D\). However, if the set of cubes
\[
\tilde{\D} = \{\tilde{P}: P\in \D\}
\]
were a dyadic lattice this would solve the problem. The issue is that this will not be a dyadic lattice because these cubes don't satisfy the disjointness property. Although \(\tilde{\D}\) is not a dyadic lattice it is in fact a union of \(3^d\) dyadic lattices.
\begin{thm}[The three lattice theorem; Theorem 3.1 from \cite{intuitive_dyadic_calculus}]\label{three lattice theorem}\index{the three lattice theorem}
Let \(\D\) be a dyadic lattice. Then there are \(3^d\) dyadic lattices \(\D^1, \dots, \D^{3^d}\) such that 
\[
\tilde{\D} = \bigcup_{j=1}^{3^d} \D^j.
\]
Moreover, for any \(Q\in \D\) and any \(j\in \{1,\dots, 3^d\}\) there is a unique cube \(P\in \D^j\) such that \(Q\subseteq P\) and \(l(P) = 3 l(Q)\).
\end{thm}
\begin{proof}
    Let \(B\) be the base of \(\D\). We put \(\tilde{B}\) as the base \(B^1\) of \(\D^1\). Then, the cubes in \(\D\) that give rise to \(\D^1_0\) by tripling their size are 
    \[
    B + 3l(B)m,\ m\in \Z^d.
    \]
    In other words, the cubes in \(\D\) that give rise to \(\D^1_0\) are precisely the cubes in \(\D_0\) whose corners \(x\) differ from \(x(B)\) by an element of \((3l(B)\Z)^d\). Now, similarly, the cubes in \(\D\) that give rise to \(\D^1_1\) are
    \[
    x(B) + \frac{1}{2}(B-x(B)) + \frac{3}{2}l(B)m,\ m\in \Z^d.
    \]
    Again, the cubes here have corners that differ from each other by an element of \((3l(B)\Z/2)^d\) (see Figure \ref{fig: generators of first generations}). 
    \begin{figure}[ht]
        \centering
        \includegraphics[width=0.25\textwidth]{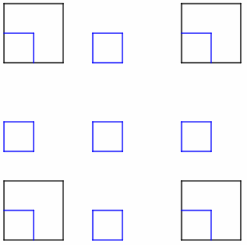}
        \caption{The black cubes give rise to \(\D^1_0\) and the blue cubes to \(\D^1_1\).}
        \label{fig: generators of first generations}
    \end{figure}
    Also, note that if \(x'\) is a corner of a cube \(B + 3l(B)m'\) and \(x''\) is a corner of a cube \(x(B) + (1/2)(B-x(B)) + 3l(B)m''/2\), then
    \begin{equation}\label{difference of corners}
    x' - x'' = 3l(B)m' - \frac{3}{2}l(B)m''.
    \end{equation}
    In particular, \(2(x'-x'') \in (3l(B)\Z)^d\). This suggests an idea: we should organize the different lattices in \(\tilde{\D}\) by a relationship of the corners. In other words, the corners of cubes in \(\D\) that give cubes in the same lattice in \(\tilde{\D}\) should be related in a way motivated by \eqref{difference of corners}. More precisely, we introduce the following equivalence relation in the set of corners \(C = \{x(Q): Q\in \D\}\). Given \(x', x''\in C\),
    \[
    x' \sim x'' \iff \exists l \in \Z: 2^l(x'- x'')\in (3l(B)\Z)^d.
    \]
    The first thing we want to establish about this equivalence relation is that there are exactly \(3^d\) equivalence classes. It is easy to see there are at least \(3^d\) equivalence classes, since all the corners \(x(B) + l(B)u, u\in\{0, 1, 2\}^d\), lie in different classes. Now take any \(3^d + 1\) corners in \(C\), \(x^j, j=1,\dots, 3^d+1\). Any corner in \(C\) has the form \(x(B) + 2^{-k}l(B)m, m\in \Z^d, k\in \Z\), therefore for each \(j\) there are \(k_j \in \Z, m^j\in \Z^d\) so that \(x^j = x(B) + 2^{-k_j}l(B)m^j\). Now define
    \[
    y^j = 2^{\max_l k_l}l(B)^{-1}(x^j - x^1) = 2^{(\max_l k_l) - k_j}m^j - 2^{(\max_l k_l) - k_1}m^1 \in \Z^d.
    \]
    This gives \(3^d\) vectors \(y^2, \dots, y^{3^d+1}\) all in \(\Z^d\). Since there are \(3^d\) classes in \(\Z^d\) modulo \((3\Z)^d\), it follows that there is some \(j\) such that \(x^j \sim x^1\). This shows that there are exactly \(3^d\) classes in \(C/\sim\). Let's call these classes \(C_1, \dots, C_{3^d}\). Then we can define the families 
    \[
    \D^j = \{\tilde{Q}: x(Q) \in C_j\}.
    \]
    Since \(\{C_j\}_j\) is a partition of the corners in \(\tilde{\D}\), we see that \(\{\D^j\}_j\) forms a partition of \(\tilde{\D}\). Now we just have to prove that each \(\D^j\) is indeed a dyadic lattice. To do this we have to check the two properties in Proposition \ref{characterization of dyadic lattices}. Let's start with the first property. We can split
    \[
    \D^j = \bigcup_{k\in \Z} \D^j_k,
    \]
    where 
    \[
    \D_k^j = \{\tilde{Q}: Q\in \D_k\text{ and }x(Q) \in C_j\}.
    \]
    Fix \(\tilde{Q}\in \D_k^j\), and let \(\tilde{Q}_1\) be any other cube in \(\D_k^j\). Then, we know that \(Q, Q_1\in \D_k\) and \(x(Q)\sim x(Q_1)\). This means that 
    \[
    x(Q)-x(Q_1) \in (2^{-k}l(B)\Z)^d\text{ and }2^l(x(Q) - x(Q_1))\in (3l(B)\Z)^d,
    \]
    for some \(l\in \Z\). Therefore, there are \(n', n''\in \Z^d\) such that 
    \[
    2^k(x(Q)- x(Q_1)) = l(B)n'\text{ and }2^l(x(Q)-x(Q_1)) = 3l(B)n'',
    \]
    which implies that \(n' = 3\times 2^{k-l}n''\). But then we must have \(n'\in (3\Z)^d\). So, we can write \(2^k(x(Q)-x(Q_1)) = 3l(B)m,\ m\in \Z^d\). Now, note that \(|\tilde{Q}| = |\tilde{Q}_1|\) and 
    \[
    x(\tilde{Q}) - x(\tilde{Q}_1) = \frac{3l(B)m}{2^k} = \frac{3m}{2^k}(3^{-1}\times 2^kl(\tilde{Q})) = l(\tilde{Q})m.
    \]
    This shows that \(\tilde{Q}_1 \in \{\tilde{Q} + l(\tilde{Q})m: m\in \Z^d\}\). On the other hand, any cube from \(\{\tilde{Q} + l(\tilde{Q})m: m\in \Z^d\}\) must be in \(\D^j_k\). Indeed, suppose \(P = \tilde{Q} + l(\tilde{Q})m\), and let \(Q_1\) be such that \(\tilde{Q}_1 = P\). Then, this means that
    \[
    \tilde{Q} + l(\tilde{Q})m = x(Q_1) + 3(Q_1 - x(Q_1)).
    \]
    This is the same as saying that
    \[
    x(Q_1) = x(Q) + 3l(Q)m\text{ and }l(Q_1) = l(Q).
    \]
    In particular, since \(Q\in \D_k\) it follows that \(Q_1\in \D_k\). Moreover, since \(x(Q_1) - x(Q) \in (3l(B)\Z)^d\) this shows that \(x(Q)\sim x(Q_1)\). Thus, \(P\in \D_k^j\). This way, we have shown that
    \[
    \D_k^j = \{\tilde{Q} + l(\tilde{Q})m: m\in \Z^d\},
    \]
    which is property 1 from Proposition \ref{characterization of dyadic lattices}. Now let's check property 2. Let \(\tilde{Q} \in \D_k^j\) and let \(P = x(\tilde{Q}) + (1/2)(\tilde{Q}-x(\tilde{Q}))\) be the corner child of \(\tilde{Q}\). Let \(Q_1\) be such that \(P = \tilde{Q}_1\). Then, 
    \[
    x(Q) + \frac{1}{2}(\tilde{Q} - x(Q)) = x(Q_1) + 3(Q_1 - x(Q_1)),
    \]
    which simplifies to
    \[
    x(Q) = x(Q_1) \text{ and }l(Q) = 2l(Q_1).
    \]
    In other words, \(Q_1\) is the corner child of \(Q\). In particular \(Q_1 \in \D_{k+1}\), and since \(x(Q_1) = x(Q)\), then \(x(Q_1) \in C_j\). Thus, \(P \in \D_{k+1}^j\), which shows property 2.

    Finally, we check that the last statement of the theorem is true. Let \(Q\in \D\) and \(j\in \{1, \dots, 3^d\}\). 
    Now let \(P\) be the unique cube in \(\D^j_k\) such that \(x(Q)\in P\), where \(k\) is such that \(Q\in \D_k\). 
    Then, \(P = \tilde{Q}_1\) for some \(Q_1\in \D_k\). 
    So, we have that \(l(P) = 3l(Q_1) = 3l(Q)\). Now note that no other cube from \(\D^j\) can satisfy the desired properties. 
    This is because the condition \(l(P)=3l(Q)\) fixes the generation \(\D^j_k\), and no other cube in this generation contains \(x(Q)\), 
    so no other cube from this generation can contain \(Q\). This way, it remains only to show that \(Q\subseteq P\). 
    First note that, since \(Q, Q_1\) are in the same generation, then there is some \(m\in \Z^d\) such that \(x(Q) = x(Q_1) + l(Q)m\). 
    Now, given that \(x(Q)\in P\), this means that
    \[
    x(Q_1)_j\leq x(Q)_j < x(Q_1)_j + 3l(Q), \forall j.
    \]
    This holds if and only if \(0\leq m_j < 3\), for all \(j\), which implies that \(0\leq m_j\leq 2\), for all \(j\). Thus, 
    \[
    x(Q_1)_j\leq x(Q)_j\leq x(Q_1)_j+2l(Q).
    \]
    Now let \(y\in Q\). Then, \(x(Q)_j\leq y_j < x(Q)_j + l(Q)\), for all \(j\). But then,
    \[
    x(Q_1)_j\leq x(Q)_j\leq y_j < x(Q)_j + l(Q) \leq x(Q_1)_j + 3l(Q),
    \]
    which shows that \(y\in P\).
\end{proof}
\begin{remark}\label{one quadrant implies one third shifts have one quadrant}
    Note that if \(\D\) has only one quadrant, then \(\D^1,\dots, \D^{3^d}\) have also only one quadrant. Indeed, if \(x\in \R^d\) and \(\{Q_k(x)\}\) is the dyadic tower over \(x\) from \(\D\), then for each \(k, j\) there is a unique cube \(P^j_k\in \D^j\) such that \(Q_k(x)\subseteq P_k^j\) and \(l(P_k^j) = 3l(Q_k(x))\). But then
    \[
    \R^d = \bigcup_k Q_k(x) = \bigcup_k P_k^j,
    \]
    and so \(\D^j\) has only one quadrant.
\end{remark}
There is a simple way to relate the dyadic lattices \(\D^1, \dots, \D^{3^d}\). Indeed, let \(\varepsilon^1,\dots, \varepsilon^{3^d}\) be an enumeration of \(\{-1, 0 , 1\}^d\) where \(\varepsilon^1 = (0, \dots, 0)\). Now put
\[
C_j = \{x\in C: x \sim x(B) + \varepsilon^j l(B)\}.
\]
Then it follows that 
\begin{equation}\label{one third trick}
\D^j_k = \{ P + (-1)^k\frac{\varepsilon^j}{3}l(P): P\in \D^1_k\}.
\end{equation}
For example, to see that the left hand side contains the right hand side, let \(Q\) be a cube such that \(\tilde{Q} = P + (-1)^k \varepsilon^j l(P)/3\), for some \(P\in \D^1_k\). Since \(P\in \D_k^1\) we know that \(P = \tilde{Q}_1\) for some \(Q_1\in \D_k\) with \(x(Q_1)\sim x(B)\). Now, note that
\[
x(Q) = x(P) + (-1)^k \varepsilon^j \frac{l(P)}{3} = x(Q_1) + (-1)^k\varepsilon^j l(Q_1),
\]
and since \(Q_1\in \D_k\), this means that \(Q\) is also in \(\D_k\). Therefore to show that \(\tilde{Q}\in \D_k^j\) it remains to show only that \(x(Q) \sim x(B) + \varepsilon^j l(B)\). We know that \(x(Q_1) \sim x(B)\) and so there is some \(l\in \Z\) such that
\[
2^l(x(P) - x(B)) = 3l(B)m,\ m\in \Z^d.
\]
But then, 
\[
x(Q) = x(B) + 3\times 2^{-l}l(B)m + (-1)^k\varepsilon^j l(Q).
\]
So, for \(a\in \N_1\),
\[
2^a(x(Q) - x(B) - \varepsilon^j l(B)) = 2^a l(B)(3\times 2^{-l}m + (-1)^k \varepsilon^j 2^{-k} - \varepsilon^j).
\]
It is not difficult to check that, if \(a\) is large enough, then \(3\times 2^{a-l}m +( (-1)^k2^{a-k}-2^a)\varepsilon^j\) is in \((3\Z)^d\), as required. One can then verify the other inclusion in a very similar way, thus showing that \eqref{one third trick} holds. This equality says that the lattices \(\D^j\) are obtained from \(\D^1\) by shifts of one-third of their lengths in a direction carefully chosen based on \(j\) and the generation. This is why this construction is sometimes called the one-third trick\index{one-third trick}.

Now using the three lattice theorem we can show that any cube is well approximated by some dyadic cube coming from one of a finite number of dyadic families.
\begin{cor}\label{corollary of the 3 lattice theorem}
There are \(3^d\) dyadic lattices, \(\D^1, \dots, \D^{3^d}\), such that, given any cube \(Q\) there is some \(j(Q)\in \{1, \dots, 3^d\}\) and a cube \(P\in \D^{j(Q)}\) such that \(Q\subseteq P\) and \(|P|\leq 3^d |Q|\).
\end{cor}
\begin{proof}
    Fix any dyadic lattice \(\D\), and let \(\D^1,\dots, \D^{3^d}\) be the dyadic lattices from Theorem \ref{three lattice theorem}. Now let \(Q\) be any cube, and as before,\footnote{Recall that this is the construction explained in Figure \ref{fig: threefold expansion}.} let \(Q_1\) be the cube from \(\D_k\) that contains \(x(Q)\) where
    \[
    2^{-k}l(B)\leq l(Q) < 2^{-k+1}l(B).
    \]
    We have already checked above that \(Q\subseteq \tilde{Q}_1\) and \(|\tilde{Q}_1|/|Q| \leq 3^d\). Now note that \(P = \tilde{Q}_1 \in \tilde{\D}\) and therefore there is some \(j(Q) \in \{1, \dots, 3^d\}\) such that \(P\in \D^{j(Q)}\), as required.
\end{proof}
\begin{remark}
In fact it is possible to get by using just \(d + 1\) dyadic lattices (see \cite{note_dyadic_coverings}). However, for our purposes it won't be necessary to obtain any such sharpening of this result.
\end{remark}

Now that we can approximate an arbitrary cube we can finally control the fractional maximal operator by the dyadic maximal operator. Consider \(x\in \R^d\) and let \(Q\) be any cube containing \(x\). Then we know that we have some \(j(Q)\) so that \(Q\subseteq P\) for some \(P\in \D^{j(Q)}\) with \(|P|\leq 3^d |Q|\). So,
\[
|Q|^{\frac{\alpha}{d}-1}\int_Q |f| \leq 3^{d-\alpha} |P|^{\frac{\alpha}{d}-1}\int_P|f| \leq 3^{d-\alpha} M_\alpha^{\D^{j(Q)}}f(x) \leq 3^{d-\alpha}\sum_{j=1}^{3^d}M_\alpha ^{\D^j}f(x).
\]
Taking the supremum in \(Q\) we get
\begin{equation}\label{eq: control by dyadic max opers}
    M_\alpha f(x) \leq 3^{d-\alpha}\sum_{j=1}^{3^d}M_\alpha^{\D^j}f(x).
\end{equation}
As a consequence of this pointwise estimate and Propositions \ref{weak type inequality for the fractional maximal operator} and \ref{strong type inequality for the fractional maximal operator} we immediately get the following result.
\begin{prop}\label{strong type estimate for fractional max operator with lebesgue}
    Let \(1< p\leq  q\leq \infty\) and \(0\leq \alpha < d\) be such that
    \[
    \frac{1}{q} + \frac{\alpha}{d} = \frac{1}{p}.
    \]
    Then, 
    \[
    \|M_\alpha f\|_{L^q(\R^d)}\lesssim_{\alpha,d,p,q} \|f\|_{L^p(\R^d)},\ \forall f\in L^p(\R^d).
    \]
    Moreover, we have the weak type endpoint estimate,
    \[
    \|M_\alpha f\|_{L^{\frac{d}{d-\alpha}, \infty}(\R^d)}\lesssim_{d, \alpha} \|f\|_{L^1(\R^d)},\ \forall f\in L^1(\R^d).\footnote{
        Here we can include \(\alpha = d\), seeing as this would simply be the \(L^1\) to \(L^\infty\) estimate 
        for the fractional maximal operator.}
    \]    
\end{prop}
We might also be interested in obtaining this for different measures, 
but we need to be careful because Corollary \ref{corollary of the 3 lattice theorem} works specifically for the Lebesgue measure. 
However, if our weighted measure can be compared with the Lebesgue measure the argument still carries through. 
To achieve this more general result we use the notion of \(A_\infty\) weights. See Appendix \ref{appendix: weighted spaces} for the definition 
of the \(A_\infty\) class of weights and for some of its most important properties. For our immediate purposes, it suffices to 
recall that if \(w\in A_\infty\), then there are \(C_1, C_2, \delta_1, \delta_2 > 0\) such that, for all cubes \(Q\) 
and for all positive measure subsets \(A\subseteq Q\) we have
\[
C_1 \left(\frac{|Q|}{|A|}\right)^{\delta_1}\leq \frac{w(Q)}{w(A)}\leq C_2\left(\frac{|Q|}{|A|}\right)^{\delta_2}. 
\]
\begin{remark}\label{remark: A_infty weights give infinite measure to quadrants}
    Note that if \(\D\) is a dyadic lattice and \(w\in A_\infty\), then \(w(H) = +\infty\) for any quandrant \(H\) of \(\D\). Indeed, let \(x\in \R^d\) and consider the dyadic tower over \(x\), \(\{Q_k(x)\}_k\). We have that
    \[
    w(Q_k(x))\geq C_1 \left(\frac{|Q_k(x)|}{|Q_0(x)|}\right)^{\delta_1} w(Q_0(x)) = C_1 2^{-k\delta_1} w(Q_0(x)) \xrightarrow[k \rightarrow -\infty]{} +\infty,
    \]
    and therefore \(w(H_x) = +\infty\).
\end{remark}
Using \(A_\infty\) weights we can give a more general version of 
Proposition \ref{strong type estimate for fractional max operator with lebesgue} where the 
maximal operator is defined with respect to a weighted measure\index{weighted fractional maximal operator}
\[
M_{\alpha, w}f(x) = \sup_{Q \ni x} w(Q)^{\frac{\alpha}{d}-1}\int_Q |f|w.
\]
\begin{prop}\label{strong type estimate for fractional max operator with a weight}
    Let \(w\in A_\infty\), \(1< p\leq q\leq \infty\) and \(\alpha \in [0, d[\) be such that
    \[
    \frac{1}{q}+\frac{\alpha}{d} = \frac{1}{p}.
    \]
    Then, 
    \[
    \|M_{\alpha, w}f\|_{L^q(\R^d, w)} \lesssim \|f\|_{L^p(\R^d, w)},\ \forall f\in L^p(\R^d, w).
    \]
    Also,
    \[
    \|M_{\alpha, w}f\|_{L^{\frac{d}{d-\alpha}, \infty}(\R^d, w)}\lesssim_{d, \alpha, w}\|f\|_{L^1(\R^d, w)},\ \forall f\in L^1(\R^d, w).
    \footnote{As before, here we can also include \(\alpha = d\).}
    \]
\end{prop}
\begin{proof}
    Fix \(x\in \R^d\) and let \(Q\) be a cube containing \(x\). From Corollary \ref{corollary of the 3 lattice theorem}, we know that we can find \(j(Q)\) and \(P\in \D^{j(Q)}\) so that \(Q\subseteq P\) and \(|P|\leq 3^d|Q|\). Therefore, using the assumption that \(w\in A_\infty\) we get
    \[
    \frac{1}{w(Q)^{1- \frac{\alpha}{d}}}\int_Q|f|w \leq C_2^{1-\frac{\alpha}{d}} 3^{\delta_2(d-\alpha)} \sum_j M_{\alpha, w}^{\D^j}f(x).
    \]
    Therefore, from Proposition \ref{strong type inequality for the fractional maximal operator},
    \[
    \|M_{\alpha, w}f\|_{L^q(w)}\lesssim_{\alpha, d, p, q, w}\|f\|_{L^p(w)},
    \]
    where we used Remark \ref{remark: A_infty weights give infinite measure to quadrants} to justify that \(w\) satisfies the assumptions of Proposition \ref{strong type inequality for the fractional maximal operator}. The endpoint inequality also follows from the comparison with \(M^{\D^j}_{\alpha, w}\).
\end{proof}

\clearpage
\section{Sparse Domination}\label{sec: sparse domination}

In the previous section, we proved boundedness results for the fractional maximal operator, which served as a good way of introducing the basic ideas of dyadic Harmonic Analysis. By restricting cubes to be dyadic cubes we obtained a certain extra structure that proved very useful in carrying out the proofs. This structure involved the compatibility between the partial order given by inclusion and the disjointness of the cubes. In turn this gave us a notion of maximal cubes which proved very useful. Moreover, while this dyadic structure is a simplification of the overall structure of the continuum, it was still enough to recover the non-dyadic results. In that step, it was essential that the dyadic structure was able to approximate every location in space, as well as every scale. 

In this section, we will introduce the ideas of sparse domination, which is the main objective of this expository paper. 
To do this we begin by motivating the techniques using the problem of obtaining two-weight estimates for the fractional maximal operator. 
In other words, we wish to find conditions on two weights \(v, w\) such that
\begin{equation}\label{two-weight estimate fractional max oper}
    \|M_\alpha f\|_{L^q(w)}\lesssim \|f\|_{L^p(v)},\ \forall f\in L^p(v).
\end{equation}
Note that this is a very different setup from what we saw in the previous section. 
Before, we considered weighted \(L^p\) spaces but where the maximal operator was defined in terms of the same measure. 
As we saw, this case is very similar to the case when we simply deal with the Lebesgue measure both in the definition of the maximal operator and in the norms. 
However, a much more interesting case arises when we choose different measures for the maximal operator and the norms. 
Here we consider the maximal operator with respect to the Lebesgue measure, but we change the measure of the underlying space. 
One possibility is to choose the same measure for both the source and target spaces, 
which has a simpler theory because one can interpolate between weak type estimates using the reverse Hölder inequality. 
See for instance the proof of Theorem 3 in \cite{MuckenhouptWheeden}. 
However, here we are interested in the case when the measures of the source and target space can differ. 
Although there are complete characterizations for the two-weight inequality for the fractional maximal operator 
(see \cite{sawyermaximal} and Theorem 1 in \cite{Wheeden}), for more general operators the two-weight case is usually much more subtle. 
Moreover, it is uncommon to find complete characterizations of the weights that satisfy two-weight norm inequalities. 
Usually, the results focus on obtaining sufficient conditions for such inequalities to hold, but there are many approaches to proving sufficient conditions. 
One approach is to use testing conditions, as in \cite{Sawyer2}, where such conditions are used to find estimates for fractional integrals. 
Here we will follow a different approach and use so-called `bump' conditions, which are well suited for dyadic methods.

Let us begin by assuming that \eqref{two-weight estimate fractional max oper} holds with implicit constant \(C\), so that we may find some necessary conditions on the weights. Let \(Q\) be any cube, let \(\varepsilon > 0\) and put \(f = (v+\varepsilon)^t\chi_Q\), where \(t<0\) is to be chosen later. Note that, since \(v\) is a weight, then \(f\) is locally integrable and \(f\in L^p(v)\), \(1\leq p \leq \infty\). Also, \(\int_Q f \in ]0, \infty[\). Now let \(\lambda < |Q|^{\alpha / d -1}\int_Q (v+\varepsilon)^t\). Then, \(Q\subseteq \{x: M_\alpha f(x) > \lambda\}\). If \(q<\infty\) then this implies that
\[
\begin{split}
\lambda^q w(Q) &\leq \lambda^q w(\{x: M_\alpha f(x)>\lambda \}) \\
&\leq  \int M_\alpha f(x)^q w \\
&\leq C^q\left( \int_Q (v+\varepsilon)^{tp}v\right)^\frac{q}{p}\\
&\leq C^q \left( \int_Q (v+\varepsilon)^{tp+1}\right)^\frac{q}{p}.
\end{split}
\]
Since this holds for every \(\lambda\) under \(|Q|^{\alpha/d-1}\int f\), then if we let \(\lambda\) approach this value we see that
\[
|Q|^{\frac{\alpha}{d}-1}\int_Q (v+\varepsilon)^t w(Q)^{1/q}\leq C \left(\int_Q (v+\varepsilon)^{tp+1}\right)^{1/p}.
\]
At this point it makes sense to choose \(t\) so that \(tp+1 = t\), which gives us \(t = -p'/p\). For this to work, we need to assume here that \(p > 1\). With this choice for \(t\) we obtain
\[
|Q|^{\frac{\alpha}{d}-1}\left(\int_Q (v+\varepsilon)^{-\frac{p'}{p}}\right)^\frac{1}{p'}\left(\int_Q w\right)^\frac{1}{q} \leq C.
\]
If we now let \(\varepsilon \rightarrow 0^+\), then by the monotone convergence theorem we get
\[
|Q|^{\frac{\alpha}{d}-1}\left(\int_Q v^{-\frac{p'}{p}}\right)^\frac{1}{p'}\left(\int_Q w\right)^\frac{1}{q} \leq C,
\]
for every cube. This leads us to the definition of an important class of weights.
\begin{defn}\label{defn of Apq alpha weights}
    Let \(1<p,q<\infty\) and \(\alpha \in \R\). We say that \((v, w)\in A_{p, q}^\alpha\)\index{\(A_{p,q}^\alpha\) condition} if and only if 
    \[
    [v, w]_{A_{p, q}^\alpha} := \sup_Q |Q|^{\frac{\alpha}{d}-1}\left(\int_Q v^{-\frac{p'}{p}}\right)^\frac{1}{p'}\left(\int_Q w\right)^\frac{1}{q} < \infty.
    \]
\end{defn}
\begin{remark}
    This class of weights was considered explicitly in \cite{sawyer}, and it is an extension of the Muckenhoupt classes \(A_p\) and \(A_{p, q}\). 
    For the definition of the \(A_p\) weights and for some of their most important properties see Appendix \ref{appendix: weighted spaces}. It is interesting 
    to see how the classes \(A_p, A_{p, q}\) and \(A_{p, q}^\alpha\) are related. 
    We say that a pair of weights \((v, w)\) satisfies the \(A_{p, q}\) condition\index{\(A_{p, q}\) weight} if 
    \[ 
        [v, w]_{A_{p, q}} = \sup_Q \left(\dashint_Q v^{-p'/p}\right)^{1/p'} \left(\dashint_Q w\right)^{1/q} < \infty. 
    \] 
    In particular, 
    \[ 
        [w]_{A_p} = [w, w]_{A_{p, p}}^p, 
    \] 
    and thus 
    \[ 
        w \in A_p \iff (w, w) \in A_{p, p}. 
    \] 
    Also, 
    \[ 
        (v, w) \in A_{p, q} \iff (v, w)\in A_{p, q}^\alpha \text{ where }\alpha = d\left(\frac{1}{p} - \frac{1}{q}\right). 
    \]
\end{remark}
With these weights, we can easily prove the weak type \((p, q)\) inequality by following an argument which is very similar to the proof of Proposition \ref{weak type inequality for the fractional maximal operator}.
\begin{prop}\label{two weight weak type estimate for frac max oper}
    Let \(1 < p\leq q < \infty\) and \(\alpha \in [0, d[\). Suppose that \(v, w\) are weights such that \((v, w)\in A_{p, q}^\alpha\). Then,
    \[
    \|M_\alpha f\|_{L^{q, \infty}(w)} \lesssim_{\alpha, d, q}[v,w]_{A_{p,q}^\alpha} \|f\|_{L^p(v)},\ \forall f\in L^p(v).
    \]
\end{prop}
\begin{proof}
    First we argue that, under the assumptions, if \(f\in L^p(v)\) then \(f\) is locally integrable. Indeed, if \(Q\) is a cube, then
    \[
    \begin{split}
    \int_Q |f| &= \int_Q |f| v^\frac{1}{p} v^{-\frac{1}{p}} \\
    &\leq \left(\int_Q |f|^p v\right)^{1/p}\left(\int_Q v^{-\frac{p'}{p}}\right)^{1/p'} \\
    &\leq \|f\|_{L^p(v)}[v, w]_{A_{p, q}^\alpha} w(Q)^{-1/q}|Q|^{1-\alpha/d} < \infty. 
    \end{split}
    \]
    This way, for any \(f\in L^p(v)\) the operator \(M_\alpha\) is well-defined on \(f\). Now note that in virtue of \eqref{eq: control by dyadic max opers} we know that
    \[
    \|M_\alpha f\|_{L^{q,\infty}(w)} \leq 3^{d-\alpha} \|\sum_{j=1}^{3^d}M_\alpha ^{\D^j}f\|_{L^{q,\infty}(w)} \lesssim_{\alpha, d, q}\sum_{j=1}^{3^d}\|M_\alpha ^{\D^j}f\|_{L^{q,\infty}(w)}.
    \]
    So, it is enough to prove that
    \[
    \|M_\alpha^\D f\|_{L^{q,\infty}(w)}\leq [v,w]_{A_{p,q}^\alpha} \|f\|_{L^p(v)},
    \]
    for any given dyadic lattice \(\D\). This way, let \(\D\) be a dyadic lattice and put \(E_\lambda = \{x\in \R^d: M_\alpha^\D f(x) > \lambda\}\). As before, we want to argue that there are disjoint dyadic cubes \(\{Q_j\}_j\) from \(\D\) such that
    \[
    E_\lambda = \bigcup_j Q_j \text{ and }\lambda < |Q_j|^{\frac{\alpha}{d}-1}\int_{Q_j}|f|.
    \]
    However, there is a small technical obstruction here. The problem is that to ensure the existence of the maximal 
    dyadic cubes we need to know that \(|Q|^{\alpha/d - 1}\int_Q |f| \rightarrow 0\) as \(|Q|\rightarrow +\infty\). 
    But this is not necessarily the case since we only know that \(f\in L^p(v)\).\footnote{If \(w\in A_\infty\), then it would follow that 
    \(w(Q)^{-1/q}\rightarrow 0\) as \(|Q|\rightarrow +\infty\), which is enough to ensure that \(|Q|^{\alpha/d - 1}\int_Q |f| \rightarrow 0\). 
    However, we only assume that \((v, w)\in A_{p, q}^\alpha\) which does not imply that \(w\in A_\infty\).}
    To overcome this issue we argue first with \(f_N = f \chi_{B(0, N)}\). 
    For these functions we can indeed say that \(|Q|^{\alpha/d-1}\int_Q |f_N| \rightarrow 0\) 
    as \(|Q|\rightarrow \infty\) because \(\alpha < d\) and \(f\in L^1_\text{loc}\). 
    Set \(E_{\lambda, N} = \{x: M_\alpha^\D f_N (x) > \lambda\}\). Now we can indeed ensure the existence of 
    disjoint dyadic cubes \(\{Q_j\}_j\) from \(\D\) such that
    \[
    E_{\lambda, N} = \bigcup_j Q_j \text{ and }\lambda < |Q_j|^{\frac{\alpha}{d}-1}\int_{Q_j}|f_N|.
    \]
    This way we have that
    \[
    \begin{split}
    \lambda^q w(E_{\lambda,N}) &= \sum_j \lambda^q w(Q_j) \leq \sum_j |Q_j|^{q(\alpha/d-1)}\left(\int_{Q_j}|f_N|\right)^q w(Q_j) \\
    &\leq \sum_j w(Q_j)|Q_j|^{q(\alpha/d-1)}\left(\int_{Q_j}|f|v^\frac{1}{p}v^{-\frac{1}{p}}\right)^q \\
    &\leq \sum_j w(Q_j)|Q_j|^{q(\alpha/d-1)}\left(\int_{Q_j}|f|^pv\right)^\frac{q}{p}\left(\int_{Q_j}v^{-\frac{p'}{p}}\right)^\frac{q}{p'}\\
    &\leq [v,w]_{A_{p, q}^\alpha}^q \left(\sum_j \int_{Q_j}|f|^pv\right)^\frac{q}{p} \leq [v, w]^q_{A_{p,q}^\alpha} \|f\|_{L^p(v)}^q,
    \end{split}
    \]
    where we used the fact that \(p\leq q\). To finish the proof we just have to justify that \(w(E_{\lambda,N}) \rightarrow_N w(E_\lambda)\). First note that \((E_{\lambda,N})_N\) is an increasing family of sets, and so we know that \(\lim_N w(E_{\lambda,N}) = w(\cup_N E_{\lambda,N})\), so it is enough to show that 
    \[
    E_\lambda = \bigcup_N E_{\lambda, N}.
    \]
    If \(x\in E_{\lambda, N}\) then \(M_\alpha ^\D f(x) \geq M_\alpha^\D f_N (x) > \lambda\) and so \(x \in E_\lambda\). For the converse inclusion suppose that \(x\in E_\lambda\). Then, \(M_\alpha^\D f(x) > \lambda\) which means that there is some \(Q\in\D\) such that 
    \[
    |Q|^{\alpha/d-1}\int_Q|f| > \lambda.
    \]
    But, by the monotone convergence theorem, we know that
    \[
    |Q|^{\alpha/d-1}\int_Q |f_N| \xrightarrow[N\rightarrow +\infty]{} |Q|^{\alpha/d-1}\int_Q|f|.
    \]
    Therefore there is some \(N\) large enough so that 
    \[
    |Q|^{\alpha/d-1}\int_Q|f_N| > \lambda,
    \]
    which implies that \(x\in E_{\lambda,N}\).
\end{proof}
\begin{remark}
    This proposition should be a generalization of the Lebesgue measure case of Proposition \ref{weak type inequality for the fractional maximal operator}, and in fact, if we choose \(v=w=1\) then the \(A_{p,q}^\alpha\) condition gives
    \[
    \begin{split}
    \sup_Q |Q|^{\alpha/d-1}\left(\int_Q 1\right)^{1/p'}\left(\int_Q 1\right)^{1/q} <\infty &\iff \sup_Q |Q|^{\alpha/d-1+1/p'+1/q}<\infty \\
    &\iff \frac{1}{q}+\frac{\alpha}{d} = \frac{1}{p}.
    \end{split}
    \]
    This way we see that the \(A_{p,q}^\alpha\) condition incorporates the `Sobolev' relation \(q^{-1}+\alpha/d = p^{-1}\) implicitly. In the weighted case this relation does not need to be satisfied as long as the weights compensate for that imbalance in the way that the \(A_{p,q}^\alpha\) condition demands.
\end{remark}
\begin{remark}\label{remark: Apq alpha implies sub-sobolev relation}
    One might wonder where the assumption \(\alpha \geq 0\) was used in the proof of this result. This assumption is due to the condition \(p\leq q\). In fact, for two weights to satisfy the \(A_{p,q}^\alpha\) condition we must have
    \[
    \sup_Q |Q|^{\alpha/d - 1/p + 1/q}\left(\dashint_Q v^{-\frac{p'}{p}}\right)^{1/p'}\left(\dashint_Q w\right)^{1/q}<\infty.
    \]
    So, if we let the cubes shrink to a point, then by the Lebesgue differentiation theorem we must have
    \[
    \frac{\alpha}{d} + \frac{1}{q} \geq \frac{1}{p}
    \]
    for the supremum to be finite. This condition in conjunction with \(p\leq q\) implies \(\alpha \geq 0\).
\end{remark}

In Section \ref{sec: the basic definitions} we used the off-diagonal Marcinkiewicz 
interpolation theorem and the weak type estimates to deduce the strong 
type estimates. However, in the weighted case this approach is problematic. 
The issue is that, in order to prove that
\[
\|M_\alpha f\|_{L^q(w)}\lesssim \|f\|_{L^p(v)}
\]
we should start with weights in \(A_{p, q}^\alpha\). But to apply the interpolation 
argument we need to have \(p_0, p_1, q_0, q_1\) such that
\[
\|M_\alpha f\|_{L^{q_j, \infty}(w)}\lesssim \|f\|_{L^{p_j}(v)}, j=0, 1
\]
for \(p_0, p_1 \approx p\) and \(q_0, q_1 \approx q\). However, these estimates 
mean that our weights should satisfy \(A_{p_j, q_j}^\alpha\) conditions, which is 
not the assumption on the weights that we start with. In the one-weight case and 
if \(q^{-1}+\alpha/d = p^{-1}\) this is not really a problem. In fact, if 
\(v = w^{p/q}\), then the \(A_{p,q}^\alpha\) condition means that \(w\in A_r\) 
where \(r = 1+q/p'\). But we know that the Muckenhoupt weights satisfy the following 
property:
\[
A_p = \bigcup_{1<q<p} A_q.
\]
Therefore, if \(w\in A_r\), then we can say that \(w\in A_{r-\varepsilon}\cap 
A_{r+\varepsilon}\) for \(\varepsilon>0\) small enough. One then has two weak type 
estimates which imply the strong type estimate by interpolation. For the same 
argument to work in the two-weight case we would need to have some property of the 
form
\[
(v, w) \in A_{p,q}^\alpha \implies (v, w)\in A_{p_j, q_j}^\alpha,
\]
in such a way so that the points \((1/p, 1/q), (1/p_0, 1/q_0), (1/p_1, 1/q_1)\) are 
collinear. If \((v, w)\in A_{p, q}, v^{-p'/p}, w\in A_\infty\), then one can show that 
indeed such collinear points exist. However, when considering the \(A_{p, q}^\alpha\) 
condition, not only would \(p\) and \(q\) change, but also \(\alpha\). This means that 
the operator would change between the two weak type inequalities. Perhaps it would be 
possible to argue using a mixture of the off-diagonal Marcinkiewicz interpolation theorem 
and Stein's interpolation of an analytic family of operators. Although this question is 
interesting we pursue here a different line of reasoning, 
because we can solve the problem using the method of domination by sparse operators.

\subsection{Controlling the fractional maximal operator}\label{section: controlling the fractional maximal operator}

In this section, our goal is to obtain two weight estimates for the fractional maximal operator. That is, we want to find conditions on \(v, w\) such that
\[
\|M_\alpha f\|_{L^q(w)}\lesssim \|f\|_{L^p(v)}. 
\]
Again we can assume without loss of generality that \(f\geq 0\). The basic idea starts with the observation that, if \(M_\alpha f(x)<\infty\), then there is some cube \(Q_x\) such that
\[
M_\alpha f(x) \approx |Q_x|^\frac{\alpha}{d}\dashint_{Q_x} f.
\]
Moreover, it is natural to expect that the same approximation works in some region around \(x\). This might inspire the idea of trying to find some family of cubes \(\S\) so that we have the estimate
\[
M_\alpha f(x) \lesssim \sum_{Q\in \S} \1_Q(x)|Q|^\frac{\alpha}{d}\dashint_Q f.
\]
The fact that we are replacing a supremum with a superposition is helpful for the purpose of obtaining \(L^p\) bounds, especially if we argue by duality. Of course, this estimate can always be achieved if we choose the family \(\S\) to be very large. As a trivial example, the estimate holds if \(\S\) is the set of all cubes. However, since our goal is to obtain bounds on the \(L^q(w)\) norm of this expression, it is useful to have the smallest family we can get away with on the right hand side. As a simple example, suppose that the family were disjoint. That is, let \(\S\) be a disjoint family of cubes and put
\[
\A_\S^\alpha f(x) = \sum_{Q\in \S} \1_Q(x) |Q|^\frac{\alpha}{d}\dashint_Q|f|.
\]
If we assume that \(q^{-1} + \alpha/d = p^{-1}\) and \(p\leq q\), then
\[
\begin{split}
  \int \A_\S^\alpha f g dx &= \sum_{Q\in \S} |Q|^\frac{\alpha}{d}\dashint_Q f \int_Q g \\
  &\leq \sum_{Q\in \S} |Q|^{\frac{1}{p}+\frac{1}{q'}}\left(\dashint_Q f^p\right)^{1/p}\left(\dashint_Q g^{q'}\right)^{1/q'}\\
  &= \sum_{Q\in \S} \left(\int_Q f^p\right)^{1/p}\left(\int_Q g^{q'}\right)^{1/q'}\\
  &\leq \left(\sum_{Q\in \S}\int_Q f^p\right)^{1/p}\left(\sum_{Q\in \S}\left(\int_Q g^{q'}\right)^{p'/q'}\right)^{1/p'}\\
  &\leq \|f\|_{L^p}\|g\|_{L^{q'}},
\end{split}
\]
where in the last step we used the fact that \(p'/q'\geq 1\). This estimate would imply an \(L^p\rightarrow L^q\) bound for \(\A^\alpha_\S\), which in turn would imply the same estimate for \(M_\alpha\). In general, we won't be able to find a family of cubes which are pairwise disjoint, but we are still going to be able to find a sufficiently `small' family so that a variation of the argument given above still works. 

To do this we start by assuming we have a locally integrable function \(f\geq 0\) which has compact support, sitting inside a sufficiently large cube, say \(Q_0\). We assume also that \(f\) is not identically zero so that \(\int_{Q_0}f > 0\). Now we consider the sets
\[
G(Q_0) = \{x \in Q_0: M_\alpha f(x) \leq C |Q_0|^\frac{\alpha}{d}\dashint_{Q_0}f \}\text{ and }B(Q_0) = Q_0\setminus G(Q_0),
\]
where \(C>0\) is a sufficiently large constant to be chosen later. In the `good' part of the cube, \(G(Q_0)\), we already have the desired estimate if only we add the cube \(Q_0\) to the family \(\S\). The set \(B(Q_0)\) is then the `bad' part of the cube, where we don't have the desired control. If the bad part had zero measure then we would have nothing to worry about. The issue is when the measure of the bad part is large. However, note that we can use the weak type estimate for \(M_\alpha\) (a particular case of Proposition \ref{strong type estimate for fractional max operator with lebesgue}) in order to show that the bad part has small measure. In fact, if we assume that \(0\leq \alpha < d\) and writing \(q = d/(d-\alpha)\),
\[
\begin{split}
|B(Q_0)| &= |\{x\in Q_0: M_\alpha f(x)> C |Q_0|^{\frac{\alpha}{d}-1}\int_{Q_0}f\}|\leq \frac{\|M_\alpha\|^q_{L^1\rightarrow L^{q,\infty}}}{C^q |Q_0|^{q(\alpha/d-1)}(\int_{Q_0}f)^q} \|f\|_{L^1}^q \\
&= \frac{\|M_\alpha\|^q_{L^1\rightarrow L^{q,\infty}}}{C^q} |Q_0|.
\end{split}
\]
If we choose \(C = 2^{1/q} \|M_\alpha \|\), then we get
\[
|B(Q_0)|\leq \frac{1}{2}|Q_0|,
\]
so the bad part occupies a small region of the cube. The main idea of the method of sparse domination is to iterate this procedure. That is, we now want to find what is the good part and the bad part of \(B(Q_0)\), and we want the new bad part to occupy only a fraction of the measure of \(B(Q_0)\). If we can iterate this procedure then, in the limit, the bad part will have measure zero and outside of that region we will have our desired estimate. 

To obtain the good and bad parts of \(B(Q_0)\) it would be useful to write this as a union of cubes, and then in each cube we use the same 
splitting as before. Of course, to write \(B(Q_0)\) as a union of cubes we can use the ideas of the localized Calderón-Zygmund decomposition. 
In particular, since we are dealing with the fractional maximal operator here, we can simply change to the dyadic maximal operator. 
Since we can control \(M_\alpha\) by \(3^d\) dyadic maximal operators, it is enough to control \(M^\D_\alpha\) for an arbitrary dyadic lattice 
\(\D\). Given that we start the argument by assuming we have a cube that contains the support of the function we have to assume that any such
\(\D\) has only one quadrant. Then, given \(f\in L^1_\text{loc}(\R^d)\) with compact support we can indeed find some \(Q_0\in \D\) such that 
\(\supp(f)\subseteq Q_0\). We split \(Q_0\) as before
\[
Q_0 = G(Q_0) \cup B(Q_0)
\]
and we choose \(C = 2^{1-\alpha/d}\), so that \(|B(Q_0)|\leq |Q_0|/2\). Now we want to write
\[
B(Q_0) = \bigcup_{Q\in \mcal{F}}Q,
\]
where \(\mcal{F}\) are the dyadic cubes \(Q\in \D\) such that 
\[
|Q|^{\frac{\alpha}{d}-1}\int_Q f > C |Q_0|^{\frac{\alpha}{d}-1}\int_{Q_0}f.
\]
To see that this works we argue as before, the only point which is different is that we have to justify why the cubes in the family \(\mcal{F}\) need to be inside of \(Q_0\). First note that, since \(C>1\) then \(Q_0 \notin \mcal{F}\). Now, if \(P\) is an ancestor of \(Q_0\), then since \(\supp(f)\subseteq Q_0\), 
\[
|P|^{\frac{\alpha}{d}-1}\int_P f \leq |Q_0|^{\frac{\alpha}{d}-1}\int_{Q_0}f < C|Q_0|^{\frac{\alpha}{d}-1}\int_{Q_0}f,
\]
so \(P\notin \mcal{F}\). Also, if \(P\) does not intersect \(Q_0\), then \(\int_P f = 0\) and so \(P\notin \mcal{F}\). Thus, if \(P\in \mcal{F}\), then \(P\subsetneq Q_0\). Next we can throw away all cubes in \(\mcal{F}\) that are included in another cube from the family, thus obtaining the family of maximal cubes, \(\mcal{F}^\text{max}\). The maximal cubes exist because \(f\in L^1\). We have a countable number of maximal cubes so we can parameterize them as \(\{Q_{j,1}\}_j\). This way, we have disjoint dyadic cubes inside \(Q_0\) such that
\[
B(Q_0) = \bigcup_j Q_{j, 1}\text{ and }|Q_{j,1}|^{\frac{\alpha}{d}-1}\int_{Q_{j,1}}f > C|Q_0|^{\frac{\alpha}{d}-1}\int_{Q_0}f.
\]
Also, the maximality of these cubes implies that
\[
P \in \D\ \land\ Q_{j, 1}\subsetneq P \implies |P|^{\frac{\alpha}{d}-1}\int_P f \leq C |Q_0|^{\frac{\alpha}{d}-1}\int_{Q_0}f.
\]
Now we can define the good and the bad part of each of these cubes,
\[
G(Q_{j,1}) = \{x\in Q_{j,1}: M_\alpha^\D f(x) \leq C |Q_{j,1}|^{\frac{\alpha}{d}-1}\int_{Q_{j,1}}f \},\ B(Q_{j,1}) = Q_{j,1}\setminus G(Q_{j,1}).
\]
These then allow us to define the good and the bad part of \(B(Q_0)\),
\[
G_1 = \bigcup_j G(Q_{j,1})\text{ and }B_1 = \bigcup_j B(Q_{j,1}).
\]
This gives us a partition of \(Q_0\) as \(B_1 \cup G_1 \cup G_0\), where \(G_0 = G(Q_0)\). Moreover, if \(x\in G_0 \cup G_1\), then
\[
M_\alpha^\D f(x) \leq C \chi_{Q_0}(x) |Q_0|^{\frac{\alpha}{d}-1}\int_{Q_0}f + C\sum_j \chi_{Q_{j,1}}(x) |Q_{j,1}|^{\frac{\alpha}{d}-1} \int_{Q_{j,1}} f,
\]
which gives the desired bound as long as we add the cubes \(\{Q_{j,1}\}_j\) to \(\S\). We can write this more compactly if we set \(Q_{j, 0} = Q_0\). Then we have
\[
M_\alpha ^\D f(x) \leq C \sum_{k=0}^1 \sum_j \chi_{Q_{j, k}}(x) |Q_{j, k}|^{\frac{\alpha}{d}-1}\int_{Q_{j,k}}f,\ \forall x \in Q_0 \setminus B_1.
\]
At this point we want the measure of \(B_1\) to be small when compared to the measure of \(B(Q_0)\). Here we need to be a bit careful because if we simply use the weak type estimate as before then we will get
\[
|B(Q_{j,1})| = |\{x\in Q_{j, 1}: M_\alpha^\D f(x) > C |Q_{j,1}|^{\frac{\alpha}{d}-1}\int_{Q_{j,1}}f\}| \leq \frac{|Q_{j,1}|}{C^q(\int_{Q_{j,1}}f)^q} \left(\int_{Q_0}f\right)^q.
\]
The issue here is that we want the integral on the right-hand side to be over the set \(Q_{j,1}\) instead of \(Q_0\). Fortunately, we can use the fact that the cubes \(Q_{j,1}\) are maximal to solve this problem, because we can show that \(M_\alpha^\D f(x) = M_\alpha^\D (f\chi_{Q_{j,1}})(x)\), for all \(x\in Q_{j,1}\). Indeed, take \(x\in Q_{j,1}\) and let \(Q\in \D\) be a dyadic cube that contains \(x\). Then, either \(Q_{j,1}\subsetneq Q\) or \(Q\subseteq Q_{j,1}\). In the first case \(Q\) is an ancestor of \(Q_{j,1}\) and therefore by maximality we know that
\[
|Q|^{\frac{\alpha}{d}-1}\int_Q f \leq C |Q_0|^{\frac{\alpha}{d}-1}\int_{Q_0}f < |Q_{j,1}|^{\frac{\alpha}{d}-1}\int_{Q_{j,1}}f \leq M_\alpha^\D (f\chi_{Q_{j,1}})(x).
\]
In the other case \(Q \subseteq Q_{j,1}\) so we can simply say that
\[
|Q|^{\frac{\alpha}{d}-1}\int_Q f = |Q|^{\frac{\alpha}{d}-1}\int_Q f\chi_{Q_{j,1}} \leq M_\alpha^\D (f\chi_{Q_{j,1}})(x).
\]
Therefore, we conclude that for all dyadic cubes \(Q\) containing \(x\),
\[
|Q|^{\frac{\alpha}{d}-1}\int_Q f \leq M_\alpha^\D (f\chi_{Q_{j,1}})(x),
\]
and so \(M_\alpha ^\D f(x) \leq M_\alpha^\D (f\chi_{Q_{j,1}})(x)\). Since the other inequality is trivial this is in fact an equality. This way we can estimate \(B(Q_{j,1})\) adequately. We have that
\[
\begin{split}
|B(Q_{j,1})| &= |\{x\in Q_{j,1}: M_\alpha^\D (f \chi_{Q_{j,1}})(x) > C |Q_{j,1}|^{\frac{\alpha}{d}-1}\int_{Q_{j,1}} f\}| \\
&\leq \frac{|Q_{j,1}|}{C^q(\int_{Q_{j,1}}f)^q} \|f\chi_{Q_{j,1}}\|_{L^1}^q \\
&\leq \frac{1}{2}|Q_{j,1}|.
\end{split}
\]
Therefore, 
\[
|B_1| = \sum_j |B(Q_{j,1})| \leq \sum_j \frac{1}{2} |Q_{j,1}| = \frac{1}{2} |B(Q_0)| \leq \frac{1}{2^2}|Q_0|.
\]
We can now iterate this argument (see Figure \ref{fig: sparse cubes}). 
\begin{figure}[ht]
    \centering
    \includegraphics[width=0.25\textwidth]{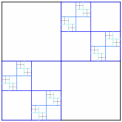}
    \caption{The cube \(Q_0\) with an example of the first iterations of the sparse domination argument. The bad part of \(Q_0\) consists of the top right and bottom left squares. In each of those, the bad part consists of the top left and bottom right squares and so on.}
    \label{fig: sparse cubes}
\end{figure}
We decompose each \(B(Q_{j,1})\) as a union of maximal dyadic cubes
\[
B(Q_{j,1}) = \bigcup_l Q_{j, 1}^l, \text{ where }|Q_{j,1}^l|^{\frac{\alpha}{d}-1}\int_{Q_{j,1}^l}f > C |Q_{j,1}|^{\frac{\alpha}{d}-1}\int_{Q_{j,1}}f.
\]
Then, \(Q_0 = B_2 \cup G_2 \cup G_1\cup G_0\) where 
\[
B_2 = \bigcup_{j, l}B(Q_{j,1}^l) \text{ and } G_2 = \bigcup_{j, l} G(Q_{j,1}^l).
\]
We then have that
\[
M_\alpha^\D f(x) \leq C \sum_{k=0}^1 \sum_j \chi_{Q_{j,k}}(x) |Q_{j,k}|^{\frac{\alpha}{d}-1}\int_{Q_{j,k}}f + C \sum_{j, l}\chi_{Q_{j,1}^l}(x) |Q_{j,1}^l|^{\frac{\alpha}{d}-1}\int_{Q_{j,1}^l}f,
\]
for all \(x \in Q_0 \setminus B_2\). Moreover, as before we can use the maximality of the cubes \(\{Q_{j,1}^l\}\) to argue that \(M_\alpha^\D f(x) = M_\alpha^\D (f\chi_{Q_{j,1}^l})(x)\), for all \(x\in Q_{j,1}^l\), and this in turn implies, by the same weak type estimate argument, that \(|B(Q_{j,1}^l)|\leq |Q_{j,1}^l|/2\). Combining this with the fact that \(|B_1| \leq 2^{-2}|Q_0|\) we see that
\[
|B_2| = \sum_{j,l}|B(Q_{j,1}^l)| \leq \frac{1}{2}\sum_{j,l} |Q_{j,1}^l| = \frac{1}{2}\sum_j |B(Q_{j,1})| = \frac{1}{2} |B_1|\leq \frac{1}{2^3}|Q_0|.
\]
This way, if we reorganize the cubes \(\{Q_{j,1}^l\}_{j,l}\) as \(\{Q_{j, 2}\}_j\) then we see that
\[
Q_0 = B_2 \cup G_2 \cup G_1 \cup G_0,
\]
where 
\[
M_\alpha^\D f(x) \leq C \sum_{k=0}^2 \sum_j \chi_{Q_{j,k}}(x) |Q_{j,k}|^{\frac{\alpha}{d}-1}\int_{Q_{j,k}}f,\ \forall x \in Q_0\setminus B_2,\text{ and }|B_2|\leq \frac{1}{2^3}|Q_0|.
\]
If we carry on in this way, then after \(N+1\) steps we have a decomposition \(Q_0 = B_N \cup G_N \cup \dots \cup G_0\) such that
\[
M_\alpha^\D f(x) \leq C \sum_{k=0}^N \sum_j \chi_{Q_{j,k}}(x) |Q_{j,k}|^{\frac{\alpha}{d}-1}\int_{Q_{j,k}}f,\ \forall x \in Q_0\setminus B_N,\text{ and }|B_N|\leq \frac{1}{2^{N+1}}|Q_0|.
\]
Therefore if we set \(B_\infty = \cap_{N\geq 1} B_N\), then \(|B_\infty| = 0\) and 
\[
M_\alpha^\D f(x) \leq C \sum_{k, j}\chi_{Q_{j,k}}(x) |Q_{j,k}|^{\frac{\alpha}{d}}\int_{Q_{j,k}}f,\ \forall x\in Q_0\setminus B_\infty.
\]
Thus, we have obtained the desired estimate if we add all the cubes \(\{Q_{j,k}\}_{j,k}\) to the family \(\S\). At least, when \(x\in Q_0\). If we want this to hold for almost every \(x\in \R^d\) we have to think a little bit about what happens when \(x\in \R^d\setminus Q_0\). Take any \(x\in \R^d\setminus Q_0\). Since \(\supp(f)\subseteq Q_0\) then we know that any cube \(Q\) in \(\D\) that contains \(x\) and is such that \(\int_Q f > 0\) must be an ancestor of \(Q_0\). Also, if we consider the sequence of ancestors, \(A_k(Q_0)\), then
\[
|A_k(Q_0)|^{\frac{\alpha}{d}-1}\int_{A_k(Q_0)}f = |A_k(Q_0)|^{\frac{\alpha}{d}-1}\int_{Q_0}f \searrow_k 0,
\]
and so
\[
M_\alpha f(x) = |A_{k^*}(Q_0)|^{\frac{\alpha}{d}-1}\int_{A_{k^*}(Q_0)}f,
\]
where \(A_{k^*}(Q_0)\) is the first ancestor of \(Q_0\) that contains \(x\). This way, if we add the cubes \(\{A_k(Q_0)\}_{k\geq 1}\) to the family \(\S\) we see that we have the estimate
\[
M_\alpha^\D f(x) \leq 2^{1-\frac{\alpha}{d}}\sum_{Q\in \S}\chi_Q(x) |Q|^{\frac{\alpha}{d}-1}\int_Q f,\ \text{for a.e. }x\in \R^d.
\]
At this point, we should think about how `large' this family \(\S\) that we have obtained is. Recall that for a family of disjoint cubes we could carry through with the argument, but in this case the cubes in \(\S\) are not disjoint. Despite this, they satisfy the next best thing. Note that although the cubes \(\{Q_{j, k}\}\) are not disjoint, they have large parts which are pairwise disjoint. Indeed, if we consider \(G(Q_{j,k})\), then we see that
\[
|G(Q_{j,k})|\geq \frac{1}{2}|Q_{j,k}|\text{ and the sets in }\{G(Q_{j,k})\}_{j,k}\text{ are pairwise disjoint}.
\]
Moreover, for each ancestor \(A_k(Q_0)\) we can set \(G(A_k(Q_0)) = A_k(Q_0) \setminus A_{k-1}(Q_0)\). This way we see that each \(Q\in \S\) has a measurable subset \(G(Q)\subseteq Q\) such that \(|G(Q)|\geq |Q|/2\) and the sets \(\{G(Q)\}_{Q\in \S}\) are pairwise disjoint. This is what we call a \textit{sparse}\index{sparse family} family of cubes.
\begin{defn}\label{defn: sparse family}
    Let \(0<\eta<1\). We say that a collection of measurable subsets of \(\R^d\), \(\S\), is \(\eta\)-sparse if for each set \(Q\in \S\) there is some measurable subset \(G(Q)\subseteq Q\) such that \(|G(Q)|\geq \eta |Q|\) and the sets \(\{G(Q)\}_{Q\in \S}\) are pairwise disjoint.
\end{defn}
Given a sparse family we can associate to it what we call a sparse operator.
\begin{defn}\label{defn: sparse operator}
    Given an \(\eta\)-sparse family of sets \(\mcal{S}\), we define the associated 
    \textit{sparse operator}\index{sparse operator} as
    \[ 
        \mcal{A}_\S^\alpha f(x) = \sum_{Q\in \S}\chi_Q(x) |Q|^\frac{\alpha}{d}
        \dashint_Q|f|.
    \] 
\end{defn}
We have now proved the following result.
\begin{prop}\label{prop: sparse domination for dyadic M_alpha}
    Let \(\D\) be a dyadic lattice with one quadrant, let \(f\in L^1_\text{loc}(\R^d)\) have compact support and let \(\alpha \in [0,d[\). Then, there is a \(1/2\)-sparse family of cubes \(\S = \S(f)\subseteq \D\) such that
    \[
    M_\alpha^\D f(x) \leq 2^{1-\frac{\alpha}{d}} \A_\S^\alpha f(x) \text{ for a.e. }x\in \R^d,
    \]
\end{prop}
\begin{remark}\label{remark: sparse estimate improvement for fractional maximal operator}
    The proof actually yields a stronger result. In fact, we have proved that 
    \[ 
        M_\alpha^\D f(x) \lesssim \sum_{Q\in \S}\chi_{G(Q)}(x) |Q|^\frac{\alpha}{d}\dashint_Q|f|.
    \] 
    The improvement of going from \(\chi_Q(x)\) to \(\chi_{G(Q)}(x)\) is relevant. However, as we will see later, 
    this improvement does not work in general for other operators. For this reason we have chosen to give here 
    a unified approach without keeping track of this improvement for \(M_\alpha\). For more information about this 
    improved result see \cite{Uribe} where this proposition is shown by a slightly different method.
\end{remark}
Since our goal is to obtain 
two-weight norm estimates for \(M_\alpha\), we need to obtain weighted inequalities 
for sparse operators.

Let \(\S\) be an \(\eta\)-sparse family of cubes, not necessarily dyadic. We have seen above how to obtain an \(L^p \rightarrow L^q\) bound for \(\A_\S^\alpha\) under the assumption that \(\S\) consists of disjoint cubes. The idea now is to use the sparseness property to carry through with the argument even though the cubes from \(\S\) are not necessarily disjoint. To start with we use a duality argument again, so we assume we have \(f, g\geq 0\) and we write
\[
\begin{split}
\int_{\R^d}\A_\S^\alpha f(x) g(x) dx &= \sum_{Q\in \S} |Q|^{\frac{\alpha}{d}}\dashint_Q f \int_Q g \\
&= \sum_{Q\in \S}|Q| |Q|^\frac{\alpha}{d}\dashint_Q f \dashint_Q g \\
&\leq \frac{1}{\eta}\sum_{Q\in \S}|G(Q)||Q|^\frac{\alpha}{d}\dashint_Q f \dashint_Q g.
\end{split}
\]
Now note that if \(x\in Q\), then 
\[
|Q|^\frac{\alpha}{d} \dashint_Q f \leq M_\alpha f(x)
\]
and likewise for \(g\). In particular, the same holds for \(x\in G(Q)\subseteq Q\), so
\[
\int \A_\S^\alpha f g \leq \frac{1}{\eta}\sum_{Q\in \S}\int_{G(Q)}\left(|Q|^\frac{\alpha}{d} \dashint_Q f\right)\left(\dashint_Q g\right)dx\leq \frac{1}{\eta}\sum_{Q\in \S}\int_{G(Q)}M_\alpha f(x) Mg(x)dx.
\]
At this point we simply use the disjointness of the good parts of the cubes to argue that
\[
\int \A_\S^\alpha f g \leq \frac{1}{\eta} \int_{\R^d}M_\alpha f Mg \leq \frac{1}{\eta} \|M_\alpha f \|_{L^q}\|Mg\|_{L^{q'}}.
\]
If we now use Proposition \ref{strong type estimate for fractional max operator with lebesgue}, together with the assumption that \(1/q + \alpha/d = 1/p\) then we see that
\[
\int \A_\S^\alpha f g \lesssim_{\alpha, d, p, q} \frac{1}{\eta} \|f\|_{L^p}\|g\|_{L^{q'}}.
\]
Since this holds for all \(g\in L^{q'}\) it follows by duality that 
\[
\|\A_\S^\alpha f\|_{L^q(\R^d)}\lesssim_{\alpha,d,p,q}\frac{1}{\eta} \|f\|_{L^p(\R^d)}.
\]
This way we see that, even though our cubes are not pairwise disjoint, the sparseness condition still allows us to argue easily that the operator is bounded from \(L^p\) to \(L^q\). Our goal now is to modify this argument to allow for weighted inequalities, instead of using just the Lebesgue measure. 

First we can write
\[
\|\A_\S^\alpha f \|_{L^q(w)} = \| w^{1/q}\A_\S^\alpha  f \|_{L^q},
\]
so it suffices to prove that
\[
\left| \int_{\R^d} \A_\S^\alpha f(x)w(x)^\frac{1}{q} g(x) dx \right| \lesssim \|f\|_{L^p(v)}\|g\|_{L^{q'}},\ \forall g\in L^{q'}.
\]
Again, without loss of generality, we can assume that \(f, g\geq 0\). Now, the left-hand side is
\[
\int_{\R^d} \sum_{Q\in \S}\1_Q(x)g(x) w(x)^\frac{1}{q} |Q|^{\frac{\alpha}{d}-1}\int_Q f dx = \sum_{Q\in \S}|Q|^{\frac{\alpha}{d}-1}\int_Q f \int_Q gw^\frac{1}{q},
\]
where we switched the integral and the sum by the monotone convergence theorem. If we follow the previous argument directly we would get
\[
\begin{split}
\sum_{Q\in \S} |Q| |Q|^{\frac{\alpha}{d}-1}\int_Q f \dashint_Q g w^\frac{1}{q} &\leq \frac{1}{\eta} \sum_{Q\in \S} \int_{G(Q)} \left(|Q|^{\frac{\alpha}{d}-1}\int_Q f\right) \left(\dashint_Q gw^\frac{1}{q}\right) dx \\
&\leq \frac{1}{\eta} \int_{\R^d} M_\alpha f(x) M(gw^\frac{1}{q})(x)dx\\
&\lesssim \|M_\alpha f\|_{L^q}\|M(gw^\frac{1}{q})\|_{L^{q'}}.
\end{split}
\]
This is not ideal for many reasons. First of all, we do not necessarily have the relation \(1/q + \alpha/d = 1/p\), so we cannot bound the fractional maximal operator. Also, even if we could, we would not get the weighted norm we want. Finally, we can bound the Hardy-Littlewood operator by \(\|gw^{1/q}\|_{L^{q'}}\) which is also not what we want. This way, we will have to significantly modify the argument to solve these problems. The way to achieve this is to follow the argument in Theorem 4.5 from \cite{Uribe} with suitable modifications for the present case, which is slightly more general. The key idea to making this work is to consider weighted averages, with respect to carefully chosen weights.

\begin{thm}\label{thm: weighted estimate for sparse operators}
    Let \(\S\) be an \(\eta\)-sparse family of cubes. Suppose that \(1<p\leq q<\infty\) and \(\alpha\geq 0\). Suppose also that we have two weights \(v, w\) such that \((v,w)\in A_{p,q}^\alpha\) and \(v^{-p'/p}, w\in A_\infty\). Then,
    \[
    \|\A_\S^\alpha f\|_{L^q(w)}\lesssim \|f\|_{L^p(v)},\ \forall f\in L^p(v).
    \]
\end{thm}
\begin{proof}
    Let \(f, g\geq 0\) and put \(\sigma = v^{-p'/p}\). Take any \(q_1 \in [p, q]\) and define
    \[
    \begin{dcases}
    \alpha_1 = d\left(\frac{1}{p}-\frac{1}{q_1}\right), \\ 
    \alpha_2 = d\left(\frac{1}{q_1}-\frac{1}{q}\right).
    \end{dcases}
    \]
    Note that, with these definitions, we have that
    \[
    \frac{1}{q_1} + \frac{\alpha_1}{d} = \frac{1}{p},\ \frac{1}{q_1'}+\frac{\alpha_2}{d}=\frac{1}{q'}\text{ and }\alpha_1, \alpha_2\in[0,d[.
    \]
    Now, as we've seen above, to prove the theorem it suffices to show that
    \[
    \left| \int_{\R^d} \A_\S^\alpha f(x)w(x)^\frac{1}{q} g(x) dx \right| \lesssim \|f\|_{L^p(v)}\|g\|_{L^{q'}},
    \]
    and the left hand side is 
    \begin{equation}\label{eq: sum over cubes}
        \sum_{Q\in \S} |Q|^{\frac{\alpha}{d}-1}\int_Q f \int_Q gw^\frac{1}{q}.
    \end{equation}
    Now note that if \(x\in Q\), then
    \[
    M_{\alpha_1,\sigma} (f\sigma^{-1})(x) \geq \frac{1}{\sigma(Q)^{1-\frac{\alpha_1}{d}}} \int_Q f
    \]
    and 
    \[
    M_{\alpha_2,w}(gw^{-\frac{1}{q'}})(x) \geq \frac{1}{w(Q)^{1-\frac{\alpha_2}{d}}}\int_Q gw^\frac{1}{q}.
    \]
    Therefore, we have that
    \[
    \int_Q f \leq \frac{\sigma(Q)^{1-\frac{\alpha_1}{d}}}{\sigma(G(Q))^\frac{1}{q_1}}\left(\int_{G(Q)} M_{\alpha_1,\sigma}(f\sigma^{-1})(x)^{q_1} \sigma(x)dx\right)^{1/q_1}
    \]
    and
    \[
    \int_Q gw^\frac{1}{q} \leq \frac{w(Q)^{1-\frac{\alpha_2}{d}}}{w(G(Q))^\frac{1}{q_1'}} \left(\int_{G(Q)}M_{\alpha_2,w}(gw^{-\frac{1}{q'}})(x)^{q_1'}w(x)dx\right)^{1/q_1'}.
    \]
    In turn, this means that \eqref{eq: sum over cubes} can be controlled by
    \[
     \sum_{Q\in \S}\frac{|Q|^{\frac{\alpha}{d}-1}\sigma(Q)^{1-\frac{\alpha_1}{d}}w(Q)^{1-\frac{\alpha_2}{d}}}{\sigma(G(Q))^\frac{1}{q_1}w(G(Q))^\frac{1}{q_1'}}\left(\int_{G(Q)} M_{\alpha_1,\sigma}(f\sigma^{-1})^{q_1} \sigma\right)^{\frac{1}{q_1}}\left(\int_{G(Q)}M_{\alpha_2,w}(gw^{-\frac{1}{q'}})^{q_1'}w\right)^{\frac{1}{q_1'}}.
    \]
    From the assumptions we know that both \(\sigma\) and \(w\) are in \(A_\infty\), which means that there are \(\delta_1, \delta_2>0\) such that
    \[
    \frac{\sigma(Q)}{\sigma(G(Q))}\lesssim \left(\frac{|Q|}{|G(Q)|}\right)^{\delta_1}\lesssim \frac{1}{\eta^{\delta_1}}\text{ and }\frac{w(Q)}{w(G(Q))}\lesssim \left(\frac{|Q|}{|G(Q)|}\right)^{\delta_2}\lesssim \frac{1}{\eta^{\delta_2}}.
    \]
    Therefore,
    \[
    \begin{split}
    |Q|^{\frac{\alpha}{d}-1}\frac{\sigma(Q)^{1-\frac{\alpha_1}{d}}}{\sigma(G(Q))^\frac{1}{q_1}} \frac{w(Q)^{1-\frac{\alpha_2}{d}}}{w(G(Q))^\frac{1}{q_1'}}&\lesssim |Q|^{\frac{\alpha}{d}-1}\frac{\sigma(Q)^{1-\frac{\alpha_1}{d}}}{\sigma(Q)^\frac{1}{q_1}} \frac{w(Q)^{1-\frac{\alpha_2}{d}}}{w(Q)^\frac{1}{q_1'}}\\
    &\lesssim |Q|^{\frac{\alpha}{d}-1}\sigma(Q)^{1/p'}w(Q)^{1/q}\\
    &\lesssim [v, w]_{A_{p,q}^\alpha}.
    \end{split}
    \]
    If we now use Hölder's inequality for series and the fact that the sets \(G(Q)\) are pairwise disjoint we obtain
    \[
    \begin{split}
        \int \A_\S^\alpha f g w^\frac{1}{q}&\lesssim \sum_{Q\in \S}\left(\int_{G(Q)} M_{\alpha_1,\sigma}(f\sigma^{-1})^{q_1} \sigma\right)^{1/q_1}\left(\int_{G(Q)}M_{\alpha_2,w}(gw^{-\frac{1}{q'}})^{q_1'}w\right)^{1/q_1'}\\
        &\lesssim \left(\sum_{Q\in \S}\int_{G(Q)} M_{\alpha_1,\sigma}(f\sigma^{-1})^{q_1} \sigma\right)^{1/q_1}\left(\sum_{Q\in \S}\int_{G(Q)}M_{\alpha_2,w}(gw^{-\frac{1}{q'}})^{q_1'}w\right)^{1/q_1'}\\
        &\lesssim \|M_{\alpha_1,\sigma}(f\sigma^{-1})\|_{L^{q_1}(\sigma)}\|M_{\alpha_2, w}(gw^{-\frac{1}{q'}})\|_{L^{q_1'}(w)}\\
        &\lesssim \|f\sigma^{-1}\|_{L^p(\sigma)} \|gw^{-\frac{1}{q'}}\|_{L^{q'}(w)}\\
        &\lesssim \|f\|_{L^p(v)}\|g\|_{L^{q'}},
    \end{split}
    \]
    where we have applied Proposition \ref{strong type estimate for fractional max operator with a weight}, whose hypothesis are satisfied here because \(\sigma, w\in A_\infty\).
\end{proof}

\begin{remark}
    Looking at the proof of this theorem, one might think that the \(A_\infty\) condition is overkill. After all, we only use 
    the estimate 
    \[ 
        \frac{w(Q)}{w(A)}\lesssim \left(\frac{|Q|}{|A|}\right)^\delta,
    \] 
    and the fact that quadrants have infinite \(w\) measure. It may seem like the first property could be replaced by 
    the apparently weaker condition that the weights satisfy 
    \[ 
        \frac{w(Q)}{w(A)}\leq f\left(\frac{|Q|}{|A|}\right),
    \] 
    for some positive and increasing function \(f\). As it turns out, this condition implies that \(w\in A_\infty\), as 
    we show in appendix \ref{appendix: weighted spaces}, Proposition \ref{prop: order of growth of Ainfty}.
\end{remark}

\begin{remark}
    In Definition \ref{defn: sparse family} we considered sparse operators associated
    to families of generic sets. However, from the above proof we see that
    to actually obtain useful estimates we also need to be able to control the 
    maximal operator associated to the family of sets we use in our sparse operator. 
    This way, it was essential for us to use cubes in the construction of our sparse
    families. Other families of sets can be problematic because we lose control of the 
    associated maximal operator, even for sets which are very similar to cubes. For example,
    had we considered generic products of intervals (i.e. `rectangles'), then we would 
    have run into issues trying to obtain the boundedness of the strong maximal operator 
    with respect to a measure \(\mu = w(x)dx\). Still, the strong maximal operator is bounded on \(L^p\) 
    when \(p > 1\). An even more challenging situation would arise if we consider rectangles with arbitrary 
    orientations. In that case, the operator is not even bounded on \(L^p\) (see Exercise 2.1.9 in \cite{Cgrafakos}).
    In general, we need to be careful when considering families of sets which have uncontrolled excentricity and many 
    normal directions. Although it is interesting to give examples of sparse families which do not come from cubes, it is 
    important to emphasize that in this text all the fundamental results regarding sparse domination work in the 
    specific case of sparse families of cubes.
\end{remark}

Combining this theorem with Proposition \ref{prop: sparse domination for dyadic M_alpha}, we obtain the following result.
\begin{prop}\label{prop: two weight estimate for dyadic frac max op}
    Let \(\D\) be a dyadic lattice with one quadrant, let \(1<p\leq q<\infty\) and \(\alpha\in [0, d[\). Suppose that \(v, w\) are weights satisfying \((v,w)\in A_{p, q}^\alpha\) and \(v^{-p'/p},w\in A_\infty\). Then,
    \[
    \|M_\alpha^\D f\|_{L^q(w)}\lesssim \|f\|_{L^p(v)},\ \forall f\in L^p(v).
    \]
\end{prop}
\begin{proof}
    Let \(f\in L^p(v)\). For each \(N\) set \(f_N = f \chi_{B(0,N)}\). Then \(f_N\in L^1_\text{loc}\) with compact support, 
    and therefore we can apply Proposition \ref{prop: sparse domination for dyadic M_alpha} to conclude that 
    there is a \(1/2\)-sparse family of cubes \(\S_N = \S_N(f_N)\subseteq \D\) such that
    \[
        M_\alpha^\D(f_N)(x) \leq 2^{1-\frac{\alpha}{d}} \A_{\S_N}^\alpha f_N(x) \text{ a.e. }x\in \R^d.
    \]
    But then, from Theorem \ref{thm: weighted estimate for sparse operators},
    \[
    \|M_\alpha^\D f_N\|_{L^q(w)}\lesssim \|f_N\|_{L^p(v)}\lesssim \|f\|_{L^p(v)}.
    \]
    To finish the proof it is enough to show that \(M_\alpha^\D f_N (x) \rightarrow_N M_\alpha^\D f(x)\), because if this is the case we have, by Fatou's lemma,
    \[
    \|M_\alpha^\D f\|_{L^q(w)} = \|\liminf_N M_\alpha^\D f_N\|_{L^q(w)}\leq \liminf_N \|M_\alpha^\D f_N\|_{L^q(w)}\lesssim \|f\|_{L^p(v)}.
    \]
    First suppose that \(M_\alpha^\D f(x)<\infty\). Given \(\delta>0\), there is some cube \(Q\in \D\) such that \(x\in Q\) and
    \[
    M_\alpha^\D f(x) - \frac{\delta}{2} < |Q|^{\frac{\alpha}{d}-1}\int_Q |f|.
    \]
    By the monotone convergence theorem we know that there is some \(N_0\) such that for \(N\geq N_0\),
    \[
    |Q|^{\frac{\alpha}{d}-1}\int_Q|f| - \frac{\delta}{2} < |Q|^{\frac{\alpha}{d}-1}\int_Q |f_N|.
    \]
    Therefore,
    \[
    M_\alpha^\D f(x) -  \delta < |Q|^{\frac{\alpha}{d}-1}\int_Q|f| - \frac{\delta}{2} < |Q|^{\frac{\alpha}{d}-1}\int_Q |f_N| \leq M_\alpha^\D f_N(x) \leq M_\alpha^\D f(x),\ \forall N\geq N_0.
    \]
    This shows that \(M_\alpha^\D f_N(x) \rightarrow_N M_\alpha^\D f(x)\). If \(M_\alpha ^\D f(x) = +\infty\) the argument is analogous.
\end{proof} 

Finally, if we approximate the fractional maximal operator by its dyadic version using \eqref{eq: control by dyadic max opers}, with 
dyadic lattices with one quadrant, then we arrive at the following result.
\begin{prop}\label{prop: two weight estimate for frac max op}
    Let \(1<p\leq q < \infty\) and \(\alpha\in [0, d[\). Suppose that \(v, w\) are weights such that \((v, w)\in A_{p, q}^\alpha\) and \(v^{-p'/p}, w \in A_\infty\). Then,
    \[
    \|M_\alpha f\|_{L^q(w)}\lesssim \|f\|_{L^p(v)},\ \forall f\in L^p(v).
    \]
\end{prop}

\subsection{Examples of weights}\label{sec: examples of weights}

Proposition \ref{prop: two weight estimate for frac max op} gives very general conditions under which a two-weight estimate for the fractional maximal operator holds. 
These conditions, namely that \((v, w)\in A_{p, q}^\alpha\) and \(v^{-p'/p}, w\in A_\infty\), 
come from the comparison of the fractional maximal operator with the sparse forms \(\A_\S^\alpha\). 
Thus, the essential result for obtaining these two-weight inequalities is Theorem \ref{thm: weighted estimate for sparse operators}. 
Of course, we should consider what kinds of weights satisfy the \(A_{p, q}^\alpha\) and \(A_\infty\) conditions. 
One might imagine that such conditions are in general difficult to verify, but for many weights which are relevant in practice, 
we can actually fully characterize when they satisfy these types of conditions. There are three families of weights which we will focus on throughout the text. 
These are the power law\index{power law weights} weights \(w(x) = |x|^\gamma\), the inhomogeneous weights\index{inhomogeneous weights} \(w(x) = \langle x \rangle^\gamma = (1 + |x|^2)^{\gamma/2}\), 
and the monomial weights\index{monomial weights} \(w(x) = |x_1|^{\alpha_1}|x_2|^{\alpha_2}\dots |x_d|^{\alpha_d}\). 
The main goal of this subsection is to characterize when a weight from one of these three families satisfies the \(A_{p,q}^\alpha\) condition.

\subsubsection{Power law weights}

We first look at weights of the form \(w(x) = |x|^\gamma\). Our goal is to describe the set of parameters \((\beta, \gamma)\) such that \((|x|^\beta, |x|^\gamma)\in A_{p, q}^\alpha\). To do this, start by fixing \(p, q, \alpha\) such that \(1 < p, q < \infty\) and \(\alpha\in \R\). We want to find all pairs \((\beta, \gamma)\) such that
\[
\sup_Q |Q|^{\frac{\alpha}{d}-1}\left(\int_Q |x|^{-\beta \frac{p'}{p}}\right)^{1/p'}\left(\int_Q |x|^\gamma\right)^{1/q} < \infty.
\]
To simplify the notation we put \(a = -\beta p'/p\) and \(b = \gamma\). This way, we want to find all \((a, b)\) such that 
\[
\sup_Q |Q|^{\frac{\alpha}{d}-1}\left(\int_Q|x|^a\right)^{1/p'}\left(\int_Q|x|^b\right)^{1/q}<\infty.
\]
First and foremost the functions need to be locally integrable so we need \(a>-d\) and \(b>-d\). Clearly, the pairs that satisfy this condition are the same as the pairs that satisfy the corresponding condition for balls. This way, we consider balls instead of cubes to take advantage of the radial symmetry of the power law weights. Here we use an idea as in \cite{Cgrafakos}, which is to split between balls of type I\index{type I and type II balls} and balls of type II. A ball \(B=B(x_0,r)\) is said to be of type I if \(2r<|x_0|\) and of type II if \(2r\geq |x_0|\). For balls of type I the radius is small so we can compare all points to the center point, whereas balls of type II can be compared to balls centered at the origin. Suppose \(B=B(x_0,r)\) is a ball of type I. Then, if \(x\in B\), we have that
\[
\frac{1}{2}|x_0|\leq |x-x_0|+|x|-\frac{1}{2}|x_0|\leq |x|\leq |x-x_0|+|x_0|\leq \frac{3}{2}|x_0|,
\]
so \(|x|\sim |x_0|\), for all \(x\in B\). This means that
\[
|B|^{\frac{\alpha}{d}-1}\left(\int_B|x|^a\right)^{1/p'}\left(\int_B|x|^b\right)^{1/q}\sim |B|^{\frac{\alpha}{d}-1+\frac{1}{p'}+\frac{1}{q}}|x_0|^{\frac{a}{p'}+\frac{b}{q}}\sim r^{\alpha-\frac{d}{p}+\frac{d}{q}}|x_0|^{\frac{a}{p'}+\frac{b}{q}}.
\]
If we fix \(x_0\neq 0\) and let \(r\rightarrow 0^+\) we see that the condition 
\[
\alpha-\frac{d}{p}+\frac{d}{q}\geq0
\]
is necessary.\footnote{In fact, we know from Remark \ref{remark: Apq alpha implies sub-sobolev relation} that this condition had to be necessary.}
Also, if we pick \(r=|x_0|/4, x_0\neq 0\), we see that for this to be bounded we need the condition 
\begin{equation}\label{eq: first dimensional balance for power-laws}
\alpha-\frac{d}{p}+\frac{d}{q}+\frac{a}{p'}+\frac{b}{q}=0.
\end{equation}
If both of these conditions are satisfied then the supremum over type I balls is finite. Indeed, from \(\alpha - d/p + d/q \geq 0\) and \(r<|x_0|/2\) we see that
\[
r^{\alpha - \frac{d}{p}+\frac{d}{q}}|x_0|^{\frac{a}{p'}+\frac{b}{q}}\lesssim |x_0|^{\alpha - \frac{d}{p}+\frac{d}{q} + \frac{a}{p'}+\frac{b}{q}} \lesssim 1,
\]
where in the last step we used \eqref{eq: first dimensional balance for power-laws}. If instead we consider a type II ball \(B=B(x_0,r)\), then for any \(x\in B\) we see that \(|x|\leq |x-x_0| + |x_0| < 3r\). Therefore, \(B\subseteq B(0,3r)\). This then leads us to consider balls that are centered at the origin. We have that
\[
\begin{split}
|B(0,R)|^{\frac{\alpha}{d}-1}\left(\int_{B(0,R)}|x|^a\right)^\frac{1}{p'}\left(\int_{B(0,R)}|x|^b\right)^\frac{1}{q}&\sim R^{\alpha-d}\left(\int_0^R\rho^{a+d-1}d\rho\right)^\frac{1}{p'}\left(\int_0^R\rho^{b+d-1}d\rho\right)^\frac{1}{q}\\
&\sim R^{\alpha-d+\frac{a+d}{p'}+\frac{b+d}{q}}.
\end{split}
\]
So, this quantity remains bounded for all \(R>0\) if and only if condition \eqref{eq: first dimensional balance for power-laws} is satisfied. These considerations lead us to the following lemma. 
\begin{lemma}\label{lemma: Apq alpha condition for power-laws}
Suppose \(1<p,q<\infty\) and \(\alpha\in \R\) is such that \(\alpha \geq d/p-d/q\). Then,
\[
(|x|^\beta, |x|^\gamma)\in A_{p,q}^\alpha \iff \alpha-\frac{d}{p}+\frac{d}{q}=\frac{\beta}{p}-\frac{\gamma}{q},\ \beta<d(p-1)\text{ and }\gamma>-d.
\]
\end{lemma}
Recalling that the \(A_{p,q}\) condition is just the \(A_{p,q}^\alpha\) condition with \(\alpha=d/p-d/q\), it follows that
\[
(|x|^\beta, |x|^\gamma)\in A_{p,q} \iff \frac{\beta}{p}=\frac{\gamma}{q},\ \beta<d(p-1)\text{ and }\gamma>-d.
\]
In particular,
\[
|x|^\gamma\in A_p\iff (|x|^\gamma,|x|^\gamma)\in A_{p,p}\iff -d<\gamma<d(p-1).
\]
As an important consequence of this characterization, we see that any locally integrable power law weight \(w(x) = |x|^\gamma\), with \(\gamma > -d\), will belong to some \(A_p\) class, for instance \(w \in A_{\gamma/d + 2}\). In turn, this implies that any such weight always belongs to \(A_\infty\). The definition of \(A_\infty\) weights is in general difficult to verify, 
and one might have presumed it would be difficult to describe which power law weights are in \(A_\infty\). 
However, not only is it possible to fully characterize which power law weights satisfy the \(A_\infty\) condition, as it turns out, the condition is always true for power law weights.

\subsubsection{Inhomogeneous weights}

The next interesting example consists of inhomogeneous weights \(w(x) = \langle x \rangle^\gamma\). Again we wish to see when such weights satisfy the conditions of Theorem \ref{thm: weighted estimate for sparse operators}. 

Homogeneous weights \(|x|^\gamma\) are only locally integrable when \(\gamma > -d\), 
but the same is not true for inhomogeneous weights, which are always locally 
integrable. This makes the characterization of inhomogeneous weights more 
cumbersome. However, despite them always being locally integrable, the weights \(w(x) = \langle x\rangle^\gamma\) are in \(A_\infty\) only if \(\gamma > -d\). This way, we will characterize the pairs \((\beta, \gamma)\) such that \((\langle x\rangle^\beta, \langle x \rangle^\gamma)\in A_{p, q}^\alpha\) under the assumption that \(a = -\beta p'/p, b = \gamma > -d\). We note that we could do a full characterization without this constraint, but for our purposes this result will be enough and its proof is simpler. 

We start by checking that indeed \(\langle x \rangle^\gamma \in A_\infty\) if and only if \(\gamma > -d\). To do this we first look at the \(A_p\) condition
\[
[\langle x\rangle^\gamma]_{A_p} = \sup_Q \left(\dashint_Q \langle x \rangle ^{-\gamma \frac{p'}{p}}\right)^{p-1}\dashint_Q \langle x \rangle^\gamma.
\]

\begin{lemma}\label{lemma: Ap characterization for non-hom weights}
    Let \(1 < p < \infty\) and consider the weight \(w(x) = \langle x \rangle^\gamma\). Then,
    \[
    w\in A_p \iff -d < \gamma < d(p-1).
    \]
\end{lemma}
\begin{proof}
As before, since our weight is radial it is more convenient to work with balls instead of cubes. We start by looking at balls centered at the origin \(B = B(0, r)\). Define
\[
\Phi(r, a) = \int_{B(0, r)}\langle x\rangle^a dx =d|B_1| \int_0^r\rho^{d-1}(1 + \rho^2)^{a/2}d\rho. 
\]
With this notation,
\[
\left(\dashint_B \langle x \rangle ^{-\gamma \frac{p'}{p}}\right)^{p-1}\dashint_B \langle x \rangle^\gamma \sim r^{-dp} \Phi\left(r, -\gamma \frac{p'}{p}\right)^{p-1}\Phi(r, \gamma).
\]
It is clear that
\[
\lim_{r\rightarrow 0}\frac{\Phi(r,a)}{r^d}=|B_1|.
\]
Therefore,
\[
r^{-dp} \Phi\left(r, -\gamma \frac{p'}{p}\right)^{p-1}\Phi(r, \gamma) = \left(\frac{\Phi(r, -\gamma p'/p)}{r^d}\right)^{p-1}\frac{\Phi(r, \gamma)}{r^d}
\]
has a finite limit as \(r\rightarrow 0^+\). Now we look at what happens as \(r\rightarrow +\infty\). If \(a \neq -d\) and \(r>1\), we have that 
\[
\begin{split}
\Phi(r,a)&=d|B_1|\int_0^1\rho^{d-1}(1+\rho^2)^{a/2}d\rho+d|B_1|\int_1^r\rho^{d-1}(1+\rho^2)^{a/2}d\rho \\
&\sim_ad|B_1|\int_0^1\rho^{d-1}d\rho+d|B_1|\int_1^r\rho^{d-1+a}d\rho\\
&\sim_a|B_1|+d|B_1|\frac{r^{a+d}-1}{a+d}\sim_{a,d}\frac{a}{a+d}+\frac{d}{a+d}r^{a+d}.
\end{split}
\]
On the other hand, if \(a = -d\), then
\[
\begin{split}
\Phi(r, -d) &= d|B_1|\int_0^1 \rho^{d-1}(1+\rho^2)^{-d/2}d\rho  + d|B_1|\int_1^r\rho^{d-1}(1+\rho^2)^{-d/2}d\rho \\
&\sim_d d|B_1|\int_0^1\rho^{d-1}d\rho + d|B_1|\int_1^r\rho^{-1}d\rho \\
&\sim_d \frac{1}{d} + \log(r).
\end{split}
\]
Now let's assume that \(\gamma < -d\). For ease of notation set \(a = -\gamma p'/p\). In this case, \(a > dp'/p > -d\), so 
\[
\begin{split}
\left(\frac{\Phi(r, -\gamma p'/p)}{r^d}\right)^{p-1}\frac{\Phi(r, \gamma)}{r^d} &\sim \left(\frac{a}{a+d}r^{-d} + \frac{d}{a+d}r^a\right)^{p-1}\left(\frac{\gamma}{\gamma+d}r^{-d} + \frac{d}{\gamma + d}r^\gamma\right) \\
&\sim \frac{\gamma}{\gamma + d}\left(\frac{a}{a+d}r^{-d-d\frac{p'}{p}}+\frac{d}{a+d}r^{a-d\frac{p'}{p}}\right)^{p-1} \\
&+ \frac{d}{\gamma + d}\left(\frac{a}{a+d}r^{-d+\gamma \frac{p'}{p}} + \frac{d}{a+d}r^{a+\gamma \frac{p'}{p}}\right)^{p-1}.
\end{split}
\]
Looking at this expression we see that it blows up as \(r\rightarrow +\infty\) because \(a-dp'/p = -\gamma p'/p - d p'/p = -(\gamma + d)p'/p > 0\). Therefore when \(\gamma < -d\), \([\langle x \rangle^\gamma]_{A_p} = +\infty\). 

Now let's look at the case \(\gamma = -d\). In this case, we have that
\[
\begin{split}
\left(\frac{\Phi(r, d p'/p)}{r^d}\right)^{p-1}\frac{\Phi(r, -d)}{r^d} &\sim \left(\frac{a}{a+d}r^{-d} + \frac{d}{a+d}r^a\right)^{p-1}\left(\frac{1}{d} + \log(r)\right).
\end{split}
\]
Therefore, this expression also blows up as \(r\rightarrow +\infty\), showing that \(\gamma>-d\) is a necessary condition for the weight \(\langle x \rangle^\gamma\) to be in \(A_p\). This also implies that we must have \(a > -d\). Indeed, it is straightforward to check that
\[
[\langle x \rangle^\gamma]_{A_p} = [\langle x \rangle^a]_{A_{p'}}^{p-1}.
\]
This way,
\[
\langle x\rangle^\gamma \in A_p \implies \langle x \rangle^a \in A_{p'} \implies a > -d.
\]
Therefore, the condition
\[
-d < \gamma < d(p-1)
\]
is necessary for \(\langle x \rangle^\gamma \in A_p\). In fact, this condition is also sufficient. To check this we assume that \(-d<\gamma < d(p-1)\) for now on. If we have a type I ball \(B = B(x_0, r)\), then
\[
\left(\dashint_B \langle x \rangle^a\right)^{p-1}\dashint_B \langle x \rangle^\gamma \sim \langle x_0\rangle^{a(p-1)}\langle x_0\rangle^\gamma \sim 1.
\]
For a type II ball \(B = B(x_0, r)\) we have that \(B\subseteq B(0, 3r)\) so it remains only to show that 
\[
\sup_{r>0}\left(\frac{\Phi(r, -\gamma p'/p)}{r^d}\right)^{p-1}\frac{\Phi(r, \gamma)}{r^d}<\infty.
\]
Since \(\gamma \neq -d\) and \(a \neq -d\), then we have the estimate from before
\[
\begin{split}
    \left(\frac{\Phi(r, -\gamma p'/p)}{r^d}\right)^{p-1}\frac{\Phi(r, \gamma)}{r^d} &\sim \frac{\gamma}{\gamma + d}\left(\frac{a}{a+d}r^{-d-d\frac{p'}{p}}+\frac{d}{a+d}r^{a-d\frac{p'}{p}}\right)^{p-1} \\
&+ \frac{d}{\gamma + d}\left(\frac{a}{a+d}r^{-d+\gamma \frac{p'}{p}} + \frac{d}{a+d}\right)^{p-1}.
\end{split}
\]
But now this expression remains bounded as \(r \rightarrow +\infty\), because \(a - d p'/p < 0\) and \(-d + \gamma p'/p < 0\). 
Since this is bounded both as \(r\rightarrow 0^+\) and \(r\rightarrow +\infty\) and the expression is continuous in \(r\), it follows that the supremum is finite. 
\end{proof}

As a consequence of this lemma we see that any weight \(\langle x \rangle ^\gamma\) with \(\gamma > -d\) will be an \(A_\infty\) weight. But this is also a necessary condition. Indeed, if \(\langle x \rangle^\gamma \in A_\infty\), then there is some \(p\) such that \(\langle x \rangle^\gamma \in A_p\), which in turn implies that \(\gamma > -d\). Therefore, we have that
\[
\langle x \rangle^\gamma \in A_\infty \iff \gamma > -d.
\]
It is curious to note that in this case there is no failure of local integrability when \(\gamma \leq -d\), which is in contrast to the case of power law weights. However, it makes sense that we need to assume \(\gamma\) is not too small for the weight to be in \(A_\infty\). Indeed, recall from Remark \ref{remark: A_infty weights give infinite measure to quadrants} that any \(A_\infty\) weight must give infinite measure to a quadrant of a dyadic lattice. In particular, they must give infinite measure to \(\R^d\). This is the case for \(w(x) = \langle x \rangle^\gamma\) only if \(\gamma \geq -d\). 

If we are interested in finding the pairs \((\beta, \gamma)\) such that \((\langle x \rangle^\beta, \langle x \rangle^\gamma)\in A_{p, q}^\alpha\) and \(\langle x \rangle^{-\beta p'/p}, \langle x\rangle^\gamma \in A_\infty\), we already know that the conditions
\[
\beta < d(p-1)\text{ and }\gamma > -d
\]
must hold. For this reason we restrict our attention to this range of exponents.

Let us now check when the \(A_{p,q}^\alpha\) condition holds. Let \(1<p,q<\infty\), \(\alpha\in \R\) and set \(a = -\beta p'/p, b = \gamma\). We will assume throughout that \(a, b > -d\) and we are interested in knowing for which pairs \((a, b)\) the condition
\[
\sup_{B}|B|^{\frac{\alpha}{d}-1}\left(\int_B\langle x\rangle^a\right)^{1/p'}\left(\int_B\langle x\rangle^b\right)^{1/q}<\infty
\]
is true. Suppose \(B=B(x_0,r)\) is a type I ball, i.e. \(2r<|x_0|\). Then,
\[
|B|^{\frac{\alpha}{d}-1}\left(\int_B\langle x\rangle^a\right)^{1/p'}\left(\int_B\langle x\rangle^b\right)^{1/q}\sim r^{\alpha-\frac{d}{p}+\frac{d}{q}}\langle x_0\rangle^{\frac{a}{p'}+\frac{b}{q}}.
\]
By fixing \(x_0\) and letting \(r\rightarrow 0\) we see that the condition \(\alpha-d/p+d/q\geq 0\) is necessary. Also, choosing \(r=\langle x_0\rangle /4, |x_0|>1\), and letting \(|x_0|\rightarrow +\infty\) we see that the condition 
\[
\alpha-\frac{d}{p}+\frac{d}{q}+\frac{a}{p'}+\frac{b}{q}\leq 0
\]
is also necessary. These two conditions combined are sufficient to control all type I balls and since they are necessary we assume they hold from now on. Since type II balls can be compared to balls centered at the origin we now consider a ball \(B=B(0,r)\). We use the same notation as before and we write
\[
\Phi(r,a)=\int_{B(0,r)}\langle x\rangle^adx=d|B_1|\int_0^r\rho^{d-1}(1+\rho^2)^{a/2}d\rho.
\]
From 
\[
|B(0,r)|^{\frac{\alpha}{d}-1}\Phi(r,a)^{1/p'}\Phi(r,b)^{1/q}\sim r^{\alpha-\frac{d}{p}+\frac{d}{q}}\left(\frac{\Phi(r,a)}{r^d}\right)^{1/p'}\left(\frac{\Phi(r,b)}{r^d}\right)^{1/q}
\]
we see that this has a finite limit as \(r\rightarrow 0\) if and only if \(\alpha-d/p+d/q\geq 0\). Now, since \(a>-d\), then for \(r>1\) we have that
\[
\begin{split}
\Phi(r,a)&\sim_{a,d}\frac{a}{a+d}+\frac{d}{a+d}r^{a+d},
\end{split}
\]
as before. Using this we get for \(r>1\),
\[
\begin{split}
r^{\alpha-d}\Phi(r,a)^{1/p'}\Phi(r,b)^{1/q}&\sim r^{\alpha-d}\left( \frac{a}{a+d}+\frac{d}{a+d}r^{a+d}\right)^{1/p'}\left(\frac{b}{b+d}+\frac{d}{b+d}r^{b+d}\right)^{1/q}\\
&\sim \left[ \frac{a}{a+d}\left( \frac{b}{b+d}r^{q(\alpha-d)}+\frac{d}{b+d}r^{q(\alpha-d)+b+d} \right)^{\frac{p'}{q}}+\right.\\
&\left.\frac{d}{a+d}\left( \frac{b}{b+d}r^{q(\alpha-d)+\frac{q}{p'}(a+d)}+\frac{d}{b+d}r^{q(\alpha-d)+b+d+\frac{q}{p'}(a+d)} \right)^{\frac{p'}{q}}  \right]^{\frac{1}{p'}}
\end{split}
\]
which remains bounded as \(r\rightarrow +\infty\) if and only if \(\alpha\leq d,\ q(\alpha-d)+b+d\leq 0\) and \(q(\alpha-d)+q(a+d)/p'\leq 0\). Therefore we have shown the following lemma.

\begin{lemma}\label{lemma: Apq alpha characterization of non-hom weights}
Suppose \(1<p,q<\infty\), \(\beta<d(p-1),\ \gamma>-d\) and \(\alpha \geq d/p - d/q\).
Then, the pair \((\langle x\rangle^\beta, \langle x\rangle^\gamma)\) satisfies the \(A_{p,q}^\alpha\) condition if and only if
\[
\alpha-\frac{d}{p}+\frac{d}{q}\leq \frac{\beta}{p}-\frac{\gamma}{q},\ \alpha\leq d,\ \frac{\gamma}{q}\leq \frac{d}{q'}-\alpha,\ \frac{\beta}{p}\geq \alpha-\frac{d}{p}.
\]
\end{lemma}

As expected we obtain a larger range of parameters for the inhomogeneous weights when compared to the power law weights. Indeed, in the power law case, the conditions were
\[
\alpha \geq \frac{d}{p} - \frac{d}{q}, \alpha - \frac{d}{p}+\frac{d}{q} = \frac{\beta}{p}-\frac{\gamma}{q}, \beta < d(p-1), \gamma> -d.
\]
Under these conditions, we have that
\[
\alpha = \frac{d}{p}-\frac{d}{q} + \frac{\beta}{d} - \frac{\gamma}{q} < \frac{d}{p}-\frac{d}{q} + \frac{d}{p'} + \frac{d}{q} = d.
\]
Moreover,
\[
\frac{\gamma}{q} = \frac{\beta}{p} - \alpha + \frac{d}{p}-\frac{d}{q} < d - \frac{d}{q} - \alpha = \frac{d}{q'}-\alpha,
\]
and
\[
\frac{\beta}{p} = \frac{\gamma}{q} + \alpha - \frac{d}{p} + \frac{d}{q} > -\frac{d}{q} + \alpha - \frac{d}{p}+\frac{d}{q} = \alpha - \frac{d}{p}.
\]
This shows that, in fact, for the inhomogeneous weights the range of parameters contains the range for the power law weights. This makes sense since the weights \(w(x) = \langle x \rangle^\gamma\) are better behaved near the origin.

\subsubsection{Monomial weights}

The third example we look at are monomial weights \(w(x) = |x_1|^{\gamma_1}\dots |x_d|^{\gamma_d}\). This example is interesting because the weights are no longer radial. Despite not being radial, they are still easy to handle because they are tensor products of radial functions in one variable. This makes it very easy to determine when such weights are in \(A_p\). Indeed, if we write \(Q = I_1\times \dots \times I_d\) for a general cube, where \(I_j = [y_j, y_j + l[\), then we have that
\[
\left(\dashint_Q w^{-\frac{p'}{p}}\right)^{p-1}\dashint_Q w = \prod_{j=1}^d\left(\dashint_{I_j}|x_j|^{-\gamma_j \frac{p'}{p}}dx_j\right)^{p-1}\dashint_{I_j}|x_j|^{\gamma_j}dx_j.
\]
In particular, we see that the conditions 
\[
-1 < \gamma_j < p-1,\ j=1,\dots, d
\]
are necessary to ensure local integrability. But, these conditions will also be sufficient because, if they hold, then
\[
\prod_{j=1}^d\left(\dashint_{I_j}|x_j|^{-\gamma_j \frac{p'}{p}}dx_j\right)^{p-1}\dashint_{I_j}|x_j|^{\gamma_j}dx_j \leq [|x_1|^{\gamma_1}]_{A_p}\dots [|x_d|^{\gamma_d}]_{A_p} 
\]
and the right-hand side is finite. Thus,
\[
|x_1|^{\gamma_1}\dots |x_d|^{\gamma_d} \in A_p \iff -1 < \gamma_j < p-1, \forall j.
\]
In particular, we see that 
\[
|x_1|^{\gamma_1}\dots |x_d|^{\gamma_d}\in A_\infty \iff \gamma_j > -1, \forall j.
\]
More generally, we want to know which tuples \((\beta_1,\dots, \beta_d), 
(\gamma_1,\dots, \gamma_d)\) satisfy 
\[
(|x_1|^{\beta_1}\dots |x_d|^{\beta_d}, |x_1|^{\gamma_1}\dots |x_d|^{\gamma_d})\in A_{p, q}^\alpha.
\]
From the tensor product structure, we can immediately find sufficient conditions. Indeed, we have that
\[
\begin{split}
    |Q|^{\frac{\alpha}{d}-1}&\left(\int_Q(|x_1|^{\beta_1}\dots |x_d|^{\beta_d})^{-\frac{p'}{p}}\right)^{1/p'}\left(\int_Q|x_1|^{\gamma_1}\dots |x_d|^{\gamma_d}\right)^{1/q} \\
    &= \prod_{j=1}^d |I_j|^{\frac{\alpha}{d}-1}\left(\int_{I_j}|x_j|^{-\beta_j\frac{p'}{p}}\right)^{1/p'}\left(\int_{I_j}|x_j|^{\gamma_j}\right)^{1/q} \\
    &\leq [|x_1|^{\beta_1}, |x_1|^{\gamma_1}]_{A_{p, q}^\frac{\alpha}{d}}\dots [|x_d|^{\beta_d}, |x_d|^{\gamma_d}]_{A_{p,q}^\frac{\alpha}{d}}.
\end{split}
\]
And therefore, if the right-hand side is finite, then this pair of monomial weights is in \(A_{p, q}^\alpha\). We already know that the right-hand side is finite precisely when
\begin{equation}\label{eq: tensor product conditions}
\frac{\alpha}{d} \geq \frac{1}{p} - \frac{1}{q},\ \frac{\alpha}{d}- \frac{1}{p}+\frac{1}{q} = \frac{\beta_j}{p} - \frac{\gamma_j}{q},\ \beta_j < p-1\text{ and }\gamma_j > -d,\ \forall j.
\end{equation}
However, the conditions 
\[ 
    \frac{\alpha}{d} - \frac{1}{p} + \frac{1}{q} = \frac{\beta_j}{d} - \frac{\gamma_j}{d}
\] 
won't be necessary in general, and so this does not 
give the best range of parameters. The reason for this is that in the 
\(A_{p, q}^\alpha\) condition we are taking a supremum over cubes, not rectangles. 
So, we don't have complete freedom in choosing the intervals that make up the cube, 
independently of one another. We may choose their locations independently, 
but all the lengths are the same. This way, characterizing this condition will 
require a bit more work, and we do that in the following lemma.

\begin{lemma}\label{lemma: Apq alpha condition for monomial weights}
    Let \(1<p, q<\infty\) and \(\alpha \in \R\). We have
    \[
    (|x_1|^{\beta_1}\dots |x_d|^{\beta_d}, |x_1|^{\gamma_1}\dots |x_d|^{\gamma_d}) \in A_{p, q}^\alpha
    \]
    if and only if the following conditions hold:
    \begin{enumerate}
        \item \( \beta_j < p-1, \gamma_j > -1\), for all \(j\); \label{item: local integrability}
        \item \(\frac{\gamma_j}{q}\leq \frac{\beta_j}{p}\); \label{item: exponent relation}
        \item We have that
        \[
        \alpha + \frac{d}{q} - \frac{d}{p} + \sum_{j=1}^d \frac{\gamma_j}{q}-\frac{\beta_j}{p} = 0. \label{item: dimensional balance}
        \]
    \end{enumerate}
\end{lemma}
\begin{proof}
    We begin by introducing some notation. We write a given cube as \(Q = I_1\times \dots \times I_d\), where \(I_j = [y_j, y_j + l[, l> 0\). Moreover, we put \(a_j = -\beta_j p'/p\), \(b_j = \gamma_j\) and we write \(c_j\) for the center of interval \(I_j\), that is \(c_j = y_j + l/2\). Then,
    \[
    \sup_Q |Q|^{\frac{\alpha}{d}-1}\left(\int_Q (|x_1|^{\beta_1}\dots |x_d|^{\beta_d})^{-\frac{p'}{p}}\right)^{1/p'}\left(\int_Q |x_1|^{\gamma_1}\dots |x_d|^{\gamma_d}\right)^{1/q} 
    \]
    simplifies to
    \[
    \sup_{\substack{y \in \R^d \\ l > 0}} \prod_{j=1}^d |I_j|^{\frac{\alpha}{d}-1}\left(\int_{I_j}|x_j|^{a_j}\right)^{1/p'}\left(\int_{I_j}|x_j|^{b_j}\right)^{1/q} = : X.
    \]
    First, we begin by checking that the conditions given are actually necessary. The condition in item \ref{item: local integrability} is clearly necessary to ensure local integrability as it corresponds to \(a_j, b_j > -1\), for all \(j\). Now we see that condition \ref{item: exponent relation} is necessary. Given \(k \in \{1, \dots, d\}\) we define \(I_j = [-1/2, 1/2[\) for all \(j\neq k\), and \(I_k = [y_k, y_k + 1[\) where \(y_k > 1\). This way, \(I_k\) is a type I interval. So, we have that
    \[
    \prod_{j=1}^d |I_j|^{\frac{\alpha}{d}-1}\left(\int_{I_j}|x_j|^{a_j}\right)^{1/p'}\left(\int_{I_j}|x_j|^{b_j}\right)^{1/q} \sim \left(\int_{I_k}|x_k|^{a_k}\right)^{1/p'}\left(\int_{I_k}|x_k|^{b_k}\right)^{1/q} \sim c_k^{\frac{a_k}{p'} + \frac{b_k}{q}}.
    \]
    If we now let \(y_k \rightarrow +\infty\), then \(c_k \rightarrow +\infty\) and therefore we see that the condition
    \[
    \frac{a_k}{p'} + \frac{b_k}{q}\leq 0
    \]
    is necessary. Now let's check that condition number \ref{item: dimensional balance} is also necessary. We can do this if we choose \(I_j = [-l/2, l/2[\) for all \(j\). In this case, we get
    \[
    \begin{split}
    \prod_{j=1}^d |I_j|^{\frac{\alpha}{d}-1}\left(\int_{I_j}|x_j|^{a_j}\right)^{1/p'}\left(\int_{I_j}|x_j|^{b_j}\right)^{1/q} &\sim \prod_{j=1}^d l^{\frac{\alpha}{d}-1}l^\frac{a_j+1}{p'} l^\frac{b_j + 1}{q} \\
    &\sim l^{\alpha - d + \frac{d}{p'} + \frac{d}{q} + \sum_{j=1}^d\frac{a_j}{p'}+\frac{b_j}{q}}.
    \end{split}
    \]
    Since \(l\) can be any positive number, we see that we must have
    \[
    \alpha - \frac{d}{p} + \frac{d}{q} + \sum_{j=1}^d\frac{a_j}{p'}+\frac{b_j}{q} = 0.
    \]
    Now that we have checked that the conditions are all necessary, let's actually prove that they are also sufficient. 
    
    Given a cube \(Q\) and a set \(A\subseteq \{1,\dots, d\}\) we write \(Q\in I_A\) 
    if \(Q = I_1\times \dots \times I_d\) and for all \(j\), \(I_j \text{ is type I }\iff j \in A\). Now note that
    \[
    X= \sup_{A \subseteq \{1,\dots, d\}} \sup_{\substack{Q \in I_A}}\prod_{j=1}^d |I_j|^{\frac{\alpha}{d}-1}\left(\int_{I_j}|x_j|^{a_j}\right)^{1/p'}\left(\int_{I_j}|x_j|^{b_j}\right)^{1/q}.
    \]
    To prove sufficiency we will show that the right hand side is finite. Fix some \(A\subseteq \{1, \dots, d\}\) and consider cubes \(Q\in I_A\). Suppose we write \(A = \{k_1, \dots, k_m\}\), with \(m\leq d\). Then, the intervals \(I_{k_1}, \dots, I_{k_m}\) are type I and the others are type II. As before, the type II intervals are contained in an interval centered at the origin, so we have that
    \[
    \begin{split}
        &\prod_{j=1}^d |I_j|^{\frac{\alpha}{d}-1}\left(\int_{I_j}|x_j|^{a_j}\right)^{1/p'}\left(\int_{I_j}|x_j|^{b_j}\right)^{1/q} \\
        &\lesssim \prod_{j\in A}l^{\frac{\alpha}{d}-\frac{1}{p} + \frac{1}{q}}|c_j|^{\frac{a_j}{p'} + \frac{b_j}{q}}  \prod_{j\notin A} l^{\frac{\alpha}{d}-\frac{1}{p}+\frac{1}{q} + \frac{a_j}{p'} + \frac{b_j}{q}} \\ 
        &\lesssim l^{m(\frac{\alpha}{d} - \frac{1}{p} + \frac{1}{q})} l^{(d-m)(\frac{\alpha}{d} - \frac{1}{p} + \frac{1}{q})} l^{\sum_{j\notin A}\frac{a_j}{p'}+\frac{b_j}{q}} \prod_{j\in A}|c_j|^{\frac{a_j}{p'} + \frac{b_j}{q}} \\
        &\lesssim l^{\alpha - \frac{d}{p} + \frac{d}{q}}  l^{\sum_{j\notin A}\frac{a_j}{p'}+\frac{b_j}{q}} \prod_{j\in A}|c_j|^{\frac{a_j}{p'} + \frac{b_j}{q}}.
    \end{split}
    \]
    Now, from condition \ref{item: dimensional balance} we know that
    \[
    \alpha - \frac{d}{p} + \frac{d}{q} + \sum_{j\notin A}\frac{a_j}{p'} + \frac{b_j}{q} = - \sum_{j\in A}\frac{a_j}{p'} + \frac{b_j}{q}.
    \]
    This way, we obtain
    \[
    \prod_{j=1}^d |I_j|^{\frac{\alpha}{d}-1}\left(\int_{I_j}|x_j|^{a_j}\right)^{1/p'}\left(\int_{I_j}|x_j|^{b_j}\right)^{1/q} \lesssim \prod_{j\in A} \left(\frac{|c_j|}{l}\right)^{\frac{a_j}{p'} + \frac{b_j}{q}}.
    \]
    If we now recall that the intervals \(I_j\) with \(j\in A\) are type I, then we have that \(l < |c_j|\). This fact, together with condition \ref{item: exponent relation} implies that 
    \[
    \prod_{j\in A} \left(\frac{|c_j|}{l}\right)^{\frac{a_j}{p'} + \frac{b_j}{q}} \leq 1.
    \]
    Thus, 
    \[
    \sup_{Q\in I_A}\prod_{j=1}^d |I_j|^{\frac{\alpha}{d}-1}\left(\int_{I_j}|x_j|^{a_j}\right)^{1/p'}\left(\int_{I_j}|x_j|^{b_j}\right)^{1/q} \lesssim_{A, p, q, \alpha, a, b} 1.
    \]
    In particular, this implies that \(X < \infty\). 
    
\end{proof}
Note that the conditions \ref{eq: tensor product conditions} satisfy this lemma. In fact, from 
\[
\frac{\alpha}{d} \geq \frac{1}{p}-\frac{1}{q}
\]
we get
\[
\frac{\beta_j}{p}-\frac{\gamma_j}{q} = \frac{\alpha}{d}-\frac{1}{p}+\frac{1}{q} \geq 0.
\]
And also, if we add up the condition
\[
\frac{\alpha}{d}-\frac{1}{p}+\frac{1}{q} +\frac{\gamma_j}{q}- \frac{\beta_j}{p} = 0
\] 
over all \(j\), then we get
\[
\alpha - \frac{d}{p}+\frac{d}{q} + \sum_j \frac{\gamma_j}{q}- \frac{\beta_j}{p} = 0.
\]
So as expected, the conditions of the lemma are a strict improvement over the conditions we get simply from the tensor product structure. Another important thing to note here is that from the conditions of the lemma we have
\[
\alpha - \frac{d}{p} + \frac{d}{q} = \sum_j \frac{\beta_j}{p} - \frac{\gamma_j}{q} \geq 0.
\]
It is important that this condition should hold because we already know from Remark \ref{remark: Apq alpha implies sub-sobolev relation} that this condition must hold in general, for any pair of weights in \(A_{p, q}^\alpha\).

Although we have chosen to focus on these three families of weights, there are many other examples of weights that satisfy Muckenhoupt-type conditions. Other examples could be tensor products of the families we have considered, i.e. 
\[
w(x) = |(x_1,\dots, x_k)|^{\gamma_1} \langle (x_{k+1},\dots, x_d)\rangle^{\gamma_2}.
\]
We could also have considered the pairing of different types of weights. That is, similiar arguments allow us to characterize when \((|x|^\beta, \langle x \rangle^\gamma)\in A_{p, q}^\alpha\) and so on.

\subsection{Sparse collections and Carleson families}

We have seen in the previous sections the usefulness of sparse families, for they can give us two weight norm inequalities as 
long as we can control our operators by sparse operators and we have good boundedness properties for the associated maximal operators. 
In later sections, we will see many more examples of operators that we can dominate by sparse operators, 
but in this section, we will look at some examples of sparse sets to get some intuition for how they behave.

The simplest example of a sparse collection is of course a collection \(\S\) of disjoint sets. In that case, we simply take \(G(Q) = Q\), for all \(Q\in \S\). Obviously, this is not the most interesting case, and in practice, the sparse sets we will work with consist of non-disjoint sets. The next simplest example of a sparse set, which is better suited for developing some intuition about these collections, is a dyadic tower. Suppose \(\D\) is a dyadic lattice, \(x\in \R^d\) and define \(\S = \{Q_k(x)\}_{k\in \Z}\). For a given \(k\) we put \(G(Q_k(x)) = Q_k(x) \setminus Q_{k+1}(x)\). Then,
\[
|G(Q)|\geq \left(1-\frac{1}{2^d}\right) |Q|
\]
and clearly the sets \(G(Q)\) are pairwise disjoint for \(Q\in \S\) (see Figure \ref{fig: dyadic_tower}). 
\begin{figure}[ht]
    \centering
    \includegraphics[width=0.25\textwidth]{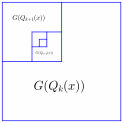}
    \caption{The cube \(Q_k(x)\) and its descendants in the dyadic tower. The good part of \(Q_k(x)\), \(G(Q_k(x))\), is the region that excludes the top left dyadic child of \(Q_k(x)\), and so on.}
    \label{fig: dyadic_tower}
\end{figure}
Another interesting example is the Sierpinski triangle, or more generally, constructions following some kind of fractal nature. 
In the case of the Sierpinski triangle, if we add to \(\S\) all the blue triangles from all the iterations in Figure 
\ref{fig: Sierpinski triangle}, then this forms a \(1/4\)-sparse set, where the good part of a triangle is the white subregion of 
the next iteration. 
\begin{figure}[ht]
    \centering
    \includegraphics[width=0.75\textwidth]{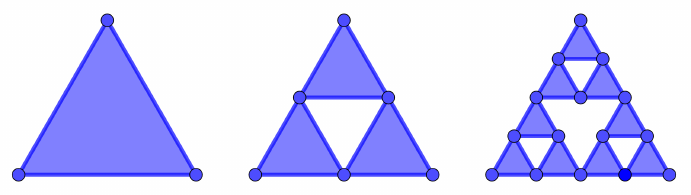}
    \caption{Several iterations of the construction of the Sierpinski triangle.}
    \label{fig: Sierpinski triangle}
\end{figure}
In general, it is not so simple to check if a given family of sets is sparse. Also, there are simple questions about sparse sets that are not at all obvious with the definition that we have given. For example is it the case that, given two sparse families \(\S_1, \S_2\subseteq \D\) of dyadic cubes, the union \(\S_1\cup\S_2\) is sparse? Fortunately, there is an equivalent condition to sparseness which is often much simpler to handle.
\begin{defn}\label{defn: carleson family}
    We say that a family of cubes \(\S\subseteq \D\) is \(\Lambda\)-Carleson\index{Carleson families} if there exists some \(\Lambda > 1\) such that
    \[
    \sum_{\substack{P\subseteq Q \\ P\in \S}} |P|\leq \Lambda |Q|, 
    \]
    for all \(Q\in \D\).
\end{defn}
It is easy to see that an \(\eta\)-sparse family of dyadic cubes must be \(\eta^{-1}\)-Carleson. Indeed, sparseness implies that
\[
\sum_{\substack{P\subseteq Q\\ P\in \S}}|P|\leq \frac{1}{\eta} \sum_{\substack{P\subseteq Q\\ P\in \S}}|G(P)|\leq \frac{1}{\eta} |Q|.
\]
However, it is not as simple to see that Carleson families are also sparse, nevertheless, this is in fact true.
\begin{prop}[Lemma 6.3 from \cite{intuitive_dyadic_calculus}]\label{prop: carleson implies sparse}
    Let \(0 < \eta< 1\), and \(\S\subseteq \D\), where \(\D\) is a given dyadic lattice. Then, \(\S\) is \(\eta\)-sparse if and only if it is \(1/\eta\)-Carleson.
\end{prop}
\begin{proof}
    We have already shown one direction, so we focus now on the other direction. For this purpose suppose that \(\S\) is \(\Lambda\)-Carleson, where \(\Lambda = \eta^{-1}\). Our goal is to show that this implies that \(\S\) is \(\Lambda^{-1}\)-sparse. The challenge here is that the trivial idea doesn't work, in fact, if we define for any given \(Q\in \S\),
    \[
    G(Q) = Q \setminus \bigcup_{\substack{P\subsetneq Q\\ P\in \S}} P,
    \]
    it is true that the sets \(G(Q)\) are disjoint. However, we would then estimate
    \[
    |G(Q)| \geq |Q| - \sum_{\substack{P\subsetneq Q \\ P\in \S}} |P| \geq |Q| - (\Lambda - 1)|Q| = (2-\Lambda)|Q|,
    \]
    which is never above \(|Q|/\Lambda\). The problem here is that we removed more than we needed in the definition of the good parts. We really only need to define a good part contained in 
    \[
    Q\setminus \bigcup_{\substack{P\subsetneq Q\\ P\in \S}} G(P),
    \]
    because then \(G(Q)\) is disjoint from the good parts of the descendants of \(Q\) and, assuming that \(|G(P)| = |P|/\Lambda\), 
    it follows that 
    \[
    \left|Q\setminus \bigcup_{\substack{P\subsetneq Q\\ P\in \S}} G(P)\right| \geq |Q| - \sum_{\substack{P\subsetneq Q\\ P\in \S}} |G(P)| = |Q|-\frac{1}{\Lambda}\sum_{\substack{P\subsetneq Q\\ P\in \S}} |P|\geq |Q| - \frac{\Lambda - 1}{\Lambda}|Q| = \frac{|Q|}{\Lambda}.
    \]
    Therefore we can pick \(G(Q)\) in \(Q\setminus \cup_{P\subsetneq P, P\in \S}G(P)\) such that \(|G(Q)| = |Q|/\Lambda\). 
    But of course, we can only do this if we have already defined the good parts of the descendants of \(Q\) in \(\S\). 
    Thus we see that to argue this way we always need to `push' the initialization of the good parts further and further down the generations. 
    With some care, this can be done by making certain choices at the level of a fixed generation, repeating the construction for later generations and taking the process to the limit.

    For each \(K\geq 0\) we will define good parts of the cubes of the family with respect to the construction that initializes the cubes in generation \(\D_K\). The main idea for this construction will be to use good parts which are built out of the corners of the cubes. Fix \(K\geq 0\) and for each \(Q\in \S \cap \D_K\) put
    \[
    G_K(Q) = x(Q) + \frac{1}{\Lambda^{1/d}}(Q-x(Q)).
    \]
    In other words, \(G_K(Q)\) is the cube with the same corner as \(Q\) but with measure \(|G_K(Q)| = |Q|/\Lambda\). Now take \(Q\in \S \cap \D_{K-1}\). As before we want \(G_K(Q)\) to be contained in
    \[
    Q \setminus \bigcup_{\substack{ P\subsetneq Q\\ P\in \S }} G_K(P)
    \]
    and such that \(|G_K(Q)| = |Q|/\Lambda\). Like in generation \(K\) we build this set out of cubes with the same corner as \(Q\). 
    Note that the function
    \[
    t \mapsto \left|[x(Q)+t(Q-x(Q))] \setminus \bigcup_{\substack{ P\subsetneq Q\\ P\in \S }} G_K(P) \right|
    \]
    is continuous for \(t\in [0, 1]\) and goes from \(0\) when \(t=0\) to 
    \[
    \left|Q \setminus \bigcup_{\substack{ P\subsetneq Q\\ P\in \S }} G_K(P) \right| \geq |Q| - \sum_{\substack{P\subsetneq Q\\ P\in \S}}|G_K(P)| = |Q| - \frac{1}{\Lambda}\sum_{\substack{P\subsetneq Q\\ P\in \S}}|P|\geq |Q| - \frac{\Lambda - 1}{\Lambda}|Q| = \frac{|Q|}{\Lambda},
    \]
    when \(t=1\). Therefore, there is some \(t^*\in [0, 1]\) such that 
    \[
    \left|(x(Q)+t^*(Q-x(Q))) \setminus \bigcup_{\substack{ P\subsetneq Q\\ P\in \S }} G_K(P) \right| = \frac{|Q|}{\Lambda}.
    \]
    We now define (see Figure \ref{fig: good part construction})
    \[
    G_K(Q) =  (x(Q)+t^*(Q-x(Q))) \setminus \bigcup_{\substack{ P\subsetneq Q\\ P\in \S }} G_K(P).
    \]
    \begin{figure}[ht]
        \centering
        \includegraphics[width = 0.3\textwidth]{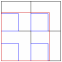}
        \caption{The blue squares represent the initialization in generation \(K\). The set \(G_K(Q)\) is then the red square minus the blue squares.}
        \label{fig: good part construction}
    \end{figure}
    This way the good parts of cubes in \(\D_{K-1}\) do not intersect the good parts of cubes in \(\D_K\) and these good parts have the required measure. Now we simply continue this way back in `time' and we define \(G_K(Q)\) for all \(Q\in \S \cap \D_k,\ k< K\). More precisely, suppose we have already defined \(G_K(Q)\) for all \(Q\in \S \cap (\cup_{m+1\leq k\leq K}\D_k)\), for some \(m\leq K-1\). Then, given \(Q\in \S\cap \D_m\), we define 
    \[
    G_K(Q) = (x(Q) + t^*(Q-x(Q))) \setminus \bigcup_{\substack{P\subsetneq Q \\ P\in \S}} G_K(P),
    \]
    where \(t^*\in [0, 1]\) is such that \(|G_K(Q)| = |Q|/\Lambda\). This way we have constructed good parts to every cube in \(\S \cap (\cup_{k\leq K}\D_k)\) with respect to generation \(K\). We note that we set \(G_K(Q) = \varnothing\) if \(Q \in \cup_{k>K}\D_k\).
    
    The idea now is to let \(K\rightarrow +\infty\). To do this carefully we first show that for a given cube \(Q\in \S \cap (\cup_{k\leq K}\D_k)\) we have that 
    \begin{equation}\label{eq: chained good parts}
         \bigcup_{\substack{P\subseteq Q \\ P\in \S}}G_K(P) \subseteq \bigcup_{\substack{P\subseteq Q \\ P\in \S}}G_{K+1}(P).
    \end{equation}
    For convenience, we write 
    \[
    E_K(Q) = \bigcup_{\substack{P\subseteq Q \\ P\in \S}}G_K(P). 
    \]
    To start with assume that \(Q\in \S \cap \D_K\). Then, we know that 
    \[
    G_{K+1}(Q) = (x(Q) + t^*(Q-x(Q))) \setminus \bigcup_{\substack{P\subsetneq Q\\ P\in \S}}G_{K+1}(P),
    \]
    where \(t^*\) is chosen so that \(|G_{K+1}(Q)| = |Q|/\Lambda\). Note that, if \(t^* < \Lambda^{-1/d}\), then
    \[
    |G_{K+1}(Q)| \leq (t^*)^d |Q| < \frac{1}{\Lambda}|Q|,
    \]
    which is absurd. This shows that \(t^* \geq \Lambda^{-1/d}\). From this we deduce that
    \[
    \begin{split}
    E_K(Q) &= G_K(Q) = x(Q) + \Lambda^{-1/d}(Q-x(Q)) \subseteq x(Q) + t^*(Q-x(Q)) \\
    &= [x(Q) + t^*(Q-x(Q))] \setminus \bigcup_{\substack{P\subsetneq Q\\ P\in \S}}G_{K+1}(P) \cup \bigcup_{\substack{P\subsetneq Q \\ P\in \S}} G_{K+1}(P)\\
    &= G_{K+1}(Q) \cup \bigcup_{\substack{P\subsetneq Q \\ P\in \S}} G_{K+1}(P) = E_{K+1}(Q).
    \end{split}
    \]
    Now suppose, by induction that the claim holds for all cubes \(Q \in \cup_{m+1\leq k \leq K} \S \cap \D_k\) and let's prove it holds for \(\S \cap \D_m\). Take \(Q\in \S \cap \D_m\), then there are \(t, t'\in [0, 1]\) such that
    \[
    E_K(Q) = \bigcup_{\substack{P\subseteq Q \\ P\in \S}} G_K(P) = G_K(Q) \cup \bigcup_{\substack{P\subsetneq Q \\ P\in \S}}G_K(P) = x(Q) + t(Q-x(Q)),
    \]
    and the same works also for \(E_{K+1}(Q)\), i.e. 
    \[
    E_{K+1}(Q) = x(Q) + t'(Q-x(Q)).
    \]
    Moreover, by definition,
    \[
    \left|[x(Q) + t(Q-x(Q))] \setminus \bigcup_{\substack{P\subsetneq Q \\ P\in \S}}G_K(P)\right| = \left| [x(Q) + t'(Q-x(Q))] \setminus \bigcup_{\substack{P\subsetneq Q \\ P\in \S}}G_{K+1}(P)\right|.
    \]
    To prove that \(E_K(Q) \subseteq E_{K+1}(Q)\) it is enough to show that \(t\leq t'\), and this follows if we show that
    \[
    \bigcup_{\substack{P\subsetneq Q \\ P\in \S}}G_K(P) \subseteq \bigcup_{\substack{P\subsetneq Q \\ P\in \S}}G_{K+1}(P).
    \]
    Take an element \(x\) of the left hand side. Then, \(x\) is in some \(G_K(P)\) for some \(P\subsetneq Q, P\in \S\), and \(P\in \cup_{m+1\leq k \leq K}\S \cap \D_k\). Therefore, we can apply the induction hypothesis to say that \(x\in E_K(P)\subseteq E_{K+1}(P)\) and thus \(x\in E_{K+1}(P)\). In turn, this implies that \(x\in G_{K+1}(P')\) for some \(P'\subseteq P, P'\in \S\), which shows that \(x\) is an element of the right hand side. This proves that indeed,
    \[
    E_K(Q) \subseteq E_{K+1}(Q).
    \]
    Now, given a cube \(Q\in \S \cap \D_k\) we define
    \[
    E(Q) = \bigcup_{K\geq k} E_K(Q)\subseteq Q.
    \]
    Once these sets are defined we can finally say what the good parts are. For a given \(Q\in \S\) we put
    \[
    G(Q) = E(Q) \setminus \bigcup_{\substack{P\subsetneq Q \\ P \in \S}} E(P).
    \]
    Then, the sets \(G(Q), Q\in \S\), are pairwise disjoint. To see that they have the required measure, we begin by noting that 
    \[
    G(Q) = \left[\bigcup_K E_K(Q)\right]\setminus \left[ \bigcup_K \bigcup_{\substack{P\subsetneq Q \\ P\in \S}} E_K(P)\right] 
    \]
    is the intersection of the decreasing (in \(K\)) family
    \[
    \left[ \bigcup_L E_L(Q)\right]\setminus \left[ \bigcup_{\substack{P\subsetneq Q \\ P\in \S}} E_K(P)\right]. 
    \]
    This means that
    \[
    |G(Q)| = \lim_K \left| \bigcup_L E_L(Q) \setminus \bigcup_{\substack{P\subsetneq Q \\ P\in \S}} E_K(P)\right| \geq \lim_K \left| E_K(Q) \setminus \bigcup_{\substack{P\subsetneq Q \\ P\in \S}} E_K(P)\right| = \lim_K |G_K(Q)| = \frac{|Q|}{\Lambda}.
    \]
    Thus, we have shown that \(\S\) is \(\Lambda^{-1}\)-sparse.
\end{proof}

Using this characterization of sparse sets it becomes very simple to see that the union of sparse sets is sparse. Indeed, suppose that \(\S_1, \S_2\subseteq \D\) are sparse sets with constant respectively \(\eta_1\) and \(\eta_2\). Then, they are Carleson families with constants \(\eta_1^{-1}, \eta_2^{-1}\). Therefore, given a cube \(Q\in \D\),
\[
\sum_{\substack{P\subseteq Q \\ P\in \S_1 \cup \S_2}}|P|\leq \sum_{\substack{P\subseteq Q \\ P\in \S_1}}|P| + \sum_{\substack{P\subseteq Q \\ P\in \S_2}}|P|\leq (\eta_1^{-1} + \eta_2^{-1})|Q|,
\]
which in turn implies that \(\S_1\cup\S_2\) is sparse with constant
\[
(\eta_1^{-1} + \eta_2^{-1})^{-1} = \frac{\eta_1\eta_2}{\eta_1+\eta_2}.
\]
Moreover, we can also use the equivalence with Carleson families to find simple criteria for a collection to be sparse. Suppose that \(\S\subseteq \D\). For a given \(Q\in \D\) let
\[
N_k(Q) = \#\{P\in \Des_k(Q): P\in \S\}.
\]
In other words, \(N_k(Q)\) measures the number of descendants of order \(k\) of \(Q\) which are in the collection \(\S\). Then, we see that 
\[
\sum_{\substack{P\subseteq Q \\ P\in \S}}|P| = \sum_{k\geq 0}\sum_{\substack{P\in \Des_k(Q) \\ P\in \S}}|P| = \sum_{k\geq 0}\sum_{\substack{P\in \Des_k(Q) \\ P\in \S}} 2^{-kd}|Q| = |Q|\sum_{k\geq 0}\frac{N_k(Q)}{2^{kd}}.
\]
So, it follows that \(\S\) is sparse if and only if 
\begin{equation}\label{eq: condition for sparseness}
\sup_{Q\in \D} \sum_{k\geq 0}\frac{N_k(Q)}{2^{kd}} < \infty.
\end{equation}
This condition gives a simple criterion for constructing sparse sets. 
Since for any \(Q\in \D\), \(\#\Des_k(Q) = 2^{kd}\), we see that to construct a sparse set we simply have to choose a small enough fraction of the cubes in each generation so that the sum of these fractions converges. Furthermore, if we consider the one-dimensional case the situation becomes even simpler. 

Suppose that \(\D\) is a dyadic lattice in \(\R\) and consider a fixed dyadic interval \(I\). We are interested in characterizing what collections \(\S\subseteq \Des(I)\) are sparse. First, we label each interval which is a descendant of \(I\) in a particular way. We start by labeling \(I\) as \(I_1\). Then, the left child of \(I\) is labeled as \(I_2\) and the right child of \(I\) is labeled as \(I_3\). Continuing in this way, the intervals in \(\Des_k(I)\) are labeled as
\[
I_{2^k}, I_{2^k + 1},\dots, I_{2^{k+1}-1},
\]
from left to right (see Figure \ref{fig: dyadic labeling}). 
\begin{figure}[ht]
    \centering
    \includegraphics[width=0.3\textwidth]{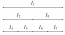}
    \caption{The labeling of the descendants of \(I\) up to generation 2.}
    \label{fig: dyadic labeling}
\end{figure}
This labeling has a few interesting properties. If \(I_n\) is in \(\Des_k(I)\), then \(2^k\leq n \leq 2^{k+1}-1\) and by definition 
\[
I_n = [x(I) + (n-2^k)\frac{l(I)}{2^k}, x(I) + (n-2^k + 1)\frac{l(I)}{2^k}[.
\]
But then, \(2^{k+1}\leq 2n \leq 2^{k+2}-2\) so \(I_{2n}\in \Des_{k+1}(I)\) and 
\[
\begin{split}
I_{2n} &= [x(I) + (2n - 2^{k+1})\frac{l(I)}{2^{k+1}}, x(I) + (2n-2^{k+1}+1)\frac{l(I)}{2^{k+1}}[ \\
&= [x(I)+(n-2^k)\frac{l(I)}{2^k}, x(I) + (n-2^k+\frac{1}{2})\frac{l(I)}{2^k}[\\
&=[x(I_n), x(I_n) + \frac{l(I_n)}{2}[,
\end{split}
\]
which is exactly the left dyadic child of \(I_n\). In other words, this shows that the dyadic children of \(I_n\) are \(I_{2n}\) and \(I_{2n+1}\), from left to right. Now note that the binary representation of the number \(2n\) is exactly the binary representation of \(n\) with a zero added to the right. And, the binary representation of \(2n+1\) is the binary representation of \(n\) with a 1 added to the right. This has the implication of giving us a simple way of knowing the ancestors of a given interval \(I_n\) if we just think of adding a \(0\) to go to the left in the next generation and adding a 1 to go to the right in the next generation. More precisely, we consider the binary representation of \(n\), say \((b_kb_{k-1}\dots b_2 b_1b_0)_2\), and this gives us the sequence of ancestors, where the parent is the interval associated with the number \((b_k b_{k-1}\dots b_2b_1)_2\), and then the parent of the parent is the interval associated to the number \((b_k b_{k-1}\dots b_2)_{2}\), and so on. Another way to think about this is to see the binary representation of a number \(n\) as giving us the `path' we need to take from \(I_1\) to \(I_n\). For example, the number \(21\) has the binary expansion \(10101\), and this means that we start at \(I_1\), the next digit is a \(0\), so we take the left child of \(I_1\), which is \(I_2\). The next digit is 1, so we take the right child of \(I_2\), which is \(I_5\). The next digit is a 0, so we take the left child of \(I_5\) which is \(I_{10}\), and finally, the last digit is a 1, so we go with the right child of \(I_{10}\) and we arrive at \(I_{21}\). This way of understanding the dyadic structure inside \(I\) is interesting because we can associate to each collection of descendants \(\S\), a certain set of integers \(A\subseteq \N_1\) such that
\[
\S = \{I_n: n\in A\}.
\]
From \eqref{eq: condition for sparseness} we know that the set \(\S\) is sparse if and only if 
\[
\sup_{Q\in \D} \sum_{k\geq 0}\frac{N_k(Q)}{2^k}<\infty. 
\]
First note that, since \(\S\subseteq \Des(I)\) we can reduce the supremum from \(Q\in \D\) to \(Q\in \Des(I)\). Now, if we take some \(n\in \N_1\), then
\[
\sum_{k\geq 0} \frac{N_k(I_n)}{2^k} = \sum_{k\geq 0}\frac{\#\{m\in A: 2^kn\leq m < 2^k(n+1)\}}{2^k}
\]
because the descendants of order \(k\) of \(I_n\) are labeled
\[
I_{2^kn}, I_{2^kn + 1},\dots, I_{2^kn + 2^k-1}.
\]
In particular, the set of integers that correspond to the descendants of \(I_n\) is exactly
\[
    \begin{split}
        D(n) &= \{n, 2n, 2n+1, 4n, 4n+1, 4n+2, 4n+3,\dots\} \\
        &= \bigcup_{k\geq 0}\{2^kn, 2^kn+1, \dots, 2^kn + 2^k-1\}.
    \end{split}
\]
Note that the sets in this union are pairwise disjoint, so if \(m\in D(n)\) there is 
a unique \(l\geq 0\) such that \(2^ln\leq m < 2^l(n+1)\).
Continuing this way, we see that
\[
\begin{split}
    \sum_{k\geq 0}\frac{N_k(I_n)}{2^k} &= \sum_{k\geq 0}\frac{1}{2^k}\sum_{m\in A} 
    \1_{2^kn\leq m<2^k(n+1)} = \sum_{m\in A}\sum_{k\geq 0}\frac{1}{2^k}\1_{2^kn\leq 
    m<2^k(n+1)}.
\end{split}
\]
Let's look carefully at the sum in \(k\). If \(m\notin D(n)\), then \(\1_{2^kn\leq m <
2^k(n+1)} = 0\), for all \(k\), and so the inner sum will be zero. If \(m\in D(n)\), then 
there is a unique \(l\geq 0\) such that \(2^ln\leq m < 2^l(n+1)\). Therefore,
\[ 
    \sum_{k\geq 0}\frac{1}{2^k}\1_{2^kn\leq m < 2^k(n+1)} = \frac{1}{2^l} \sim 
    \frac{n}{m}.
\] 
This shows that,
\[ 
    \sum_{k\geq 0}\frac{N_k(I_n)}{2^k} \sim 
    \sum_{m\in A}\frac{n}{m}\1_{m\in D(n)} = n\sum_{m\in A\cap D(n)}\frac{1}{m}.
\] 
Thus, \(\S\) is sparse if and only if 
\begin{equation}\label{eq: sparse set of integers}
\sup_{n\geq 1}\left(n\sum_{m\in A\cap D(n)}\frac{1}{m}\right)<\infty. 
\end{equation}
This gives a completely number-theoretic characterization of what subcollections of \(\Des(I)\) are sparse. In other words, we can define a dyadically sparse set of positive integers\index{dyadically sparse set of integers} \(A\subseteq \N_1\) to be a set such that \eqref{eq: sparse set of integers} holds. One then has a one-to-one correspondence between the concept of dyadically sparse sets of integers and sparse subcollections of \(\Des(I)\).

\subsection{A simple example: the Hardy operator}\label{sec: sparse domination for Hardy}

Above, we were able to control the fractional maximal operator with a sparse operator. 
To do this we argued using an iterative 
scheme, where we identified parts of space where we could control the maximal 
operator and we showed that the remaining parts 
had smaller and smaller measure. 
However, in some cases, it is possible to obtain sparse domination by much simpler methods. In this section, we give an elementary argument 
to show that we can control both the Hardy operator and its transpose by sparse operators.

Given \(f\in L^1_\text{loc}(\R)\) we define the Hardy operator\index{Hardy operator} by
\[
\mcal{H} f(x) = \frac{1}{x}\int_0^x f(t)dt.
\]
The Hardy operator is bounded on \(L^p(\R), 1<p<\infty\), a fact which is also known as Hardy's inequality. However, we will 
show here a simple argument to control \(\mcal{H}\) by a sparse operator, which directly implies the Hardy inequality by Theorem 
\ref{thm: weighted estimate for sparse operators}, and not only that but also implies two-weight norm inequalities for \(\mcal{H}\).

Define \(\S = \{[0, 2^{k}[: k\in \Z\} \cup \{[-2^{k}, 0[: k\in \Z\}\). Note that the family \(\S\) is a sparse collection of intervals. Indeed, if we set \(G([0, 2^k[) = [2^{k-1}, 2^k[\) and \(G([-2^k, 0[) = [-2^k, -2^{k-1}[\), then for any \(I\in \S\) we have that \(|G(I)|\geq |I|/2\) and the sets \(G(I)\) are pairwise disjoint. Now let \(x>0\) and \(k\in \Z\) such that \(2^{k-1}\leq x < 2^k\). Then, we have that
\[
|\mcal{H}f(x)| \leq  \frac{1}{x}\int_0^x |f| \leq 2 \dashint_{[0, 2^k[}|f| \leq 2\sum_{I\in \S}\chi_I(x) \dashint_I|f|.
\]
We can prove this estimate also for \(x<0\) in an analogous way. Therefore we can dominate the Hardy operator with a sparse operator. As an immediate consequence, we obtain the following proposition.
\begin{prop}\label{prop: two weight estimate for hardy op}
    Suppose \(1<p\leq q <\infty\), \((v,w)\in A_{p,q}^0\) and \(v^{-p'/p},w\in A_\infty\), then
    \[
    \|\mcal{H}f\|_{L^q(\R, w)}\lesssim \|f\|_{L^p(\R, v)},\ \forall f \in L^p(\R, v).
    \]
\end{prop}
\begin{remark}
    Several authors, including Tomaselli, Talenti, Artola and Muckenhoupt, have investigated two-weight norm inequalities for the Hardy operator. Moreover,
    the pairs of weights that satisfy such an estimate have been completely characterized. See for instance \cite{MuckenhouptHardy} and references therein.
\end{remark}
As a particular case, we recover the Hardy inequality when \(v = w = 1\). Note that if \((v, w) \in A_{p, q}^0\), 
then by remark \ref{remark: Apq alpha implies sub-sobolev relation},
\[ 
    0 \geq \frac{1}{p} - \frac{1}{q},
\] 
which implies that \(p \geq q\). Seeing as we need \(p\leq q\), it follows that Proposition \ref{prop: two weight estimate for hardy op} actually only holds 
in the case \(p = q\). 
Now, the same can be done with the transpose of the Hardy operator. We have that
\[
\begin{split}
\int_\R \mcal{H}f(x) g(x)dx &= \int_0^{+\infty} \mcal{H}f(x) g(x)dx + \int_{-\infty}^0 \mcal{H}f(x) g(x) dx \\
&= \int_0^{+\infty} \frac{g(x)}{x}\int_0^x f(t)dt dx - \int_{-\infty}^0\frac{g(x)}{x}\int_x^0 f(t)dt dx\\
&= \int_0^{+\infty}f(t) \int_t^{+\infty}\frac{g(x)}{x}dx dt - \int_{-\infty}^0 f(t) \int_{-\infty}^t \frac{g(x)}{x}dx dt.
\end{split}
\]
Therefore, the transpose of the Hardy operator is 
\[
\mcal{H}^t f(x) = 
\begin{dcases}
    \int_x^{+\infty}\frac{f(t)}{t}dt,& \text{ if }x > 0,\\
    -\int_{-\infty}^x \frac{f(t)}{t}dt,& \text{ if }x<0.
\end{dcases}
\]
\begin{lemma}\label{lemma: sparse domination for the transpose hardy op}
    Let \(\S = \{[0, 2^{k}[: k\in \Z\} \cup \{[-2^{k}, 0[: k\in \Z\}\), then
    \[
    |\mcal{H}^tf(x)|\leq 2 \sum_{I\in \S}\chi_I(x) \dashint_I|f|.
    \]
\end{lemma}
\begin{proof}
    The cases \(x>0\) and \(x<0\) are very similar, so we focus only on the case \(x>0\). Take \(k_0\in \Z\) such that \(2^{k_0}\leq x < 2^{k_0+1}\). Then, we have that
    \[
    |\mcal{H}^tf(x)|\leq \int_x^\infty \frac{|f(t)|}{t}dt \leq \sum_{k\geq k_0}\int_{2^k}^{2^{k+1}}\frac{|f(t)|}{t}dt \leq \sum_{k\geq k_0}\frac{1}{2^k}\int_0^{2^{k+1}}|f| \leq 2 \sum_{I\in \S}\chi_I(x) \dashint_I|f|.
    \]
\end{proof}

We may also conclude from this lemma that Proposition \ref{prop: two weight estimate for hardy op} also works for \(\mcal{H}^t\).

\subsection{General pointwise sparse domination}\label{sec: general sparse domination}

In this section, we want to revisit the idea of sparse domination and apply it more 
generally to operators other than the fractional maximal operator. 
The goal here is to prove the Lerner-Ombrosi theorem (Theorem 1.1 from 
\cite{LernerOmbrosi}). To do this we must consider what parts of the 
argument given in Section \ref{section: controlling the fractional maximal operator} 
generalize easily and what parts have to be changed. In general, 
the argument has two distinct parts: a local part, where we obtain the desired 
control inside some cube \(Q_0\); and a global part, where we argue that for 
\(x\in \R^d\setminus Q_0\) we still have control. The global part was simple 
because we were dealing with the fractional maximal operator and in fact the value at 
\(x\in \R^d\setminus Q_0\) was already an average over an ancestor of \(Q_0\), 
so we could simply add it to the family. This argument clearly depends on the 
specific nature of the fractional maximal operator and will have to be modified in 
the present case. Fortunately this can be easily done. 
The main idea is that we can actually run the same local argument over any cube 
that contains the support of \(f\). By then choosing the cubes over which we 
run the local argument carefully we can combine all those local sparse families 
into a single global sparse family. Before going through all the details in the 
next lemma, it is convenient to introduce a notation for a dilation of a cube. 
Given a cube \(Q\) we write \(D_\lambda Q\) for the cube which has the 
same center as \(Q\) but with \(l(D_\lambda Q) = \lambda l(Q)\).
\begin{lemma}[Lemma 2.1 from \cite{LernerOmbrosi}]\label{lemma: local to global sparse bound}
    Let \(\alpha \in \R\), \(1\leq s < \infty\), let \(T\) be some operator and let 
    \(f\in L^1(\R^d)\) be a function with compact support. Suppose that, for any 
    cube \(Q\), there is some \(1/2\)-sparse family \(\mcal{F}_Q = \mcal{F}_Q(f)\) of subcubes\footnote{These don't have to be dyadic descendants of \(Q\).} of \(Q\) such that
    \[
    |T(f\chi_{Q^*})(x)|\leq C \sum_{R\in \mcal{F}_Q}\chi_R(x) |R^*|^\frac{\alpha}{d}\left(\dashint_{R^*}|f|^s\right)^{1/s}\text{ for a.e. }x\in Q,
    \]
    where \(Q^* = D_\lambda Q, R^* = D_\lambda R\) and \(\lambda \geq 3\). Then, there is some \(1/(2\lambda^d)\)-sparse 
    family of cubes \(\S = \S(f)\) such that
    \[
    |Tf(x)|\leq C\sum_{Q\in \S}\chi_Q(x) |Q|^\frac{\alpha}{d}\left(\dashint_Q |f|^s\right)^{1/s}\text{ for a.e. }x\in \R^d.
    \]
\end{lemma}
\begin{proof}
    Fix \(f\in L^1(\R^d)\) with compact support. Let \(Q_0\) be any cube which contains the support of \(f\). Now, the cube \(D_3Q_0\) will consist of \(3^d\) copies of \(Q_0\). We label those cubes that cover \(D_3 Q_0\setminus Q_0\) as \(Q_1, Q_2,\dots, Q_{3^d-1}\). Note that for any of these cubes, \(Q_0\subseteq D_3 Q_j\). We then continue in this way, that is, we label the cubes that cover \(D_{3^2}Q_0\setminus D_3Q_0\) and are congruent to \(D_3 Q_0\) as \(Q_{3^d},\dots, Q_{2\cdot 3^d-1}\), and again note that the three-fold expansions of these cubes contain \(Q_0\). We thus obtain with this process a sequence of disjoint cubes \(\{Q_j\}_j\) such that
    \[
    \bigcup_{j\geq 0} Q_j = \R^d\text{ and }\supp(f)\subseteq D_3Q_j,\ \forall j.
    \]
    Now we apply the assumptions to each of these cubes. So, for each \(j\) there is a \(1/2\)-sparse family \(\mcal{F}_j\) of subcubes of \(Q_j\) such that
    \[
    |T(f\chi_{Q_j^*})(x)|\leq C\sum_{R\in \mcal{F}_j}\chi_R(x)|R^*|^{\frac{\alpha}{d}} \left(\dashint_{R^*}|f|^s\right)^{1/s}\text{ for a.e. }x\in Q_j.
    \]
    Note that, since \(\supp(f)\subseteq D_3 Q_j\) and \(\lambda \geq 3\), then \(f\chi_{Q_j^*} = f\). Now put \(\mcal{F} = \cup_j \mcal{F}_j\). Since all the cubes \(Q_j\) are pairwise disjoint, then it follows that \(\mcal{F}\) is still a \(1/2\)-sparse family of cubes and we have that
    \[
    |Tf(x)|\leq C \sum_{R\in \mcal{F}}\chi_R(x)  |R^*|^{\frac{\alpha}{d}}\left(\dashint_{R^*}|f|^s\right)^{1/s} \text{ for a.e. }x\in \R^d.
    \]
    Finally we define \(\S = \{R^*: R\in \mcal{F}\}\). For each \(Q\in \S\), if \(Q = R^*\), we define \(G(Q) = G(R)\). Then the sets \(\{G(Q)\}\) are pairwise disjoint and 
    \[
    |G(Q)| = |G(R)| \geq \frac{1}{2}|R| = \frac{1}{2\lambda^d}|Q|,
    \]
    which shows that \(\S\) is a \(1/(2\lambda^d)\)-sparse family of cubes. Moreover we clearly have the desired estimate
    \[
    |Tf(x)|\leq C \sum_{Q\in \S}\chi_Q(x) |Q|^{\frac{\alpha}{d}}\left(\dashint_{Q}|f|^s\right)^{1/s}\text{ for a.e. }x\in \R^d.
    \]
\end{proof}

Now that we have seen how to go from local domination to global sparse domination we can focus on proving local control. 
The basic idea is the same as the argument in Section \ref{section: controlling the fractional maximal operator}. 
In particular, we want to decompose the cube we start with into a good part and a bad part such that we have the desired 
control in the good part and the bad part has small measure. 
To argue that the bad part has a small measure we used the weak type estimate for the fractional maximal operator, 
so we will want our current operator to satisfy such an estimate. So as to avoid making the result too specific we 
will assume that we are working with a sublinear operator \(T\) which is weak type \((p, q)\), for some \(p,q\in [1, \infty[\). 
Now, in Section \ref{section: controlling the fractional maximal operator}, it was crucial that \(M^\D_\alpha f(x) = M^\D_\alpha(f\chi_Q)(x)\) for maximal cubes, 
so that the weak type estimate gave us the desired control on the measure of the bad part. 
In the present case, we don't have enough information about \(T\) to determine the validity of such a property so we 
instead include the characteristic function in the definition of the bad part. More precisely, given a cube \(Q\), we define
\[
B(Q) = \left\{x\in Q: |T(f\chi_{Q^*})(x)|> C |Q^*|^\frac{\alpha}{d}\left(\dashint_{Q^*} |f|^s\right)^{1/s}\right\}. 
\]
Let's check that this has small measure for \(C\) large enough. We have that
\[
\begin{split}
|B(Q)| &\leq \|T\|_{L^{p}\rightarrow L^{q, \infty}}^q C^{-q} |Q^*|^{\frac{q}{s}-q\frac{\alpha}{d}}\left(\int_{Q^*}|f|^s\right)^{-q/s}\|f\chi_{Q^*}\|_{L^p}^q\\
& = \|T\|_{L^{p}\rightarrow L^{q,\infty}}^q C^{-q} |Q^*|^{\frac{q}{s}-q\frac{\alpha}{d}}\left(\int_{Q^*}|f|^s\right)^{-q/s}\left(\int_{Q^*}|f|^p\right)^{q/p}\\
&\leq \|T\|_{L^{p}\rightarrow L^{q, \infty}}^q C^{-q} |Q^*|^{\frac{q}{s}-q\frac{\alpha}{d}}\left(\int_{Q^*}|f|^s\right)^{-q/s}\left(\int_{Q^*}|f|^s\right)^{q/s}|Q^*|^{(1-p/s)q/p}\\
&\leq \|T\|_{L^{p}\rightarrow L^{q, \infty}}^q C^{-q} |Q^*|^{q(1/p-\alpha/d)},
\end{split}
\]
where we used Hölder's inequality and because of this we have to assume that \(p\leq s\). For the argument to work we need the exponent of \(|Q^*|\) to be 1. This is so that we can properly add the measures coming from diferent cubes in the next iterations. For this reason we assume that 
\[
\frac{1}{q}+\frac{\alpha}{d} = \frac{1}{p}.
\]
Outside of \(B(Q)\) we have the estimate
\[
|T(f\chi_{Q^*})(x)|\leq C |Q^*|^\frac{\alpha}{d}\left(\dashint_{Q^*}|f|^s\right)^{1/s},
\]
so we simply add the cube \(Q\) to the family \(\mcal{F}_Q\) and the desired estimate holds in \(G(Q)\). Now we want to iterate the argument, and for this the first step is to write \(B(Q)\) as a union of cubes. However, this is not as straightforward as before since for the fractional maximal operator it followed almost immediately from the definition of \(B(Q)\). The idea here is to use a local Calderón-Zygmund decomposition of the function \(\chi_{B(Q)}\) in order to approximate a decomposition of \(B(Q)\) into disjoint cubes. So, applying Proposition \ref{local CZ decomposition} to the function \(\chi_{B(Q)}\) with height 
\[
\lambda'\geq \dashint_Q\chi_{B(Q)} = \frac{|B(Q)|}{|Q|},
\]
we obtain disjoint cubes \(\{Q_j\}\subseteq \Des(Q)\) such that:
\begin{itemize}
    \item \(\chi_{B(Q)}(x)\leq \lambda'\) for a.e. \(x\in Q\setminus \cup_j Q_j\);
    \item \(\lambda'<\dashint_{Q_j}\chi_{B(Q)} \leq 2^d \lambda'\);
    \item \(|\cup_j Q_j|\leq (\lambda')^{-1}\|\chi_{B(Q)}\|_{L^1(Q)} = |B(Q)|/\lambda'\).
\end{itemize}
For the first property to be non-trivial we need to choose \(\lambda'<1\). Of course, such a choice is 
possible because if \(C\) is large enough, then \(|B(Q)|<|Q|\). In that case, the first property tells 
us that \(|B(Q)\setminus \Omega(Q)|=0\) where \(\Omega(Q) = \cup_j Q_j\). In other words, this means 
that \(\Omega(Q)\) `almost' contains \(B(Q)\), in the sense that, the set of points that are in \(B(Q)\) 
but not in \(\Omega(Q)\) has measure zero. Since we expect the measure of \(\Omega(Q)\) to be larger than 
the measure of \(B(Q)\), we need to be careful that the measure of \(\Omega(Q)\)
is not too big, because \(\Omega(Q)\) represents in some sense the actual bad 
part of the cube. From the third property we get \(|\Omega(Q)|\leq (\lambda')^{-1}|B(Q)|\), so as long 
as \(|B(Q)|\) has small measure, \(\Omega(Q)\) will have small measure too. Finally, the second property 
gives us some control of the bad part of each new cube. Indeed, \(\lambda'|Q_j|<|B(Q)\cap Q_j|\leq 2^d \lambda'|Q_j|\). 
For reasons that we will see later, we will want to have \(|B(Q)\cap Q_j|\) be a sufficiently small fraction of \(|Q_j|\). 
For example, it will be enough to have \(|B(Q)\cap Q_j|\leq |Q_j|/2\) and so we choose \(\lambda' = 2^{-d-1}\). 
Moreover, it will also be convenient to have \(|\Omega(Q)|\leq |Q|/2\) and so we choose \(C\) large enough so 
that \(|B(Q)|\leq |Q|/2^{d+2}\).

Now, almost every point \(x\in Q\setminus \Omega(Q)\) is outside \(B(Q)\) and therefore we have
\[
|T(f\chi_{Q^*})(x)|\leq C |Q^*|^\frac{\alpha}{d}\left(\dashint_{Q^*}|f|^s\right)^{1/s}\text{ for a.e. }x\in Q\setminus \Omega(Q).
\]
This means that we can write
\[
|T(f\chi_{Q^*})(x)| = |T(f\chi_{Q^*})(x)|\chi_{Q\setminus \Omega(Q)}(x) + \sum_j |T(f\chi_{Q^*})(x)|\chi_{Q_j}(x),
\]
and we already have the desired estimate in the first term. The idea now is to iterate the argument, however, this presents a technical problem. Note that in the sum above we are applying our operator to \(f\chi_{Q^*}\) whereas what we actually need for the iteration argument to work is to apply it to \(f\chi_{Q_j^*}\). Here, we use the sublinearity of \(T\) to write
\[
|T(f\chi_{Q^*})(x)|\chi_{Q_j}(x) \leq |T(f\chi_{Q^*\setminus Q_j^*})(x)|\chi_{Q_j}(x) + |T(f\chi_{Q_j^*})(x)|\chi_{Q_j}(x).
\]
The last term is what we need for the next iteration of the argument, which means that we need some way of controlling the first term. To do this we will have to add an assumption about our operator \(T\). The idea is to introduce a maximal operator that controls the oscillations of the truncations of \(T\). More precisely, we define
\[
M_{T,\lambda}f(x) = \sup_{Q\ni x}\underset{x', x''\in Q}{\text{ess sup }}|T(f\chi_{\R^d\setminus Q^*})(x')-T(f\chi_{\R^d\setminus Q^*})(x'')|,
\]
where \(Q^* = D_\lambda Q\). It turns out that we can carry through with the argument if we also assume a weak type estimate for \(M_{T,\lambda}\).

\begin{thm}[Lerner-Ombrosi\index{Lerner-Ombrosi theorem}]\label{thm: lerner-ombrosi}
Let \(1\leq p_0, p_1, q_0, q_1<\infty\) and \(s\geq \max\{p_0, p_1\}\) be such that
\[
\frac{1}{q_0}+\frac{\alpha}{d} = \frac{1}{p_0}\text{ and }\frac{1}{q_1}+\frac{\alpha}{d} = \frac{1}{p_1},
\]
where \(0\leq \alpha < d/s\). Suppose that \(T\) is a sublinear operator that satisfies a weak type \((p_0, q_0)\) estimate and suppose also that the associated operator \(M_{T,\lambda}\) satisfies a weak type \((p_1, q_1)\) estimate, where \(\lambda \geq 3\). Then, for each \(f\in L^s(\R^d)\) with compact support, there is a \(1/(2\lambda^d)\)-sparse family of cubes \(\S = \S(f)\) such that
\[
|Tf(x)|\lesssim (\|T\|_{L^{p_0}\rightarrow L^{q_0,\infty}} + \|M_{T,\lambda}\|_{L^{p_1}\rightarrow L^{q_1,\infty}})\sum_{Q\in \S}\chi_Q(x) |Q|^\frac{\alpha}{d}\left(\dashint_Q |f|^s\right)^{1/s}\text{ for a.e. }x\in \R^d,
\]
where the implicit constant depends on \(d, \lambda, \alpha, s, q_0\) and \(q_1\).
\end{thm}
\begin{proof}
    Let's now run through the argument with care. First we make a small change of notation. We've seen above that the `obvious' choice for the bad part of a cube is not really what we should think of as the bad part of a cube, but rather we should think of the region coming from the Calderón-Zygmund decomposition as the bad part, so we will change \(B\) with \(\Omega\) in the notation. Also, we will have to be more careful when defining \(\Omega\) because we want to control not only \(T\) but also \(M_{T,\lambda}\). And actually, as we will see, it is a good idea to control also the maximal function
    \[
    M_\alpha ^{(s)}f(x) = \sup_{Q\ni x}|Q|^\frac{\alpha}{d}\left(\dashint_Q |f|^s\right)^{1/s},
    \]
    which is just \(M_{\alpha s}(|f|^s)(x)^{1/s}\) and therefore satisfies a weak type \((s, t)\) estimate by Proposition \ref{strong type estimate for fractional max operator with lebesgue}, where \(t = sd/(d-\alpha s)\).\footnote{Note that \(sd/(d-\alpha s)\in [1,\infty[\) because of our assumption that \(0\leq \alpha < d/s\).} For simplicity we define also 
    \[
    \tilde{M}f(x) = \max\{ |Tf(x)|, M_{T,\lambda}f(x)\}.
    \]
    Now, given any cube \(Q\), we define 
    \[
    \Omega(Q) = \left\{x\in Q: \max\left\{\frac{M_\alpha^{(s)}(f\chi_{Q^*})(x)}{A}, \frac{\tilde{M}(f\chi_{Q^*})(x)}{B}\right\} > |Q^*|^\frac{\alpha}{d}\left(\dashint_{Q^*}|f|^s\right)^{1/s}\right\}.
    \]
    The first thing we should do is check that, if we choose \(A, B\) large enough, we can show that \(|\Omega(Q)|\leq 2^{-d-2}|Q|\). To do this of course we will use the weak type estimates for \(T, M_{T,\lambda}\) and \(M_\alpha^{(s)}\). First, we split this as 
    \[
    \begin{split}
        |\Omega(Q)|& \leq \left| \left\{x\in Q: M_\alpha^{(s)}(f\chi_{Q^*})(x)> A |Q^*|^\frac{\alpha}{d}\left(\dashint_{Q^*}|f|^s\right)^{1/s}\right\}\right|\\
        &+\left| \left\{x\in Q: |T(f\chi_{Q^*})(x)|> B |Q^*|^\frac{\alpha}{d}\left(\dashint_{Q^*}|f|^s\right)^{1/s}\right\}\right|\\
        &+ \left| \left\{x\in Q: M_{T,\lambda}(f\chi_{Q^*})(x)> B |Q^*|^\frac{\alpha}{d}\left(\dashint_{Q^*}|f|^s\right)^{1/s}\right\}\right|.
    \end{split}
    \]
    Call these terms \(I_1, I_2\) and \(I_3\), respectively. For the first term, we use the weak type estimate for \(M_\alpha^{(s)}\) to see that
    \[
    \begin{split}
        I_1 & \leq \|M_\alpha^{(s)}\|_{L^s\rightarrow L^{t, \infty}}^tA^{-t}|Q^*|^{-t\frac{\alpha}{d}} \left(\dashint_{Q^*}|f|^s\right)^{-t/s} \|f\chi_{Q^*}\|_{L^s}^t\\
        & = \|M_\alpha^{(s)}\|_{L^s\rightarrow L^{t, \infty}}^tA^{-t}\lambda^{d}|Q|\left(\int_{Q^*}|f|^s\right)^{-t/s}\left(\int_{Q^*}|f|^s\right)^{t/s}\\
        &= \|M_\alpha^{(s)}\|_{L^s\rightarrow L^{t, \infty}}^tA^{-t}\lambda^{d}|Q|,
    \end{split}
    \]
    where we assumed that \(\int_{Q^*}|f|^s>0\). Otherwise, it would follow that \(f = 0\) a.e. on \(Q\) and in that case there 
    would be nothing to prove. Next we estimate \(I_2\),
    \[
    \begin{split}
        I_2 &\leq \|T\|_{L^{p_0}\rightarrow L^{q_0,\infty}}^{q_0} B^{-q_0}|Q^*|^{-q_0\frac{\alpha}{d}}
        \left(\dashint_{Q^*}|f|^s\right)^{-q_0/s}\|f\chi_{Q^*}\|_{L^{p_0}}^{q_0} \\
        &= \|T\|_{L^{p_0}\rightarrow L^{q_0,\infty}}^{q_0} B^{-q_0}|Q^*|^{-q_0\frac{\alpha}{d}}
        \left(\dashint_{Q^*}|f|^s\right)^{-q_0/s}\left(\dashint_{Q^*}|f|^{p_0}\right)^{q_0/p_0}|Q^*|^{q_0/p_0} \\
        &\leq \|T\|_{L^{p_0}\rightarrow L^{q_0,\infty}}^{q_0} B^{-q_0}|Q^*|^{q_0(1/p_0 - \alpha/d)}
        \left(\dashint_{Q^*}|f|^s\right)^{-q_0/s} \left(\dashint_{Q^*}|f|^s\right)^{q_0/s} \\ 
        &= \|T\|_{L^{p_0}\rightarrow L^{q_0,\infty}}^{q_0} B^{-q_0}|Q^*| \\ 
        &= \|T\|_{L^{p_0}\rightarrow L^{q_0,\infty}}^{q_0} B^{-q_0}\lambda^d |Q|.
    \end{split}
    \]
    Finally, we do the same with \(I_3\) using the weak type \((p_1,q_1)\) estimate for \(M_{T,\lambda}\), and we get
    \[
    \begin{split}
        I_3 &\leq \|M_{T, \lambda}\|_{L^{p_1}\rightarrow L^{q_1,\infty}}^{q_1} B^{-q_1}\lambda^d |Q|.
    \end{split}
    \]
    Putting these three estimates together we see that
    \[
    |\Omega(Q)|\leq (\|M_\alpha^{(s)}\|_{L^s\rightarrow L^{t, \infty}}^tA^{-t}+\|T\|_{L^{p_0}\rightarrow L^{q_0,\infty}}^{q_0} B^{-q_0}+\|M_{T, \lambda}\|_{L^{p_1}\rightarrow L^{q_1,\infty}}^{q_1} B^{-q_1})\lambda^d|Q|.
    \]
    Now, if we choose \(A = (3\cdot 2^{d+2}\lambda^d)^{1/t}\|M_\alpha^{(s)}\|\) and \(B = (3\cdot 2^{d+2}\lambda^d)^{1/\min\{q_0,q_1\}}(\|T\|+\|M_{T,\lambda}\|)\), then it follows that
    \[
    |\Omega(Q)|\leq \frac{1}{2^{d+2}}|Q|,
    \]
    as we wanted. The next step is to perform the local Calderón-Zygmund decomposition of \(\chi_{\Omega}\) at the height \(\lambda' = 2^{-d-1}\). This is possible because 
    \[
    \lambda'\geq \dashint_Q \chi_{\Omega(Q)} = \frac{|\Omega(Q)|}{|Q|}.
    \]
    Therefore, we obtain a disjoint family of cubes \(\{Q_j\}_j\subseteq \Des(Q)\) such that
    \begin{itemize}
        \item \(|\Omega(Q)\setminus B(Q)| = 0\);
        \item \(2^{-d-1}|Q_j|< |\Omega(Q) \cap Q_j| \leq |Q_j|/2\);
        \item \( |B(Q)|\leq 2^{d+1}|\Omega(Q)|\),
    \end{itemize}
    where 
    \[
    B(Q) = \bigcup_j Q_j
    \]
    is the bad part of the cube. Note that from the third property and the estimate we've obtained already for the measure of \(\Omega(Q)\), it follows that
    \[
    |B(Q)|\leq \frac{1}{2} |Q|.
    \]
    Now, from the first property we have that
    \[
    |T(f\chi_{Q^*})(x)|\leq B |Q^*|^\frac{\alpha}{d}\left(\dashint_{Q^*}|f|^s\right)^{1/s}\text{ for a.e. }x\in Q\setminus B(Q).
    \]
    Therefore, for almost every \(x\in Q\) we can write
    \[
    |T(f\chi_{Q^*})(x)|\leq B\chi_{Q\setminus B(Q)}(x)|Q^*|^\frac{\alpha}{d}\left(\dashint_{Q^*}|f|^s\right)^{1/s} + \sum_j |T(f\chi_{Q^*})(x)|\chi_{Q_j}(x),
    \]
    which is starting to look like the local estimate we want. Now, we have to solve the problem with the last sum, where we need to compare \(|T(f\chi_{Q^*})(x)|\) with \(|T(f\chi_{Q_j^*})(x)|\). Using the sublinearity of \(T\) we write
    \[
    \sum_j |T(f\chi_{Q^*})(x)|\chi_{Q_j}(x) \leq \sum_j |T(f\chi_{Q^*\setminus Q_j^*})(x)|\chi_{Q_j}(x) + \sum_j |T(f\chi_{Q_j^*})(x)|\chi_{Q_j}(x).
    \]
    The last sum is what we need in order to iterate the argument, so we now focus on the first sum on the right hand side. The basic idea is that, since \(\Omega(Q)\) occupies a small region of \(Q_j\), then we can compare \(|T(f\chi_{Q^*\setminus Q_j^*})(x)|\) to the same function evaluated at a different point \(x'\) such that \(x'\notin \Omega(Q)\). In doing this, we can use the fact that \(x'\notin \Omega(Q)\) to get good bounds for the operator. The problem is that we are going to create an oscillation of the truncation of \(T\) and that is where the auxiliary maximal operator \(M_{T,\lambda}\) is going to be useful. We have that 
    \[
    \begin{split}
        |T(f\chi_{Q^*\setminus Q_j^*})(x)|&\leq |T(f\chi_{Q^*}\chi_{\R^d\setminus Q_j^*})(x)-T(f\chi_{Q^*}\chi_{\R^d\setminus Q_j^*})(x')| + |T(f\chi_{Q^*\setminus Q_j^*})(x')|\\
        &\leq |T(f\chi_{Q^*}\chi_{\R^d\setminus Q_j^*})(x)-T(f\chi_{Q^*}\chi_{\R^d\setminus Q_j^*})(x')| + |T(f\chi_{Q^*})(x')| \\
        &+ |T(f\chi_{Q_j^*})(x')|.
    \end{split}
    \]
    This way we see that for almost every pair \((x, x')\in Q_j^2\) we have that
    \[
    |T(f\chi_{Q^*\setminus Q_j^*})(x)|\leq M_{T,\lambda}(f\chi_{Q^*})(x') + |T(f\chi_{Q^*})(x')| + |T(f\chi_{Q_j^*})(x')|.
    \]
    In other words, the set 
    \[
    X = \{(y', y'')\in Q_j^2: |T(f\chi_{Q^*\setminus Q_j^*})(y')-T(f\chi_{Q^*\setminus Q_j^*})(y'')|> M_{T,\lambda}(f\chi_{Q^*})(y'')\}
    \]
    has zero measure in \(Q_j^2\). At this point what we want to do is show that for almost every \(x\in Q_j\) we can find some \(x'\in Q_j\) such that all of these three terms are controlled. For the first two terms to be controlled we simply need to pick \(x'\in Q_j\setminus \Omega(Q)\). For the last term we define the problematic set 
    \[
    Z = \left\{x\in Q_j: |T(f\chi_{Q_j^*})(x)|> B |Q_j^*|^\frac{\alpha}{d}\left(\dashint_{Q_j^*}|f|^s\right)^{1/s}\right\},
    \]
    and show that it has small measure. This can be done as before using the weak type estimate for \(T\). Indeed, we have that
    \[
    |Z|\leq \frac{1}{3}\frac{1}{2^{d+2}}|Q_j|\leq \frac{1}{24}|Q_j|.
    \]
    This way we see that
    \[
    \begin{split}
    |Q_j\setminus (\Omega(Q)\cup Z)| &= |Q_j| - |(Q_j\cap \Omega(Q))\cup Z|\geq |Q_j| - |Q_j\cap \Omega(Q)| - |Z| \\
    &\geq |Q_j| -\frac{1}{2}|Q_j| - \frac{1}{24}|Q_j| = \frac{11}{24}|Q_j|.
    \end{split}
    \]
    Now recall that our goal here is to show that the set 
    \[
    P = \{x\in Q_j: \exists x'\in Q_j\text{ with }(x, x')\notin X, x'\notin \Omega(Q), x'\notin Z\}
    \]
    has full measure in \(Q_j\). To do this we begin by noticing that 
    \[
    \chi_P(y')(1-\chi_X(y',y''))\chi_{Q_j\setminus (\Omega(Q)\cup Z)}(y'') = \chi_{Q_j}(y')(1-\chi_X(y',y''))\chi_{Q_j\setminus (\Omega(Q)\cup Z)}(y''),
    \]
    because if \(y'\in P\) both sides are obviously equal and if \(y'\in Q_j\setminus P\) this implies that \((y', y'')\in X\lor y''\in \Omega(Q) \lor y''\in Z\) and in any case both sides are equal to zero. Therefore, using the fact that \(\chi_X = 0\) almost everywhere in \(Q_j^2\), we have that
    \[
    \begin{split}
    |Q_j||Q_j\setminus (\Omega(Q)\cup Z)| &= \int_{Q_j} \int_{Q_j} \chi_{Q_j\setminus (\Omega\cup Z)}(y'')dy'' dy' \\
    &= \int_{Q_j^2}\chi_{Q_j}(y') (1-\chi_X(y',y''))\chi_{Q_j\setminus (\Omega\cup Z)}(y'')dy' dy'' \\
    &= \int_{Q_j^2}\chi_P(y')(1-\chi_X(y',y''))\chi_{Q_j\setminus (\Omega \cup Z)}(y'')dy'dy'' \\
    &= |P| |Q_j\setminus (\Omega(Q) \cup Z)|.
    \end{split}
    \]
    Given that we already know that \(|Q_j\setminus (\Omega(Q)\cup Z)|>0\), it follows that \(|P| = |Q_j|\) as we wanted to show. This means that for almost every \(x\in Q_j\) we can choose \(x'\in Q_j\) such that \((x, x')\notin X\) and \(x'\notin\Omega \cup Z\). But then this implies that
    \[
    \begin{split}
          |T(f\chi_{Q^*\setminus Q_j^*})(x)|&\leq  M_{T,\lambda}(f\chi_{Q^*})(x') + |T(f\chi_{Q^*})(x')| + |T(f\chi_{Q_j^*})(x')| \\
          &\leq B|Q^*|^\frac{\alpha}{d}\left(\dashint_{Q^*}|f|^s\right)^{1/s} + B|Q^*|^\frac{\alpha}{d}\left(\dashint_{Q^*}|f|^s\right)^{1/s} + B|Q_j^*|^\frac{\alpha}{d}\left(\dashint_{Q_j^*}|f|^s\right)^{1/s} \\
           &\leq 2B|Q^*|^\frac{\alpha}{d}\left(\dashint_{Q^*}|f|^s\right)^{1/s} + BM_\alpha^{(s)}(f\chi_{Q^*})(x')\\
           &\leq 2B|Q^*|^\frac{\alpha}{d}\left(\dashint_{Q^*}|f|^s\right)^{1/s} + AB|Q^*|^\frac{\alpha}{d}\left(\dashint_{Q^*}|f|^s\right)^{1/s} \\ 
           & = (2+A)B|Q^*|^\frac{\alpha}{d}\left(\dashint_{Q^*}|f|^s\right)^{1/s}.
    \end{split}
    \]
    Putting everything together we conclude that for almost every \(x\in Q\),
    \[
    |T(f\chi_{Q^*})(x)|\leq (2+ A)B\chi_Q(x)|Q^*|^\frac{\alpha}{d}\left(\dashint_{Q^*}|f|^s\right)^{1/s} + \sum_j |T(f\chi_{Q_j^*})(x)|\chi_{Q_j}(x).
    \]
    The idea now is that we add the cube \(Q\) to the family \(\mcal{F}_Q\) with the choice \(G(Q) = Q\setminus B(Q)\). The set \(G(Q)\) is then disjoint from all the cubes \(Q_j\) and \(|G(Q)|\geq |Q|/2\). Next we simply iterate the precedure. In each cube \(Q_j\) we define \(\Omega(Q_j)\) as before and we obtain a decomposition 
    \[
    Q_j = \bigcup_k Q_{j, k}
    \]
    as before. We then have the estimate
    \[
    |T(f\chi_{Q_j^*})(x)|\leq (2+A)B\chi_{Q_j}(x)|Q_j^*|^\frac{\alpha}{d}\left(\dashint_{Q_j^*}|f|^s\right)^{1/s} + \sum_k |T(f\chi_{Q_{j,k}^*})(x)|\chi_{Q_{j,k}}(x)
    \]
    for almost every \(x\in Q_j\). If we then add all the cubes \(\{Q_j\}\) to \(\mcal{F}_Q\), this leads to the estimate
    \[
    |T(f\chi_{Q^*})(x)|\leq (2+A)B \sum_{R\in \mcal{F}_Q}\chi_R(x) |R^*|^\frac{\alpha}{d}\left(\dashint_{R^*}|f|^s\right)^{1/s} + \sum_{j, k}|T(f\chi_{Q_{j,k}^*})(x)|\chi_{Q_{j,k}}(x)
    \]
    for almost every \(x\in Q\). Moreover, for each \(Q_j\) we define \(G(Q_j) = Q_j\setminus \cup_k Q_{j,k}\) and this way all the sets in \(\mcal{F}_Q\) are pairwise disjoint and \(|G(R)|\geq |R|/2\), for all \(R\in \mcal{F}_Q\). In other words, \(\mcal{F}_Q\) remains a \(1/2\)-sparse set. Denote \(B(Q)\) by \(B_1\) and put
    \[
    B_2 = \bigcup_{j,k} Q_{j, k} = \bigcup_j B(Q_j).
    \]
    We already know that \(|B_1|\leq |Q|/2\), but we also have
    \[
    |B_2| = \sum_j |B(Q_j)| \leq \sum_j \frac{1}{2}|Q_j| = \frac{1}{2}|B(Q)|\leq \frac{1}{2^2}|Q|.
    \]
    Therefore,
    \[
    |T(f\chi_{Q^*})(x)|\leq (2+A)B\sum_{R\in \mcal{F}_Q}\chi_R(x) |R^*|^\frac{\alpha}{d}\left(\dashint_{R^*}|f|^s\right)^{1/s}\text{ for a.e. }x\in Q\setminus B_2,
    \]
    and \(|B_2|\leq 2^{-2}|Q|\). If we continue with this process, then after \(N\) steps we will still have a \(1/2\)-sparse set \(\mcal{F}_Q\) with the estimate
    \[
    |T(f\chi_{Q^*})(x)|\leq (2+A)B\sum_{R\in \mcal{F}_Q}\chi_R(x) |R^*|^\frac{\alpha}{d}\left(\dashint_{R^*}|f|^s\right)^{1/s}\text{ for a.e. }x\in Q\setminus B_N,
    \]
    where \(|B_N|\leq 2^{-N}|Q|\). Defining \(B_\infty = \cap_N B_N\) we get that \(|B_\infty| = 0\) and 
    \[
    |T(f\chi_{Q^*})(x)|\leq (2+A)B\sum_{R\in \mcal{F}_Q}\chi_R(x) |R^*|^\frac{\alpha}{d}\left(\dashint_{R^*}|f|^s\right)^{1/s}\text{ for a.e. }x\in Q.
    \]
    This then shows that the assumptions of Lemma \ref{lemma: local to global sparse bound} hold, and therefore there is some \(1/(2\lambda^d)\)-sparse family of cubes \(\S = \S(f)\) such that
    \[
    |Tf(x)|\leq (2+A)B\sum_{Q\in \S}\chi_Q(x)|Q|^\frac{\alpha}{d}\left(\dashint_Q |f|^s\right)^{1/s}\text{ for a.e. }x\in \R^d,
    \]
    as we wanted to show.
\end{proof}

\begin{remark}
    This theorem is a refinement of the original work of Andrei Lerner (see \cite{lerner}). 
    Indeed, in \cite{lerner}, instead of using the auxiliary maximal operator \(M_{T,\lambda}\), 
    a stronger assumption was considered involving the use of the grand maximal truncated operator 
    \[
    \mcal{M}_T f(x) = \sup_{Q\ni x}\underset{x'\in Q}{\text{ess sup }}|T(f\chi_{\R^d\setminus Q^*})(x')|.
    \]
    We should note here that the pointwise sparse domination principle in \cite{lerner} is based on the work of several others. 
    In fact, Lerner himself calls it the Lacey-Hytönen-Roncal-Tapiola theorem. Later, in \cite{LernerOmbrosi}, 
    Lerner and Ombrosi improve the pointwise sparse domination principle in two ways. 
    The first improvement is the switch from using \(\mcal{M}_T\) to considering the oscillations of truncations with the 
    operator \(M_{T,\lambda}\). The second is a weakening of the weak type estimates used in the proof, 
    which are replaced by what the authors call a \(W_q\) condition. We are not that interested in this second improvement 
    since in the applications of this theorem that follow we will be able to deduce weak type estimates. However, the first 
    improvement is useful for our purposes as it simplifies many proofs. 
    This is the reason why we have given a version of the theorem that includes the first improvement but not the second. 
    We remark also that both in \cite{lerner} and in \cite{LernerOmbrosi}, only the case \(\alpha = 0\) is considered. 
    Here we have chosen to give a version of the theorem with this slight improvement so that we 
    unify the application of this theorem to the integral operators we consider in Section \ref{sec: applications of sparse domination}.
\end{remark}

\begin{remark}
    In Remark \ref{remark: sparse estimate improvement for fractional maximal operator} we have noted that the sparse estimate could 
    be improved for the fractional maximal operator. This improvement consisted in replacing \(\chi_Q(x)\) with \(\chi_{G(Q)}(x)\). 
    This was essentially due to the ability to say that for maximal cubes \(M_\alpha^\D f(x) = M_\alpha^\D(f\chi_{Q_{j,1}})(x)\). 
    However, in the present case, we got around this issue by having control on the maximal oscillations of \(T\). But this 
    introduces a contribution coming from \(B(Q)\) which has to be taken into account. That is, we argued that 
    \[ 
        \begin{split}
        |T(f\chi_{Q^*})(x)| &\leq B\chi_{Q\setminus B(Q)}(x)|Q^*|^\frac{\alpha}{d}\left(\dashint_{Q^*}|f|^s\right)^{1/s} + 
        \sum_j |T(f\chi_{Q^*\setminus Q_j^*})(x)|\chi_{Q_j}(x) \\
        &+ \sum_j |T(f\chi_{Q_j^*})(x)|\chi_{Q_j}(x) \\
        &\leq B\chi_{Q\setminus B(Q)}(x)|Q^*|^\frac{\alpha}{d}\left(\dashint_{Q^*}|f|^s\right)^{1/s} \\
        &+ \sum_j \chi_{Q_j}(x) (2+A)B|Q^*|^\frac{\alpha}{d}\left(\dashint_{Q^*}|f|^s\right)^{1/s} \\ 
        &+ \sum_j |T(f\chi_{Q_j^*})(x)|\chi_{Q_j}(x) \\
        &\leq (2 + A)B |Q^*|\left(\dashint_{Q^*}|f|^s\right)^{1/s} \chi_Q(x) + \sum_j |T(f\chi_{Q_j^*})(x)|\chi_{Q_j}(x).
        \end{split}
    \] 
    This way, we see that the contribution from \(\cup_j Q_j\) is added to the term from \(Q\setminus B(Q)\) to yield 
    a characteristic function over the entire cube, which prevents us from improving the estimate.
\end{remark}

Note that the cubes we have obtained in the sparse family \(\S\) are not cubes coming from a certain dyadic lattice, because of the threefold expansions that occur in Lemma \ref{lemma: local to global sparse bound}. However, by using the three lattice theorem we can always compare arbitrary cubes with dyadic cubes. Indeed, given a dyadic lattice \(\D\), we know from Theorem \ref{three lattice theorem} that there are dyadic lattices \(\D^1,\dots, \D^{3^d}\) such that, given a cube \(Q\in \S\),  there is some \(Q'\in \D^{j(Q)}\) with \(Q\subseteq Q'\) and \(|Q'|\leq 3^d |Q|\), and therefore
\[
\chi_Q(x)|Q|^\frac{\alpha}{d}\left(\dashint_Q |f|^s\right)^{1/s} \leq 3^{\frac{d}{s}-\alpha}\chi_{Q'}(x)|Q'|^\frac{\alpha}{d}\left(\dashint_{Q'}|f|^s\right)^{1/s}.
\]
So, if we set \(\S_j = \{Q': Q\in \S \text{ and }Q'\in \D^j\}\),\footnote{Note that these sets are also sparse with a possibly smaller constant.} we obtain the estimate
\[
|Tf(x)|\lesssim \sum_{j=1}^{3^d}\sum_{Q'\in \S_j}\chi_{Q'}(x)|Q'|^\frac{\alpha}{d}\left(\dashint_{Q'}|f|^s\right)^{1/s}.
\]
In other words, we can dominate \(T\) by a finite number of sparse operators whose associated collections of sparse cubes all come from a fixed dyadic lattice.

Before finishing this chapter we remark that the approach to sparse domination presented here, which is based on 
obtaining pointwise bounds for operators with controlled oscillations of their truncations, is not the only approach that has been developed. 
For example, we highlight the work of Bernicot, Frey and Petermichl (\cite{BernicotFreyPetermichl}), who developed a non-pointwise approach, 
focusing instead on controlling operators via bilinear sparse forms. They also consider a much more general framework, 
proving sparse domination for non-integral operators on metric spaces endowed with Borel measures satisfying some conditions, 
among which the condition that the measure satisfies a volume doubling property. 
We also mention the work of Conde-Alonso, Culiuc, Di Plinio and Ou, who obtained in \cite{conde-alonso} 
bounds for rough singular integrals in terms of bilinear sparse forms, by 
using a novel method which avoids the consideration of maximal truncations. Finally, we highlight also the work of Lacey, who used a modification of the 
bilinear sparse form approach to prove sparse control for the spherical maximal function (see \cite{lacey}).

\clearpage
\section{Applications of the Sparse Domination Theorem}\label{sec: applications of sparse domination}

In this section, we will show several consequences of Theorem \ref{thm: lerner-ombrosi}. In particular,
we show how we can apply the Lerner-Ombrosi theorem to deduce sparse bounds for several operators of interest, including the Riesz potential operator 
and Calderón-Zygmund operators. The results of this section can be found in several papers throughout the literature. See for instance
\cite{UribeMoen2}, \cite{lernerCZ}, \cite{laceyA2} and \cite{lerner_diplinio}.

\subsection{The Riesz potential operator}\label{sec: the riesz potential operator}

In some cases, working with fractional derivatives is much more convenient than working with classical derivatives. 
A nice way of defining fractional derivatives is by using the Fourier transform. Recall that 
\[
\widehat{\Delta f}(\xi) = -4\pi^2 |\xi|^2\hat{f}(\xi).
\]
This leads to a natural definition of a fractional Laplacian operator\index{fractional Laplacian} by setting
\[
(-\Delta)^{s/2} f = ((2\pi |\xi|)^s \hat{f})^\vee. 
\]
When \(s\geq 0\) we think of this as a fractional derivative of \(f\). However, this operator is well-defined even if \(s<0\). 
For example, if \(f\in \S(\R^d)\) and \(s>-d\), then \(|\xi|^s \hat{f}\in L^1\) and therefore \((-\Delta )^{s/2}f\) makes sense as the 
inverse Fourier transform of an \(L^1\) function, which we know will be continuous, bounded and 0 at infinity. When \(s<0\) we can think 
of this as a fractional integral instead of a fractional derivative, and in fact, when \(-d<s<0\), we can find an integral expression 
for the fractional Laplacian. If \(0<\alpha<d\), we can use the well-known\footnote{For a reference see \cite{Cgrafakos}, Section 2.4.3.} 
fact that 
\[
(|\cdot|^{-\alpha})^\wedge (\xi) = \pi^{\alpha- \frac{d}{2}}\frac{\Gamma \left(\frac{d-\alpha}{2}\right)}{\Gamma\left(\frac{\alpha}{2}\right)} |\xi|^{\alpha -d }
\]
to see that, if we write \(s = -\alpha\), we get 
\[
(-\Delta )^{s/2}f = ((2\pi |\cdot|)^s\hat{f})^\vee = (2\pi)^s (|\cdot|^{-\alpha})^\vee * f = 2^{-\alpha} \pi^{-\frac{d}{2}}\frac{\Gamma \left(\frac{d-\alpha}{2}\right)}{\Gamma\left(\frac{\alpha}{2}\right)} |\xi|^{\alpha -d } * f.
\]
This naturally leads to the consideration of the Riesz potential operator\index{Riesz potential operator}
\[
\I_\alpha f(x) = |\cdot|^{\alpha-d} * f(x) = \int_{\R^d} \frac{1}{|y|^{d-\alpha}}f(x-y)dy.
\]
The boundedness properties of this operator are then related to Sobolev inequalities, which are of fundamental importance in the study of partial differential equations.

A standard way of obtaining boundedness properties for the Riesz potential operator is to compare it to the Hardy-Littlewood maximal operator. 
Essentially one splits the integral as follows: 
\[
\I_\alpha f(x) = \int_{|y|\leq R(x)}|y|^{\alpha- d}f(x-y)dy + \int_{|y|>R(x)}|y|^{\alpha - d}f(x-y)dy.
\]
Then, if we see the first term as a convolution between \(f\) and a radially decreasing integrable function, if we use Hölder' inequality in the second term, and we then choose \(R(x)\) so as to minimize the right-hand side (one chooses \(R(x) = \|f\|_{L^p}^{p/d}Mf(x)^{-p/d}\)), we get
\[
\I_\alpha f(x) \lesssim_{d,\alpha, p} Mf(x)^{p/q}\|f\|_{L^p}^{1-p/q},
\]
where 
\[
\frac{1}{q}+\frac{\alpha}{d} = \frac{1}{p}.
\]
By Proposition \ref{strong type estimate for fractional max operator with lebesgue} it follows that
\begin{align*}
    \|\I_\alpha f\|_{L^q(\R^d)}&\lesssim_{d, \alpha,p}\|f\|_{L^p(\R^d)}\\
    \|\I_\alpha f\|_{L^{\frac{d}{d-\alpha},\infty}(\R^d)}&\lesssim_{d, \alpha} \|f\|_{L^1(\R^d)}.
\end{align*}
These bounds can then be thought of as fractional versions of the Sobolev inequalities, considering the Riesz potential operator 
as a kind of fractional integral. If we wish to obtain two weight norm inequalities for the Riesz potential operator one strategy is to use the methods of 
Section \ref{sec: sparse domination} and try to compare it to a sparse operator. To do this we will use Theorem \ref{thm: lerner-ombrosi}.

Since the Riesz potential operator satisfies a weak type estimate, then in order to apply Theorem \ref{thm: lerner-ombrosi} it remains to check that the auxiliary maximal operator
\[
M_{\I, \lambda} f(x) = \sup_{Q\ni x}\underset{x',x''\in Q}{\text{ess sup }}|\I_\alpha(f\chi_{\R^d\setminus Q^*})(x')-\I_\alpha(f\chi_{\R^d\setminus Q^*})(x'')|
\]
also satisfies a weak type estimate. To do this we prove a result which shows that, in some generality, 
whenever we have an operator associated to a sufficiently nice kernel, we can always find a weak type estimate for the auxiliary maximal operator.
\begin{prop}\label{prop: condition for weak type auxiliary maximal op}
    Let \(0\leq \alpha < d\). Suppose we have an operator \(T\) that has the following integral representation: given a cube 
    \(Q\) and some \(f\in L^1(\R^d)\),
    \[
    T(f\chi_{\R^d\setminus Q^*})(x) = \int_{\R^d\setminus Q^*} K(x,y)f(y)dy, \text{ for a.e. }x\in Q,
    \]
    for some kernel \(K\). Moreover, assume the this kernel satisfies the following regularity condition:
    \[
    |K(x',y)-K(x'',y)|\leq \omega\left(\frac{|x'-x''|}{|x'-y|}\right)\frac{1}{|x'-y|^{d-\alpha}} \text{ when } |x'-x''|<|x'-y|/2,
    \]
    where \(\omega: [0, 1]\rightarrow[0,\infty[\) is an increasing function such that \(\omega(0) = 0\) and
    \[
    [\omega]_\text{\emph{Dini}} = \int_0^1 \frac{\omega(t)}{t}dt < \infty. 
    \]
    Then, if we take \(\lambda > 4\sqrt{d}+1\), the auxiliary maximal operator \(M_{T,\lambda}\) satisfies 
    \[
    \|M_{T,\lambda}\|_{L^1\rightarrow L^{\frac{d}{d-\alpha},\infty}} \lesssim_{d,\alpha} [\omega]_\emph{Dini}.
    \]
\end{prop}
\begin{proof}
    Fix a point \(x\in \R^d\), a cube \(Q\) which contains \(x\) and some function \(f\in L^1(\R^d)\). For almost every 
    \(x', x'' \in Q\), we can write
    \[
    |T(f\chi_{\R^d\setminus Q^*})(x')-T(f\chi_{\R^d\setminus Q^*})(x'')|\leq \int_{\R^d\setminus Q^*}|K(x',y)-K(x'',y)||f(y)|dy.
    \]
    We now clearly want to use the second property of the kernel to control the right-hand side, which we call \(I\). Since \(\lambda > 4\sqrt{d}+1\) then we get
    \[
    |x'-x''|\leq \text{diam}(Q) = \sqrt{d}l(Q) < \frac{1}{2} \frac{\lambda-1}{2}l(Q) \leq \frac{1}{2}|x'-y|,
    \]
    and so we can use the second property to get
    \[
    I\leq \int_{\R^d\setminus Q^*}\omega\left(\frac{|x'-x''|}{|x'-y|}\right)\frac{|f(y)|}{|x'-y|^{d-\alpha}}dy.
    \]
    The idea now is to compare this to the fractional maximal operator. We have that
    \[
    \begin{split}
        I&\leq  \sum_{k\geq 0}\int_{D_{2^{k+1}}Q^*\setminus D_{2^k}Q^*}\omega\left(\frac{|x'-x''|}{|x'-y|}\right) \frac{|f(y)|}{|x'-y|^{d-\alpha}}dy \\
        &\leq \sum_{k\geq 0}\frac{|D_{2^{k+1}}Q^*|^{1-\frac{\alpha}{d}}}{(2^{k-1}\lambda -2^{-1})^{d-\alpha}l(Q)^{d-\alpha}}\omega\left(\frac{\text{diam}(Q)}{(2^{k-1}\lambda-2^{-1})l(Q)}\right) |D_{2^{k+1}}Q^*|^{\frac{\alpha}{d}-1}\int_{D_{2^{k+1}}Q^*}|f|\\
        &\leq \sum_{k\geq 0}\left(\frac{2^{k+2}\lambda}{2^k\lambda - 1}\right)^{d-\alpha} \omega(2^{-k-1})M_\alpha f(x) \\
        &\leq 5^{d-\alpha} M_\alpha f(x) \sum_{k\geq 0}\omega(2^{-k-1}).
    \end{split}
    \]
    Now, since \(\omega\) is increasing we can say that
    \[
    \begin{split}
    \sum_{k\geq 0}\omega(2^{-k-1}) &= \sum_{k\geq 0} \frac{\omega(2^{-k-1})}{2^{-k-1}} 2^{-k-1} = \sum_{k \geq 0}\int_{2^{-k-1}}^{2^{-k}}\frac{\omega(2^{-k-1})}{2^{-k-1}}dt \\
    &\leq \sum_{k\geq 0}\int_{2^{-k-1}}^{2^{-k}}\frac{\omega(t)}{2^{-k-1}}dt \leq 2\sum_{k\geq 0}\int_{2^{-k-1}}^{2^{-k}}\frac{\omega(t)}{t}dt = 2 [\omega]_\text{Dini}.
    \end{split}
    \]
    Therefore, we see that 
    \[
    M_{T,\lambda}f(x) \leq 2\times 5^{d-\alpha} [\omega]_\text{Dini}M_\alpha f(x).
    \]
    The result then follows from Proposition \ref{strong type estimate for fractional max operator with lebesgue}.
\end{proof}

To obtain sparse domination for the Riesz potential operator we simply have to show that its kernel satisfies the assumption of Proposition \ref{prop: condition for weak type auxiliary maximal op}. The kernel associated to the Riesz potential operator is 
\[
K(x,y) = \Phi(x-y),\ \text{where }\Phi(\eta) = |\eta|^{\alpha- d}.
\]
Suppose that \(|x'-x''|<|x'-y|/2\). Then, we have that
\[
\begin{split}
|K(x',y)-K(x'',y)| &= |\Phi(x'-y)-\Phi(x''-y)| \\
&\leq \int_0^1 |\nabla \Phi(x''-y+t(x'-x''))|dt |x'-x''| \\ 
&\leq (d-\alpha)\int_0^1 |x''-y+t(x'-x'')|^{\alpha-d-1}dt |x'-x''|.
\end{split}
\]
Now, note that
\[
|x'-y|\leq |x''-y+t(x'-x'')| + (1-t)|x'-x''|\leq |x''-y+t(x'-x'')| + \frac{|x'-y|}{2}.
\]
Therefore, we obtain
\[
|K(x',y)-K(x'',y)|\leq (d-\alpha) 2^{d+1-\alpha} \frac{|x'-x''|}{|x'-y|} \frac{1}{|x'-y|^{d-\alpha}}.
\]
This way, we see that \(\I_\alpha\) satisfies the assumptions of Proposition \ref{prop: condition for weak type auxiliary maximal op} with \(\omega(t) =(d-\alpha)2^{d+1-\alpha} t\). If we choose for example \(\lambda = 4\sqrt{d} +2\), then \(\lambda > 4\sqrt{d}+1\) and \(\lambda \geq 3\), therefore, putting together Proposition \ref{prop: condition for weak type auxiliary maximal op} and Theorem \ref{thm: lerner-ombrosi} (with \(s = 1\)), we conclude the following proposition.
\begin{prop}\label{prop: sparse domination for the Riesz potential op}
    Let \(0< \alpha < d\). Given \(f\in L^1(\R^d)\) with compact support, there is a \(1/(2(4\sqrt{d}+2)^d)\)-sparse family of cubes \(\S = \S(f)\) such that
    \[
    |\I_\alpha f(x)|\lesssim_{d,\alpha} \sum_{Q\in \S} \chi_Q(x) |Q|^\frac{\alpha}{d}\dashint_Q |f| \text{ for a.e. }x\in \R^d.
    \]
\end{prop}
\begin{remark}
    This result was obtained by Cruz-Uribe and Moen in \cite{UribeMoen2} based on previous work by several authors. Their approach relies on a comparison between 
    the Riesz potential operator and a dyadic version, which exploits the simplicity of this operator's kernel. The proof of the proposition given here instead
    takes advantage of the more general properties of such kernels encapsulated in the off-diagonal version of the Lerner-Ombrosi theorem. 
    This is interesting as it shows how this sparse domination result fits in nicely with the general framework used 
    to obtain sparse bounds for Calderón-Zygmund operators.
\end{remark}

This proposition together with Theorem \ref{thm: weighted estimate for sparse operators} implies the following:

\begin{thm}\label{thm: two-weight norm ineq for the Riesz potential}
    Let \(1 < p \leq q < \infty\) and \(0 < \alpha < d\). Suppose that \(v, w\) are such that \((v, w)\in A_{p, q}^\alpha\) and \(v^{-\frac{p'}{p}}, w\in A_\infty\). Then, we have that
    \[
    \|\I_\alpha f\|_{L^q(w)}\lesssim \|f\|_{L^p(v)}.
    \]
\end{thm}
\begin{proof}
    Since \(|\I_\alpha f(x)|\leq \I_\alpha(|f|)(x)\) we may assume, without loss of generality, that \(f\geq 0\). Now put \(f_N = \min\{f, N\} \chi_{B_N}\). Since \(f_N\in L^1\) with compact support, we can apply Proposition \ref{prop: sparse domination for the Riesz potential op} to show the existence of a sparse family of cubes \(\S_N = \S_N(f_N)\) such that
    \[
    \I_\alpha(f_N)(x) \lesssim_{d, \alpha}\A_{\S_N}^\alpha(f_N)(x)\text{ for a.e. }x.
    \]
    Therefore, from Theorem \ref{thm: weighted estimate for sparse operators}, we get
    \[
    \|\I_\alpha(f_N)\|_{L^q(w)}\lesssim_{d,\alpha} \|\A_{\S_N}^\alpha (f_N)\|_{L^q(w)}\lesssim \|f_N\|_{L^p(v)}\lesssim \|f\|_{L^p(v)}.
    \]
    It is important to note here that despite the fact that the sparse collections \(\S_N\) depend on \(f_N\), the implicit constants in the inequalities do not depend on \(f_N\), they only depend on the sparse factor \(\eta\) we get from Proposition \ref{prop: sparse domination for the Riesz potential op} which depends only on \(d\). Now the result follows from the monotone convergence theorem since 
    \[
    \I_\alpha f_N(x) = \int_{B(0, N)}\frac{\min\{f(y), N\}}{|x-y|^{d-\alpha}}dy \nearrow_N \int_{\R^d} \frac{f(y)}{|x-y|^{d-\alpha}}dy,
    \]
    where the limit here also holds because of the monotone convergence theorem.    
\end{proof}
\begin{remark}
    This theorem generalizes the Stein-Weiss inequality (\cite{SteinWeiss}): 
    \[ 
        \||x|^\gamma \I_\alpha f\|_{L^q}\lesssim \||x|^\beta f\|_{L^p},
    \] 
    where \(-d/q < \gamma \leq \beta < d/p'\) and \(1/q + \alpha/d +\gamma/d = 1/p + \beta/d\).
\end{remark}
\begin{remark}
    This theorem is well-known and in fact a proof can be found in \cite{sawyer}. 
    The proof by Sawyer and Wheeden also makes heavy use of dyadic methods and it contains many of the ideas that were used in 
    Sections \ref{sec: dyadic lattices} and \ref{sec: sparse domination}, like the notion of a \(\Lambda\)-Carleson family, except they were applied specifically to the Riesz potential operator. 
    This has some advantages and some disadvantages. On the one hand, their proof actually allows for an improvement of the 
    \(A_\infty\) condition. Indeed, they assume only that the weights \(v^{-p'/p}, w\) satisfy a reverse doubling condition.\index{reverse doubling condition} 
    That is, a weight \(w\) satisfies the reverse doubling condition if there are \(0<\delta, \varepsilon<1\) such that
    \[
    w(D_\delta Q) \leq \varepsilon w(Q)
    \]
    uniformly over all cubes. If a weight is in \(A_\infty\), then it must satisfy the reverse doubling condition, 
    but the other implication is not true.\footnote{See \cite{fefferman-muckenhoupt} for an example of a doubling 
    weight which is not \(A_\infty\), and note that doubling implies reverse doubling.} Therefore, 
    this is a strict improvement over Theorem \ref{thm: two-weight norm ineq for the Riesz potential}. 
    However, on the other hand, the sparse domination methods have some other advantages. First of all, it is a much more general method, 
    since it is applicable to a wide range of other operators as we will shortly see. But also, 
    they can be used to obtain precise information about how the implicit constants depend on the weights. 
    However, we choose to state this theorem here, without these quantitative improvements that come from using sparse 
    domination, because we are not too concerned with the implicit constants. 
    If the reader is interested in such results, a recent paper that proves two-weight norm inequalities for the 
    Riesz potential operator with an explicit dependence of the constants on the weights is \cite{uribe-moen}.
\end{remark}

\subsection{Calderón-Zygmund operators and their maximal counterparts}\label{sec: sparse bounds for CZ op}\label{sec: sparse domination for CZ op}

In this section we are interested in obtaining sparse domination for Calderón-Zygmund operators\index{\(\omega\)-Calderón-Zygmund operators}
and also for their maximal versions. 
We assume here that we are dealing with a linear operator \(T\) which is bounded in \(L^2(\R^d)\) and has an integral representation
\[
Tf(x) = \int K(x,y)f(y)dy,\ \forall f\in L^\infty(\R^d)\text{ with compact support and }x\notin \supp(f).
\]
Moreover, we make some standard assumptions about the kernel of \(T\), in line with \cite{lerner}. More precisely, 
we assume that we have the estimates
\begin{itemize}
    \item \(|K(x,y)|\leq C_K/|x-y|^d\), for all \(x\neq y\);
    \item \(|K(x,y)-K(x',y)|+|K(y,x)-K(y,x')|\leq \omega\left(\frac{|x-x'|}{|x-y|}\right) \frac{1}{|x-y|^d}\) for \(|x-y|> 2|x-x'|\),
\end{itemize}
where \(\omega: [0,1]\rightarrow[0,\infty[\) is a subadditive increasing function with \(\omega(0) = 0\) and \([\omega]_\text{Dini} < \infty\). 
Then, under these assumptions, it follows from the standard Calderón-Zygmund theory (see for instance Theorem 5.10 from \cite{Duo} and 
Exercise 4.2.4 from \cite{Mgrafakos}) that \(T\) can be uniquely extended to all \(L^1\) functions, satisfying a weak type \((1, 1)\) 
inequality with 
\[
\|T\|_{L^1 \rightarrow L^{1,\infty}}\lesssim_d C_K + [\omega]_\text{Dini} + \|T\|_{L^2\rightarrow L^2}.
\]
To see that \(T\) satisfies the assumptions of Proposition \ref{prop: condition for weak type auxiliary maximal op}, it remains only to 
check that 
\[ 
    T(f\chi_{\R^d\setminus Q^*})(x) = \int_{\R^d\setminus Q^*} K(x,y)f(y)dy\text{ for a.e. }x \in Q,
\] 
for any \(f\in L^1(\R^d)\). To this end, let \(Q\) be some fixed cube and consider \(f\in L^1(\R^d)\). Since we have an integral 
representation of \(T\) for \(L^\infty\) functions with compact support, let's approximate \(f\) by such functions. Take \((f_n)_n\) to 
be a sequence of \(C^\infty_c(\R^d)\) functions which converges to \(f\) in \(L^1(\R^d)\). Given that \(T\) is weak type \((1, 1)\) we know that 
\[ 
    \|T(f\chi_{\R^d\setminus Q^*} - f_n\chi_{\R^d\setminus Q^*})\|_{L^{1, \infty}}\lesssim \|f-f_n\|_{L^1}\xrightarrow[n\rightarrow +\infty]{} 0.
\] 
Therefore, there is some subsequence \((f_{n_k})_k\) such that
\[ 
    T(f_{n_k}\chi_{\R^d\setminus Q^*})(x) \xrightarrow[k\rightarrow +\infty]{} T(f\chi_{\R^d\setminus Q^*})(x)
\] 
for a.e. \(x\in \R^d\). On the other hand, if \(x\in Q\)
\[ 
    T(f_{n_k}\chi_{\R^d\setminus Q^*})(x) = \int_{\R^d\setminus Q^*} K(x,y)f_{n_k}(y)dy
\] 
and 
\[ 
    \begin{split}
        \left|\int_{\R^d\setminus Q^*} K(x,y)f_{n_k}(y)dy-\int_{\R^d\setminus Q^*} K(x,y)f(y)dy\right| &\leq C_K\int_{\R^d\setminus Q^*} 
        \frac{|f_{n_k}(y) - f(y)|}{|x-y|^d}dy \\
        &\lesssim_\lambda \frac{C_K}{|Q|}\|f_{n_k}-f\|_{L^1} \xrightarrow[k\rightarrow +\infty]{} 0.
    \end{split}
\]
Thus, 
\[ 
    T(f\chi_{\R^d\setminus Q^*})(x) = \int_{\R^d\setminus Q^*} K(x,y)f(y)dy,\text{ for a.e. }x \in Q.
\] 
This way, we see that \(T\) satisfies all the assumptions of Proposition \ref{prop: condition for weak type auxiliary maximal op} and therefore, 
if we apply Theorem \ref{thm: lerner-ombrosi} with \(s = 1\) and \(\lambda = 4\sqrt{d}+2\), we show that, for any \(f\in L^1(\R^d)\) with 
compact support, there is some \(1/(2\lambda^d)\)-sparse collection of cubes \(\S = \S(f)\) such that
\[
|Tf(x)|\lesssim_d (C_K + [\omega]_\text{Dini} + \|T\|_{L^2\rightarrow L^2})\sum_{Q\in \S} \chi_Q(x) \dashint_Q|f|\text{ for a.e. }x\in \R^d.
\]
This sparse domination estimate was obtained by Lacey in \cite{laceyA2} for Calderón-Zygmund operators with a Dini condition. 
Before that work, Lerner in \cite{lernerCZ} proved such a sparse bound under stronger assumptions on the kernel of the operator.
For ease of notation, we will write the right-hand side as \(\A_\S f(x)\) instead of \(\A_\S^0f(x)\). 
One can then apply Theorem \ref{thm: weighted estimate for sparse operators} to obtain two-weight norm estimates for Calderón-Zygmund operators.

Now that we have seen that we can control Calderón-Zygmund operators by sparse operators, we want to see that the same is true for maximal Calderón-Zygmund operators.\index{maximal Calderón-Zygmund operator} Given an operator \(T\) under the above conditions we define the maximal operator
\[
T^*f(x) = \sup_{\varepsilon > 0}\left| \int_{|x-y|>\varepsilon}K(x, y)f(y)dy\right|.
\]
Again, by standard Calderón-Zygmund theory, one can use Cotlar's inequality to show that
\[
\|T^*\|_{L^1\rightarrow L^{1,\infty}} \lesssim_d C_K + [\omega]_\text{Dini} + \|T\|_{L^2\rightarrow L^2}.
\]
Therefore we just have to show that \(M_{T^*,\lambda}\) is weak type \((1, 1)\) in order to obtain sparse domination. Let \(x\in \R^d\) and \(Q\) be a cube containing \(x\). For \(\varepsilon \leq (\lambda - 1)l(Q)/2\) we have that \(B_\varepsilon(x)\subseteq Q^*\) and therefore
\[
\int_{|x-y|>\varepsilon}K(x,y)f(y)\chi_{\R^d\setminus Q^*}(y)dy = \int_{\R^d\setminus Q^*}K(x,y)f(y)dy.
\]
This means that
\[
T^*(f\chi_{\R^d\setminus Q^*})(x) = \sup_{\varepsilon\geq \frac{\lambda -1 }{2}l(Q)}\left| \int_{|x-y|>\varepsilon}K(x,y)f(y)\chi_{\R^d\setminus Q^*}(y)dy\right|,
\]
and in particular it follows that \(T^*(f\chi_{\R^d\setminus Q^*})(x)<\infty\) if we assume that \(f\in L^1(\R^d)\). If we now take any two points \(x', x''\in Q\), then we can say that
\[
|T^*(f\chi_{\R^d\setminus Q^*})(x')-T^*(f\chi_{\R^d\setminus Q^*})(x'')|
\]
is bounded by
\[
\sup_{\varepsilon \geq \frac{\lambda - 1}{2}l(Q)}\left|\int_{\substack{|x'-y|>\varepsilon \\ y\in \R^d\setminus Q^*}}K(x',y)f(y)dy-\int_{\substack{|x''-y|>\varepsilon \\ y\in \R^d\setminus Q^*}}K(x'',y)f(y)dy\right|.
\]
In turn, we can control this by
\[
    \begin{split}
        &\int_{\R^d\setminus Q^*}|K(x',y)-K(x'',y)||f(y)|dy \\
        &+ \sup_{\varepsilon \geq \frac{\lambda-1}{2}l(Q)}\int_{\R^d\setminus Q^*}|K(x'',y)||f(y)|\left| \1_{|x'-y|> \varepsilon} - \1_{|x''-y|> \varepsilon}\right|dy.
    \end{split}
\]
Assuming that \(\lambda > 4\sqrt{d}+1\), we can follow through with the argument in the proof of Proposition \ref{prop: condition for weak type auxiliary maximal op} to obtain
\[
\int_{\R^d\setminus Q^*}|K(x',y)-K(x'',y)||f(y)|dy \lesssim_d [\omega]_\text{Dini} Mf(x),
\]
so we just have to concentrate on the second term, which we call \(I\). First note that the integral can be reduced to the set \(\{y: |x'-y|>\varepsilon, |x''-y|\leq \varepsilon\}\cup\{y: |x'-y|\leq \varepsilon, |x''-y|>\varepsilon\}\). If \(y\) is in this set, then we have that
\[
|x-y|\leq \varepsilon + \text{diam}(Q)\leq \varepsilon + \sqrt{d}l(Q)\leq \varepsilon + \frac{2\sqrt{d}}{\lambda-1}\varepsilon \leq \frac{3}{2}\varepsilon.
\]
Moreover, if \(y\) is in the second set we have \(|x''-y|>\varepsilon\), and if \(y\) is in the first set we get
\[
\varepsilon < |x'-y|\leq |x'-x''| + |x''-y|\leq \sqrt{d}l(Q) + |x''-y| \leq \sqrt{d}\frac{2\varepsilon}{\lambda-1} + |x''-y| \leq \frac{\varepsilon}{2}+|x''-y|,
\]
so in any case we have
\[
|x''-y|\geq \frac{\varepsilon}{2}.
\]
Using both these estimates we see that
\[
\begin{split}
I&\leq C_K\sup_{\varepsilon\geq \frac{\lambda-1}{2}l(Q)}\int_{\R^d\setminus Q^*}\frac{|f(y)|}{|x''-y|^d}\left| \1_{|x'-y|> \varepsilon} - \1_{|x''-y|> \varepsilon}\right|dy \\
& \leq C_K\sup_{\varepsilon\geq \frac{\lambda-1}{2}l(Q)}\frac{2^{d+1}}{\varepsilon^d}\int_{|x-y|\leq \frac{3}{2}\varepsilon}|f(y)|dy\\
&\lesssim_d C_KMf(x).
\end{split}
\]
Putting these together we obtain
\[
M_{T^*, \lambda}f(x) \lesssim_d (C_K+[\omega]_\text{Dini}) Mf(x),
\]
and in turn this implies that we have the sparse estimate
\[
    |T^*f(x)|\lesssim_d (C_K + [\omega]_\text{Dini} + \|T\|_{L^2\rightarrow L^2})\A_\S f(x)\text{ for a.e. }x\in \R^d.\footnote{This estimate was first obtained by Lacey in \cite{laceyA2}, based on the earlier work \cite{lernerCZ} of Lerner.}
\]

\subsection{Maximally modulated Calderón-Zygmund operators}\label{sec: sparse domination for maximally mod CZ op}

The results from the previous section can be further improved to the setting of maximally modulated Calderón-Zygmund operators. 
The idea for the definition of these operators comes from the relation between the Carleson operator and the Hilbert transform. 
The Carleson operator is usually defined as 
\[
\mcal{C}f(x) = \sup_{N\in \R}\left|\int_{-\infty}^N \hat{f}(\xi) e^{2\pi i \xi x}d\xi\right|.
\]
However, we can relate this to the Hilbert transform in the following way. From the identity
\[
\chi_{]-\infty, N[} = \frac{1}{2}-\frac{1}{2}\text{sgn}(\xi- N),
\]
which holds for all \(\xi \in \R\setminus \{N\}\), we get
\[
\begin{split}
\int_{-\infty}^N \hat{f}(\xi)e^{2\pi i \xi x}d\xi &= \frac{1}{2}f(x) -\frac{1}{2}\int_\R \text{sgn}(\xi-N)\hat{f}(\xi)e^{2\pi i \xi x}d\xi \\
&= \frac{1}{2}f(x) - \frac{1}{2}\int_\R \text{sgn}(\xi)(M_{-N}f)^\wedge(\xi)e^{2\pi i \xi x}d\xi e^{2\pi i Nx}\\
&= \frac{1}{2}f(x) -\frac{i}{2}e^{2\pi i Nx} H(M_{-N}f)(x),
\end{split}
\]
where \(M_\eta f(x) = e^{2\pi i \eta x}f(x)\) is the modulation operator. 
This way we see that the boundedness properties of the Carleson operator are the same as the boundedness properties of the operator
\[
Cf(x):= \sup_{N\in \R} |H(M_Nf)(x)|,
\]
which we also call Carleson's operator.
This form of Carleson's operator suggests a more general idea of considering maximally modulated 
Calderón-Zygmund operators\index{maximally modulated Calderón-Zygmund operators}. 
Following \cite{lerner_diplinio}, we consider a family \(\mcal{F} = \{\phi_\beta\}_{\beta \in B}\) of real-valued measurable functions, 
the modulations \(M_{\phi_\beta}f(x) = e^{2\pi i \phi_\beta (x)}f(x)\) and the operator
\[
T_\mcal{F} f(x) = \sup_{\beta \in B}|T(M_{\phi_\beta}f)(x)|,
\]
where \(T\) is a given Calderón-Zygmund operator with modulus of continuity \(\omega\). 
Furthermore, we assume that there exists some \(1 \leq s < \infty\) such that
\begin{equation}\label{eq: weak type bound}
     \|T_\mcal{F}f\|_{L^{s,\infty}(\R^d)}\lesssim_{T,d, s} \|f\|_{L^s(\R^d)}.
\end{equation}
Now, if we argue as before, we can show that the auxiliary maximal operator is controlled by the Hardy-Littlewood maximal operator. Indeed, let \(x \in \R^d\) and let \(Q\) be a cube containing \(x\). First, we can say that
\[
T_\mcal{F}(f\chi_{\R^d\setminus Q^*})(x) = \sup_{\beta \in B}\left|\int_{\R^d\setminus Q^*}K(x,y)e^{2\pi i \phi_\beta(y)}f(y)dy\right|\leq \int_{\R^d\setminus Q^*}|K(x,y)||f(y)|dy<\infty,
\]
if we assume that \(f\in L^1(\R^d)\). Therefore, given any \(x', x''\in Q\),
\[
\begin{split}
    |T_\mcal{F}(f\chi_{\R^d\setminus Q^*})(x')-T_\mcal{F}(f\chi_{\R^d\setminus Q^*})(x'')|&\leq \sup_{\beta \in B}\left|\int_{\R^d\setminus Q^*}(K(x',y)-K(x'',y))e^{2\pi i \phi_\beta(y)}f(y)dy\right| \\
    &\leq \int_{\R^d\setminus Q^*}|K(x',y)-K(x'',y)||f(y)|dy \\
    &\lesssim_d [\omega]_\text{Dini}Mf(x),
\end{split}
\]
assuming that \(\lambda > 4\sqrt{d}+1\). This implies that
\[
M_{T_\mcal{F},\lambda}f(x) \lesssim_d [\omega]_\text{Dini}Mf(x).
\]
This fact, together with assumption \eqref{eq: weak type bound} and Theorem \ref{thm: lerner-ombrosi} implies that, 
given a function \(f\in L^s(\R^d)\) with compact support, there exists a \(1/(2\lambda^d)\)-sparse collection of cubes \(\S = \S(f,s)\) such that
\[
T_\mcal{F}f(x) \lesssim \sum_{Q\in \S}\chi_Q(x) \left(\dashint_Q |f|^s\right)^{1/s}\text{ for a.e. }x\in \R^d.
\]
where the implicit constant depends on \(d, \lambda, s, C_K\) and \([\omega]_\text{Dini}\). This sparse bound was first obtained by 
Di Plinio and Lerner in \cite{lerner_diplinio}.

If we now go back to considering the Carleson operator \(C\), then from the Carleson-Hunt theorem we know that
\[
    \|Cf\|_{L^{s,\infty}(\R)}\lesssim_s \|f\|_{L^s(\R)}, \text{ for all }1<s<\infty,
\]
so we can say that
\[
Cf(x) \lesssim_{s} \sum_{I\in \S}\chi_I(x) \left(\dashint_I |f|^s\right)^{1/s}.
\]
It is important to note here that the implicit constant in this estimate blows up as \(s\rightarrow 1\), and so we cannot have sparse domination here with \(s = 1\). In fact, such a sparse bound with \(s = 1\) would imply \(L^1 \rightarrow L^{1,\infty}\) boundedness (see Appendix B from \cite{conde-alonso}), which we know cannot hold for the Carleson operator. The reason for this is that a weak type \((1, 1)\) bound for the Carleson operator would imply pointwise almost everywhere convergence for \(L^1\) functions, but we know this to be false from Kolmogorov's example of an \(L^1\) function whose Fourier series diverges almost everywhere. 

Finally, we want to apply the same ideas to the maximal version of these operators. That is, we define
\[
T^*_\mcal{F}f(x) = \sup_{\beta\in B}|T^*(M_{\phi_\beta}f)(x)| = \sup_{\beta\in B}\sup_{\varepsilon>0}\left|\int_{|x-y|>\varepsilon}K(x,y)e^{2\pi i \phi_\beta(y)}f(y)dy\right|.
\]
Again, the strategy here is the same. First we note that if \(f\in L^1\) and \(x\in Q\), then \(T^*_\mcal{F}(f\chi_{\R^d\setminus Q^*})(x)<\infty\) and therefore, if \(x', x''\in Q\), then 
\[
|T^*_\mcal{F}(f\chi_{\R^d\setminus Q^*})(x')-T^*_\mcal{F}(f\chi_{\R^d\setminus Q^*})(x'')|
\]
can be controlled by
\[
\sup_{\beta\in B}\sup_{\varepsilon>0}\left|\int_{\substack{|x'-y|>\varepsilon\\ y\in \R^d\setminus Q^*}}K(x',y)e^{2\pi i \phi_\beta(y)}f(y)dy-\int_{\substack{|x''-y|>\varepsilon\\ y\in \R^d\setminus Q^*}}K(x'',y)e^{2\pi i \phi_\beta(y)}f(y)dy\right|.
\]
We can then continue the argument as in Section \ref{sec: sparse bounds for CZ op} to obtain
\[
M_{T^*_\mcal{F},\lambda}f(x) \lesssim_d (C_K + [\omega]_\text{Dini})Mf(x),
\]
for \(\lambda > 4\sqrt{d}+1\). The only thing left to show is that the operator \(T^*_\mcal{F}\) itself satisfies a weak type estimate. 
But this can be done by using Cotlar's inequality\footnote{Theorem 4.2.4 in \cite{Mgrafakos}.}
\[
T^*f(x)\lesssim Mf(x) + M(|Tf|^r)(x)^{1/r},\ 0<r<1.
\]
Indeed, using this inequality we get
\[
T^*_\mcal{F}f(x) \lesssim Mf(x) + M(|T_\mcal{F}f|^r)(x)^{1/r},
\]
from which it follows that
\[
\|T^*_\mcal{F}f\|_{L^{s,\infty}}\lesssim \|f\|_{L^s},
\]
by using the fact that the Hardy-Littlewood maximal operator sends \(L^{p,\infty}\) to \(L^{p,\infty}\) for every \(p\in ]1,\infty[\) and by using \eqref{eq: weak type bound}. This then gives sparse domination. In particular, if we define the maximal Carleson operator\index{maximal Carleson operator} as
\[
C^*f(x) = \sup_{N\in \R} \sup_{\varepsilon>0}\left|\int_{|x-y|>\varepsilon}\frac{f(y)}{x-y}e^{2\pi i Ny}dy\right|
\]
then, for any \(1<s<\infty\) and \(f\in L^s(\R)\) with compact support, 
\[
C^*f(x) \lesssim_s \sum_{I\in \S}\chi_I(x) \left(\dashint_I |f|^s\right)^{1/s}.
\]

\clearpage
\appendix
\section{Weighted Lebesgue Spaces}\label{appendix: weighted spaces}

In this appendix we will introduce the basic definitions of weighted Lebesgue spaces, including the most important
families of weights and some useful approximation results. We also prove some important properties in the 
two-weight setting.

\subsection{Basic definitions and approximation theorems}

We are interested in considering weighted modifications of the Lebesgue measure. We mostly work in \(\R^d\), however we will
sometimes need to consider also the group \((\R^+, \times)\) endowed with the natural Haar measure given by 
\[ 
    \lambda(E) = \int_E \frac{dx}{x}. 
\] 
For this reason we will consider we have a space \((G, \lambda)\) which will either be \(\R^d\) or \(\R^+\), equipped with the 
Lebesgue measure in the first case, and the previously mentioned Haar measure in the second case.\footnote{For a thorough exposition 
about Haar measures and Harmonic Analysis on locally compact Abelian groups see \cite{folland2}.}

We define a \textit{weight} as a measurable function \(w: G \rightarrow [0, +\infty]\) which is locally integrable and satisfies 
\(0 < w(x) < +\infty\) almost everywhere. A weighted space\index{weighted Lebesgue space} is then a space \(L^p(G, \mu)\), 
where the measure \(\mu\) is defined by 
\[ 
    \mu(E) = \int_E w(x)d\lambda(x),
\] 
for some weight \(w\). By abuse of notation, we usually refer to \(\mu\) as \(w\). This way, we often write \(w(E)\) to mean 
\(\int_Ew d \lambda\). As a consequence of the fact that \(0 < w(x) < \infty\) a.e., we see that \(\lambda(E) = 0\) if and 
only if \(w(E) = 0\). Since we also have \(\lambda(E) = 0\) if and only if \(|E| = 0\), this means that the Lebesgue zero measure 
sets are exactly the same as the \(w\) zero measure sets. In particular, properties which hold \(w\)-almost everywhere will hold 
almost everywhere and vice-versa. Also, any weighted measure will be \(\sigma-\)finite and \(L^\infty(G, w)\) coincides with \(L^\infty(G, dx)\). 
In the next proposition we show that any weighted measure is a Radon measure.\footnote{Recall
that a measure is called a Radon measure if it is finite on compact sets, outer 
regular on Borel sets and inner regular on all open sets.}

\begin{prop}\label{prop: weights are radon}
    If \(w\) is a weight, then the measure \(w(x)d\lambda(x)\) is a Radon measure.
\end{prop}
\begin{proof}
    Since \(w\) is locally integrable it is clear that it assigns a finite measure to compact sets. We start by 
    showing inner regularity over all open sets. First suppose that \(E\subseteq G\) is a bounded open set. Using the inner 
    regularity of the Haar measure, we know that for each \(n\in\N_1\) there is some compact set \(K_n\subseteq E\) such that 
    \[ 
        \lambda(E) - \frac{1}{n} \leq \lambda(K_n).
    \] 
    Note that 
    \[ 
        \|\1_{E\setminus K_n}\|_{L^1(G, \lambda)} = \lambda(E\setminus K_n) \leq \frac{1}{n}\xrightarrow[n\rightarrow +\infty]{} 0.
    \] 
    Therefore, there is some subsequence \((n_l)_l\) such that 
    \[ 
        \1_{E\setminus K_{n_l}} \xrightarrow[l\rightarrow +\infty]{} 0 \text{ pointwise almost everywhere}.
    \] 
    But then, by the dominated convergence theorem,
    \[
        \begin{split}
        \int_E wd\lambda &= \int_{E\setminus K_{n_l}}wd\lambda + \int_{K_{n_l}}wd\lambda \\
        &\leq \int_E w\1_{E\setminus K_{n_l}}d\lambda + \sup_{\substack{K\subseteq E \\ K\text{ compact}}}w(K)
        \xrightarrow[l\rightarrow +\infty]{} \sup_{\substack{K\subseteq E \\ K\text{ compact}}}w(K).
        \end{split}
    \] 
    To apply the dominated convergence theorem we used the fact that \(w\in L^1(E, \lambda)\), which holds because \(w\) 
    is locally integrable and \(E\) is bounded. Now let \(E\) be any open set. Fix \(\varepsilon > 0\) and put \(E_N = E \cap B_N\), 
    where \(B_N\) is \(\{x\in \R^d: |x| < N\}\) if \(G = \R^d\) or \(\{x \in \R^+: 1/(N+1) < x < N+1\}\) if \(G = \R^+\). Since 
    \(E_N\) is open and bounded we know that there is some compact subset \(K_N\) such that 
    \[ 
        w(E_N) \leq w (K_N) + \frac{\varepsilon}{2^N}. 
    \] 
    Given that the sequence \((E_N)\) increases to \(E\), we get
    \[ 
        w(E) = \lim_N w(E_N) \leq \sup_{\substack{K\subseteq E \\ K\text{compact}}}w(K).
    \] 
    This establishes inner regularity. 

    Next we will show that \(w\) is outer regular. As before, we start by assuming we have some bounded Borel set \(E\subseteq G\). 
    We want to use the fact that \(\lambda\) is outer regular to obtain open sets \(U\) which approximate closely the measure of 
    \(E\). However, to apply the dominated convergence theorem as before we need to make sure that these open sets are all contained 
    in a large compact set. This technical detail will force us to follow a more elaborate argument, but at its heart the main idea 
    is similar to the proof of inner regularity. 

    Let \(K_N\) denote the set \([-N, N]^d\) if \(G = \R^d\) or the set \(\{x>0: 1/(N+1) \leq x \leq N+1\}\) if \(G = \R^+\). Since \(E\) 
    was assumed to be bounded, then we have \(E \subseteq K_N\) for some \(N\) large enough. Since \(\lambda\) is a Lebesgue-Stieltjes
    measure, we know that for each \(n\in \N_1\), there exists a countable collection of rectangles \(\{T^{(n)}_j\}_j\) such that 
    \[ 
        E \subseteq \bigcup_{j\geq 1}T_j^{(n)}\text{ and }\sum_{j\geq 1}\lambda(T_j^{(n)}) \leq \lambda(E) + \frac{1}{n}.
    \] 
    Without loss of generality we can assume that all rectangles \(T_j^{(n)}\) are contained in \(K_N\), otherwise we could 
    replace \(T_j^{(n)}\) with \(T_j^{(n)}\cap K_N\). Now, for each \(j\), we can find an open rectangle \(U_j^{(n)}\) such that 
    \(T_j^{(n)}\subseteq U_j^{(n)}\), \(\lambda(U_j^{(n)})\leq \lambda(T_j^{(n)}) + 2^{-j}/n\) and \(U_j^{(n)}\subseteq K_{N+1}\).
    Put \(U_n = \cup_jU_j^{(n)}\). Clearly, \(E\subseteq U_n\subseteq K_{N+1}, \forall n\), \(U_n\) is open, and 
    \[ 
        \lambda(U_n) \leq \sum_j \lambda(U_j^{(n)}) \leq \sum_j \lambda(T_j^{(n)}) + \frac{2^{-j}}{n} \leq \lambda(E) + \frac{2}{n}.
    \] 
    Again, 
    \[ 
        \|\1_{U_n\setminus E}\|_{L^1(G, \lambda)}\leq \frac{2}{n} \xrightarrow[n\rightarrow +\infty]{} 0,
    \] 
    so we can find a subsequence \((n_l)_l\) such that 
    \[ 
        \1_{U_n\setminus E}(x) \xrightarrow[l\rightarrow +\infty]{}0\text{ pointwise almost everywhere}.
    \] 
    By the dominated convergence theorem,
    \[ 
        w(U_{n_l}\setminus E) = \int_{K_{N+1}}w\1_{U_{n_l}\setminus E}d\lambda \xrightarrow[l\rightarrow +\infty]{} 0,
    \] 
    because \(w\in L^1(K_{N+1}, \lambda)\). Thus,
    \[ 
        \inf_{\substack{U\supseteq E \\ U\text{ open}}}w(U) \leq w(U_{n_l})\leq w(U_{n_l}\setminus E) + w(E)\xrightarrow[l\rightarrow +\infty]{} w(E).
    \] 
    Finally, let's consider any Borel set \(E\). For each \(N\in \N_1\), we put \(E_N = E \cap K_N\). Fix \(\varepsilon > 0\). 
    From what we've just seen, for each \(N\), there is an open and bounded set \(U_N\supseteq E_N\) such that 
    \[ 
        w(U_N) \leq w(E_N) + \frac{\varepsilon}{2^N} \leq w(E) + \frac{\varepsilon}{2^N}.
    \] 
    Set \(U = \cup_N U_N\). Then, \(U\supseteq E\) is open, and 
    \[ 
        \begin{split}
            w(U) &= w(\cup_N U_N) \leq \sum_Nw(U_N\setminus E_N) + w(\cup_N E_N) \leq \sum_N \frac{\varepsilon}{2^N} + w(E)
            \leq w(E) + \varepsilon.
        \end{split}
    \]
    Since \(\varepsilon>0\) is arbitrary, this shows that
    \[ 
        w(E) = \inf_{\substack{U\supseteq E \\ U\text{ open}}}w(U).
    \]
\end{proof}

It is a general fact that on any locally compact Hausdorff space, the set of compactly supported 
continuous functions are dense in \(L^p\) spaces when the measure is a Radon measure. 
\begin{thm}[Proposition 7.9 from \cite{folland}]
    Let \(X\) be a locally compact Hausdorff space and \(\mu\) a Radon measure on \(X\). Then, the set
    \(C_c(X)\) is dense in \(L^p(X, \mu)\) for all \(p \in [1, +\infty[\).
\end{thm}

We now use this theorem to show that test functions are dense in weighted Lebesgue spaces.
\begin{prop}\label{prop: test functions are dense in weighted spaces}
    Let \(w\) be a weight on \(G\). Then, the set \(C^\infty_c(G)\) is dense in \(L^p(G, w)\), where \(1\leq p < \infty\).
\end{prop}
\begin{proof}
    Let \(f\in L^p(G, w)\) and fix \(\delta > 0\). Since \(w\) is a Radon measure, we know that there is some 
    \(g \in C_c(G)\) such that 
    \[ 
        \|f - g\|_{L^p(G, w)} < \frac{\delta}{2}.
    \] 
    Now let \(\phi \in C^\infty_c(G)\) be a positive function such that 
    \[ 
        \int_G\phi(x) d\lambda(x) = 1,
    \]
    and then define rescaled versions \(\phi_\varepsilon\) in the following way. If \(G = \R^d\) we put 
    \[ 
        \phi_\varepsilon(x) = \frac{1}{\varepsilon^d}\phi\left(\frac{x}{\varepsilon}\right).
    \] 
    If instead \(G = \R^+\), then we set 
    \[ 
        \phi_\varepsilon(x) = \frac{1}{\varepsilon}\phi(x^{1/\varepsilon}). 
    \] 
    With these definitions, \(\phi_\varepsilon \in C^\infty_c(G)\) and \(\{\phi_\varepsilon\}\) forms 
    an approximate identity. It then follows that \(g_\varepsilon := g * \phi_\varepsilon \in C^\infty_c(G)\) and 
    moreoever we know from standard approximate identity results that \(g_\varepsilon \rightarrow g\) in \(L^\infty(K)\)
    for any compact set \(K\), as \(\varepsilon \rightarrow 0^+\).\footnote{See for instance Theorem
    1.2.19 from \cite{Cgrafakos}.} Since both \(g\) and \(g_\varepsilon\) have compact support, and the support of 
    \(g_\varepsilon\) is close to the support of \(g\) when \(\varepsilon\) is small enough, we can find some compact 
    set \(K\) which contains \(\supp(g)\) and \(\supp(g_\varepsilon)\) for all small enough \(\varepsilon\). 
    Since \(g_\varepsilon \rightarrow g\) in \(L^\infty(K)\), we can pick \(\varepsilon\) such that 
    \[ 
        \|g - g_\varepsilon\|_{L^\infty(K, \lambda)}w(K)^{1/p} < \frac{\delta}{2}.
    \] 
    We then obtain 
    \[ 
        \begin{split}
        \|f - g_\varepsilon\|_{L^p(G, w)}&\leq \|f - g\|_{L^p(G,w)} + \|g - g_\varepsilon\|_{L^p(G, w)} \\ 
        &\leq \frac{\delta}{2} + \left(\int_K|g(x) - g_\varepsilon(x)|^pw(x)d\lambda(x)\right)^{1/p} \\ 
        &\leq \frac{\delta}{2} + \|g - g_\varepsilon\|_{L^\infty(K, \lambda)}w(K)^{1/p} \leq \delta. 
        \end{split}
    \] 

\end{proof}
We finish this section by showing that the same result holds in \(\R^d\) in the radial case.

\begin{prop}\label{prop: radial test functions are dense in radial weighted spaces}
    Let \(1\leq p < \infty\) and let \(w\) be a weight in \(\R^d\). If \(f\in L^p(\R^d, w)\) is radial, then there is a sequence
    \((f_n)_n\) of radial \(C^\infty_c(\R^d)\) functions such that 
    \[ 
        \|f - f_n\|_{L^p(\R^d, w)} \xrightarrow[n\rightarrow +\infty]{} 0.
    \] 
\end{prop}
\begin{proof}
    Let \(f\in L^p(\R^d, w)\) be a radial function with radial projection \(f_0\), that is, \(f(x) = f_0(|x|)\). Then, 
    \[ 
        \|f\|_{L^p(\R^d, w)}^p = \int_{\R^d}|f(x)|^pw(x)dx = \int_0^\infty|f_0(\rho)|^p\rho^{d}\int_{S^{d-1}}w(\rho u)d\sigma_u 
        \frac{d\rho}{\rho}. 
    \] 
    This means that 
    \[ 
        \|f\|_{L^p(\R^d, w)} = \|f_0\|_{L^p(\R^+, v)},\text{ where }v(\rho) = \rho^d \int_{S^{d-1}}w(\rho u)d\sigma_u. 
    \] 
    Since \(f_0\in L^p(\R^+, v)\), we can apply the previous proposition to conclude that there is a sequence \((f_0^{(n)})_n\) of 
    \(C^\infty_c(\R^+)\) functions such that \(f_0^{(n)}\rightarrow f_0\) in \(L^p(\R^+, v)\). If we then define 
    \(f_n(x) = f_0^{(n)}(|x|)\), we get a sequence of radial \(C^\infty_c(\R^d)\) functions and
    \[ 
        \|f_n - f\|_{L^p(\R^d, w)} = \|f_0^{(n)} - f_0\|_{L^p(\R^+, v)} \xrightarrow[n\rightarrow +\infty]{}0.
    \]
\end{proof}

\subsection{Muckenhoupt weights}\label{appendix: Muckenhoupt weights}

%

In this section we will introduce the basic definitions and properties of Muckenhoupt weights.\index{Muckenhoupt weights}
Here we will focus on weights belonging to the classes \(A_p\), \(A_\infty\) and \(A_{p, p}\). 

We say that a weight \(w\) on \(\R^d\) is in \(A_p\)\index{\(A_p\) weight} if
\[ 
    [w]_{A_p} = \sup_Q \left(\dashint_Q w^{-p'/p}\right)^{p/p'}\dashint_Q w < \infty,
\] 
where \(1 < p < \infty\) and the supremum runs through all cubes with sides parallel to the coordinate axes.
\begin{prop}\label{prop: basic properties of Ap weights}
    The following staments hold:
    \begin{enumerate}
        \item If \(w\in A_p\), then \(w\in A_q\) for any \(q\geq p\). In fact, \([w]_{A_q}\leq [w]_{A_p}\). 
        \item \(w\in A_p\) if and only if \(w^{-p'/p}\in A_{p'}\) and moreover \([w^{-p'/p}]_{A_{p'}} = [w]_{A_p}^{p'/p}\).
    \end{enumerate}
\end{prop}
\begin{proof}
    Let \(1 < p < q < \infty\). To prove the first statement we simply use Hölder's inequality with the exponent \(t = (q-1)/(p-1) \in ]1, \infty[\).
    We get,
    \[ 
        \dashint_Qw^{-q'/q}dx \leq \left(\dashint_Q w^{-p'/p}\right)^\frac{p-1}{q-1}.
    \] 
    Therefore,
    \[ 
        \left(\dashint_Q w^{-q'/q}\right)^{q-1}\leq \left(\dashint_Q w^{-p'/p}\right)^{p-1},
    \] 
    from which we immediately deduce that \([w]_{A_q}\leq [w]_{A_p}\). To prove the second statement simply note that 
    \[ 
        [w^{-p'/p}]_{A_{p'}} = \sup_Q\left(\dashint_Q w\right)^{p'-1}\dashint_Q w^{-p'/p} = [w]_{A_p}^{p'/p}.
    \] 
\end{proof}

From this proposition we see that the \(A_p\) classes increase with \(p\). It is interesting to think about the limit 
as \(p\rightarrow +\infty\). By Jensen's inequality we have that 
\[ 
    \log\left(\dashint_Q w^{-p'/p}\right)^{p/p'} = \frac{p}{p'}\log\left(\dashint_Q w^{-p'/p}\right) \geq \dashint_Q\log(w^{-1}),
\] 
so
\[ 
    \exp\left(\dashint_Q\log(w^{-1})\right) \leq \left(\dashint_Q w^{-p'/p}\right)^{p/p'}.
\] 
Moreover, the right-hand side converges to the left-hand side as \(p\rightarrow +\infty\).
\begin{lemma}
    Let \(w\in A_{p_0}\) for some \(p_0 \in ]1, \infty[\). Then, 
    \[ 
        \lim_{p \rightarrow +\infty}\left(\dashint_Q w^{-p'/p}\right)^{p/p'} = \exp\left(\dashint_Q\log(w^{-1})\right). 
    \]
\end{lemma}
\begin{proof}
    We have that 
    \[ 
        \dashint_Q\log(w^{-1}) \leq \frac{p}{p'}\log\left(\dashint_Q w^{-p'/p}\right)\leq \frac{p}{p'}\dashint_Q(w^{-p'/p} - 1).
    \] 
    If we put \(q = p'/p\) and \(f(x) = w(x)^{-1}\), then we see that it suffices to prove that 
    \[ 
        \lim_{q \rightarrow 0^+}\dashint_Q\frac{f(x)^q - 1}{q}dx = \dashint_Q \log(f). 
    \] 
    Set \(Q_1 = \{x\in Q: f(x) \leq 1\}\) and \(Q_2 = \{x\in Q: f(x) > 1\}\) and note that the function \((f^q-1)/q\) is 
    increasing in \(q\). Since \(w\in A_{p_0}\), then \(f\in L^{q_0}(Q)\), where \(q_0 = 1/(p_0-1)\). Therefore, if \(x\in Q_2\)
    \[ 
        \frac{|f(x)^q - 1|}{q} = \frac{f(x)^q - 1}{q} \leq \frac{f(x)^{q_0} - 1}{q_0} \in L^1(Q_2)
    \] 
    for \(0 < q < q_0\), and 
    \[ 
        \lim_{q\rightarrow 0^+}\frac{f(x)^q - 1}{q} = \log(f(x)).
    \] 
    So, by the dominated convergence theorem 
    \[ 
        \lim_{q\rightarrow 0^+}\dashint_{Q_2}\frac{f(x)^q - 1}{q} = \dashint_{Q_2} \log(f).
    \] 
    If \(x \in Q_1\), then \(|f(x)^q-1| = 1-f(x)^q\). Since \((1-f(x)^q)/q\) is increasing as \(q\rightarrow 0\), then 
    by the monotone convergence theorem 
    \[ 
        \lim_{q\rightarrow 0^+}\dashint_{Q_1}\frac{1-f(x)^q}{q}dx = -\dashint_{Q_1}\log(f).
    \] 
    Adding the two limits we get the desired result.
\end{proof}

This motivates the following definition: we say a weight \(w\) belongs to the class \(A_\infty\)\index{\(A_\infty\) weight} if 
\[ 
    [w]_{A_\infty} = \sup_Q \exp\left(\dashint_Q\log(w^{-1})\right)\dashint_Q w < \infty. 
\] 
From above we know that \([w]_{A_\infty} \leq [w]_{A_p}\), so any \(A_p\) weight is also an \(A_\infty\) weight. 
One of the most important results in the theory of Muckenhoupt weights is that a partial converse is true: if \(w\in A_\infty\), 
then there is some \(p\) large enough so that \(w \in A_p\). A fundamental building block in this direction is 
the reverse Hölder property\index{reverse Hölder inequality} of \(A_\infty\) weights.

\begin{thm}[Reverse Hölder, Theorem 7.2.2 from \cite{Cgrafakos}]
    Let \(w \in A_p\) for some \(1 < p < \infty\). Then, there exists some \(\gamma > 0\) such that 
    \[ 
        \left(\dashint_Qw^{1+ \gamma}\right)^\frac{1}{1 + \gamma} \lesssim \dashint_Q w,
    \] 
    where the implicit constant depends only on \(d, p\) and \([w]_{A_p}\), and not on \(Q\).
\end{thm}

This property has a number of important consequences. 
\begin{prop}\label{prop: important Ap properties}
    The following statements hold:
    \begin{enumerate}
        \item Suppose that \(w\in A_p\), for some \(1 < p < \infty\). Then, there is some \(\gamma_0 > 0\) such that 
        \[ 
            w^{1 + \gamma} \in A_p,\ \forall \gamma \in [0, \gamma_0[.
        \] 
        \item If \(w \in A_p\), \(1 < p < \infty\), then there is some \(q_0 < p\) such that 
        \[ 
            w \in A_q,\ \forall q \in ]q_0, p].
        \] 
    \end{enumerate}
\end{prop}
\begin{proof}
    Since \(w\in A_p\) we know that there is some \(\gamma_1 > 0\) such that 
    \[ 
        \left(\dashint_Q w^{1 + \gamma_1}\right)^\frac{1}{1+\gamma_1} \lesssim \dashint_Q w. 
    \] 
    Note also that for any \(0 < \gamma < \gamma_1\) we can use Hölder's inequality to obtain 
    \[ 
        \left(\dashint_Q w^{1 + \gamma}\right)^\frac{1}{1 + \gamma}\leq \left(\dashint_Q w^{1 + \gamma_0}\right)^\frac{1}{1+\gamma_0}. 
    \] 
    Given that \(w\in A_p\), it follows from Proposition \ref{prop: basic properties of Ap weights} that 
    \(w^{-p'/p}\in A_{p'}\). But then, there is some \(\gamma_2 > 0\) such that 
    \[ 
        \left(\dashint_Q (w^{-p'/p})^{1+ \gamma_2}\right)^\frac{1}{1+\gamma_2}\lesssim \dashint_Q w^{-p'/p}.
    \] 
    Now take any \(\gamma\) such that \(0 < \gamma < \gamma_0 := \min\{\gamma_1, \gamma_2\}\). Then,
    \[ 
        \begin{split}
            \left(\dashint_Q(w^{1+\gamma})^{-p'/p}\right)^{p-1}\dashint_Qw^{1+\gamma} & \lesssim 
            \left(\dashint_Q w^{-p'/p}\right)^{(1+\gamma)(p-1)} \left(\dashint_Q w\right)^{1+\gamma}\lesssim [w]_{A_p}^{1+\gamma},
        \end{split}
    \] 
    which shows that \(w^{1+\gamma}\in A_p\). 

    Next we show statement 2. From what we've just seen, since \(w\in A_p\) we know that there is some 
    \(\gamma>0\) such that \(w^{1+\gamma}\in A_p\). Now pick \(q = (p+\gamma)/(1+\gamma)\) and note that 
    \(1 < q < p\). Moreover,
    \[ 
        \frac{q'}{q} = \frac{1+\gamma}{p-1}.
    \] 
    Therefore, 
    \[ 
        \left(\dashint_Q w^{-q'/q}\right)^{q-1}\dashint_Q w \leq \left(\dashint_Q w^{-\frac{1+\gamma}{p-1}}\right)^\frac{p-1}{1+\gamma} 
        \left(\dashint_Q w^{1+\gamma}\right)^\frac{1}{1+\gamma} \leq [w^{1+\gamma}]_{A_q}^\frac{1}{1+\gamma},
    \] 
    where in the first step we used Hölder's inequality. This shows that \(w\in A_q\) for some \(q < p\). 
    The result now follows from the first statement of Proposition \ref{prop: basic properties of Ap weights}. 
\end{proof}

As it turns out the Reverse Hölder property is equivalent to the \(A_\infty\) condition.
\begin{prop}[Theorem 7.3.3, (c), (e), (f) from \cite{Cgrafakos}]\label{prop: properties of A_infty}
    The following statements are equivalent:
    \begin{enumerate}
        \item \(w \in A_\infty\);
        \item \(w\) satisfies a Reverse Hölder estimate;
        \item There exist \(0 < \alpha', \beta' < 1\) such that for all cubes \(Q\) and measurable 
        subsets \(A\subseteq Q\), 
        \[ 
            w(A) < \alpha' w(Q) \implies |A| < \beta'|Q|;
        \] 
        \item There is some \(p\in ]1, +\infty[\) such that \(w \in A_p\).
    \end{enumerate}
\end{prop}

As a consequence of this proposition we get the following important property of Muckenhoupt weights:
\[ 
    A_\infty = \bigcup_{1 < p < \infty} A_p. 
\] 
There is also another way of characterizing \(A_\infty\) weights in terms of their relative order of growth.
\begin{prop}\label{prop: order of growth of Ainfty}
    A weight \(w\) belongs to \(A_\infty\) if and only if there is some increasing function \(f:[0, \infty[\rightarrow [0, \infty[\) 
    such that\footnote{This function may depend on the weight.} 
    \[ 
        \frac{w(Q)}{w(A)} \leq f\left(\frac{|Q|}{|A|}\right),
    \] 
    for all cubes \(Q\) and all positive measure subsets \(A\subseteq Q\).
\end{prop}
\begin{proof}
    First suppose that \(w\in A_\infty\). Then, we there is some \(p\in ]1, \infty[\) such that \(w \in A_p\). But then, 
    \[ 
        \begin{split}
            |A|& = \int_Q \1_A w^{1/p} w^{-1/p} \leq \left(\int_Q \1_A w\right)^{1/p}\left(\int_Q w^{-p'/p}\right)^{1/p'} \\
            &\leq w(A)^{1/p}[w]_{A_p}^{1/p} |Q| \left(\int_Q w\right)^{-1/p}. 
        \end{split}
    \]
    So, 
    \[ 
        \frac{w(Q)}{w(A)}\leq [w]_{A_p}\left(\frac{|Q|}{|A|}\right)^p,
    \] 
    which shows the desired statement with \(f(t) = [w]_{A_p}t^p\).
    
    Now let's prove the converse direction. Suppose that there is some increasing function \(f\) which satisfies
    \[ 
        \frac{w(Q)}{w(A)} \leq f\left(\frac{|Q|}{|A|}\right).
    \] 
    Pick
    \[ 
        \alpha' = \frac{1}{\max\{f(2), 2\}}. 
    \]
    Note that \(0 < \alpha' < 1\). Now suppose that \(|A| \geq |Q|/2\). Then, 
    \[ 
        w(Q) \leq f\left(\frac{|Q|}{|A|}\right)w(A) \leq f(2)w(A) \leq \frac{w(A)}{\alpha'}.
    \] 
    This shows that 
    \[ 
        w(A) < \alpha' w(Q) \implies |A| < \frac{1}{2}|Q|.
    \] 
    From Proposition \ref{prop: properties of A_infty}, statement 3, it follows that \(w\in A_\infty\).
\end{proof}

From Proposition \ref{prop: properties of A_infty} we may conclude that if \(w \in A_\infty\), then there are \(C_1, C_2, \delta_1, \delta_2 > 0\) such that,
for all cubes \(Q\) and for all positive measure subsets \(A\subseteq Q\), we have 
\[ 
    C_1 \left(\frac{|Q|}{|A|}\right)^{\delta_1}\leq \frac{w(Q)}{w(A)}\leq C_2 \left(\frac{|Q|}{|A|}\right)^{\delta_2}. 
\] 
Indeed, since \(w\in A_\infty\), then there is some \(p \in ]1, +\infty[\) such that \(w \in A_p\). This way, 
\[ 
    \begin{split}
        |A| &= \int_Q \chi_A w^{1/p}w^{-1/p} \\
            &\leq w(A)^{1/p} |Q|^{1/p'}\left(\dashint_Q w^{-p'/p}\right)^{1/p'} \\ 
            &\leq [w]_{A_p}^{1/p}|Q|\left(\frac{w(A)}{w(Q)}\right)^{1/p},
    \end{split}
\]
which shows that 
\[ 
    \frac{w(Q)}{w(A)} \leq [w]_{A_p}\left(\frac{|Q|}{|A|}\right)^p. 
\] 
For the other inequality we argue using the reverse Hölder estimate,
\[ 
    \begin{split}
        w(A) &= |A| \dashint_A w \\
             &\leq |A|^\frac{\gamma}{1+\gamma}\left(\int_A w^{1+\gamma}\right)^\frac{1}{1+\gamma} \\
             &\leq |A|^\frac{\gamma}{1+\gamma}|Q|^\frac{1}{1+\gamma}\left(\dashint_Q w^{1+\gamma}\right)^\frac{1}{1+\gamma}\\
             &\leq C |A|^\frac{\gamma}{1+\gamma}|Q|^\frac{1}{1+\gamma}\dashint_Q w. 
    \end{split}
\]
Thus, 
\[ 
    \frac{w(Q)}{w(A)} \geq \frac{1}{C}\left(\frac{|Q|}{|A|}\right)^\frac{\gamma}{1+\gamma},
\] 
which shows the other estimate.

Some of these results are still true when we consider two-weight conditions. Given \(1<p<\infty\) and two weights \(v, w\),
we say that the pair \((v, w)\) satisfies the \(A_{p, p}\) condition if 
\[ 
    [v, w]_{A_{p, p}} = \sup_Q \left(\dashint_Q v^{-p'/p}\right)^{1/p'} \left(\dashint_Q w\right)^{1/p} < \infty. 
\] 
We start by observing that \(A_{p, p} \subseteq A_{q, q}\) when \(q > p\). Indeed, by Hölder's inequality,
\[ 
    \dashint_Q v^{-\frac{1}{q-1}} \leq \left(\dashint_Q (v^{-\frac{1}{q-1}})^\frac{q-1}{p-1}\right)^\frac{p-1}{q-1}. 
\] 
Therefore, 
\begin{equation}\label{eq: nested App}
    [v, w]_{A_{q, q}}^q \leq [v, w]_{A_{p, p}}^p.
\end{equation} 
An important observation here is that under the extra assumption that \(v^{-p'/p}, w\in A_\infty\), then these pairs 
of weights exhibit Reverse Hölder-like behavior. 
\begin{lemma}
    Let \(1 < p < \infty\) and suppose \(v, w\) are two weights such that \((v, w) \in A_{p, p}, v^{-p'/p}, w\in A_\infty\). Then, 
    there exists some \(\gamma_0 > 0\) such that 
    \[ 
        (v^{1 + \gamma}, w^{1 + \gamma})\in A_{p, p},\ \forall \gamma \in [0, \gamma_0]. 
    \] 
\end{lemma}
\begin{proof}
    Since \(v^{-p'/p}\) and \(w\) belong to \(A_\infty\), we know that there exist \(\gamma_1, \gamma_2 > 0\) such that 
    \begin{align*} 
        \left(\dashint_Q (v^{-p'/p})^{1+\gamma_1}\right)^\frac{1}{1+\gamma_1} &\lesssim \dashint_Q v^{-p'/p} \\ 
        \left(\dashint_Q w^{1+\gamma_0}\right)^\frac{1}{1+\gamma_0} &\lesssim \dashint_Q w \\ 
    \end{align*}
    So, if \(0 < \gamma < \min\{\gamma_1, \gamma_2\}\), then 
    \[ 
        \left(\dashint_Q (v^{1+\gamma})^{-p'/p}\right)^{1/p'}\left(\dashint_Q w^{1+\gamma}\right)^{1/p} \lesssim 
        \left(\dashint_Q v^{-p'/p}\right)^\frac{1 + \gamma}{p'} \left(\dashint_Q w\right)^\frac{1+\gamma}{p} \lesssim 
        [v, w]_{A_{p, p}}^{1+\gamma}.
    \] 
\end{proof}
We can now prove the final result of this section.
\begin{prop}\label{prop: reverse holder-like for App}
    Let \(1 < p < \infty\) and suppose that \(v, w\) are weights such that \((v, w)\in A_{p, p}, v^{-p'/p}, w \in A_\infty\). Then, 
    there exists some \(q_0 \in ]1, p[\) such that 
    \[ 
        (v, w) \in A_{q, q},\ \forall q \in ]q_0, p].
    \]
\end{prop}
\begin{proof}
    The argument is again very similar to the proof of Proposition \ref{prop: important Ap properties}. From the previous lemma, we 
    know that there exists some \(\gamma > 0\) such that \((v^{1+\gamma}, w^{1+\gamma}) \in A_{p, p}\). As before we choose 
    \(q = (p+\gamma)/(1+\gamma)\). Then, 
    \[ 
        \left(\dashint_Q v^{-q'/q}\right)^{q-1} \dashint_Q w \leq \left(\dashint_Q (v^{1+\gamma})^{-p'/p}\right)^\frac{p-1}{1+\gamma} 
        \left(\dashint_Q w^{1+\gamma}\right)^\frac{1}{1+\gamma} \leq [v^{1+\gamma}, w^{1+\gamma}]_{A_{p, p}}^\frac{p}{1+\gamma}. 
    \] 
    This implies that \((v, w)\in A_{q, q}\), from which the result follows.
\end{proof}

\clearpage
\addcontentsline{toc}{section}{References}
\fancyhead[R]{}
\fancyhead[L]{References}
\printbibliography[category=cited]

@book{Cgrafakos,
    author = {L. Grafakos},
    title = {Classical Fourier Analysis},
    year = {2014},
    publisher = {Springer}
}

@book{Mgrafakos,
    author = {L. Grafakos},
    title = {Modern Fourier Analysis},
    year = {2014},
    publisher = {Springer}
}

@book{Duo,
    author = {J. Duoandikoetxea},
    title = {Fourier Analysis},
    year = {2001},
    publisher = {Grad. Studies in Math. 29, American Mathematical Society}
}

@book{folland,
    author = {G. Folland},
    title = {Real Analysis: Modern Techniques and Their Applications},
    year = {1999},
    publisher = {John Wiley \& Sons},
    series={},
    volume = {},
    pages = {}
}

@book{folland2,
    author = {G. Folland},
    title = {A Course in Abstract Harmonic Analysis},
    year = {1995},
    publisher = {CRC Press},
    series={},
    volume = {},
    pages = {}
}

@inproceedings{Hunt,
  title={On the convergence of Fourier series},
  author={R. Hunt},
  booktitle = {Proc. Conf. Edwardsville},
  series = {III},
  year={1968},
  volume = {1},
  pages = {235-255}
}

@inproceedings{Uribe,
  title={Two weight inequalities for fractional integral operators and commutators},
  author={D. Cruz-Uribe},
  booktitle = {Advanced Courses of Mathematical Analysis},
  series = {Proceedings of the Sixth International School},
  year={2017},
  volume = {VI},
  publisher = {World Scientific},
  pages = {25-85}
}

@inproceedings{antonov,
  title={Convergence of Fourier series},
  author={N. Antonov},
  booktitle = {Proceedings of the XX Workshop on Function Theory},
  series = {},
  year={1996},
  volume = {2},
  pages = {187-196}
}

@article{muckenhoupt,
    author = {B. Muckenhoupt},
    title = {Weighted Norm Inequalities for the Hardy Maximal Function},
    journaltitle = {Trans. Am. Math. Soc.},
    year = {1972},
    volume = {165},
    number = {},
    pages = {207-226}
}

@article{lacey,
    author = {M. Lacey},
    title = {Carleson's Theorem: proof, complements, variations},
    journaltitle = {Publ. Mat.},
    year = {2004},
    volume = {48},
    number = {2},
    pages = {251-307}
}

@article{intuitive_dyadic_calculus,
    author = {A. Lerner and F. Nazarov},
    title = {Intuitive dyadic calculus: the basics},
    journaltitle = {Expo. Math.},
    year = {2019},
    volume = {37},
    number = {3},
    pages = {225-265}
}

@article{note_dyadic_coverings,
    author = {J. Conde},
    title = {A note on dyadic coverings and nondoubling Calderón-Zygmund theory},
    journaltitle = {J. Math. Anal. Appl.},
    year = {2013},
    volume = {397},
    number = {2},
    pages = {785-790}
}

@article{lerner,
    author = {A. Lerner},
    title = {On the Pointwise Estimates involving Sparse Domination},
    journaltitle = {New York J. Math.},
    year = {2016},
    volume = {22},
    number = {},
    pages = {341-349}
}

@article{lerner_diplinio,
    author = {A. Lerner and F. Di Plinio},
    title = {On Weighted norm inequalities for the Carleson and Walsh-Carleson operator},
    journaltitle = {J. London Math. Soc.},
    year = {2014},
    volume = {90},
    number = {3},
    pages = {654-674}
}

@article{prestini,
    author = {E. Prestini},
    title = {Almost Everywhere Convergence of the Spherical Partial Sums for Radial Functions},
    journaltitle = {Mh. Math.},
    year = {1988},
    volume = {105},
    number = {},
    pages = {207-216}
}

@article{fefferman,
    author = {C. Fefferman},
    title = {The Multiplier Problem for the Ball},
    journaltitle = {Ann. Math.},
    year = {1971},
    volume = {94},
    number = {2},
    pages = {330-336}
}

@article{laceyA2,
    author = {M. Lacey},
    title = {An elementary proof of the \(A_2\) bound},
    journaltitle = {Isr. J. Math.},
    year = {2017},
    volume = {217},
    number = {},
    pages = {181-195}
}

@article{lernerCZ,
    author = {A. Lerner},
    title = {On an estimate of Calderón-Zygmund operators by dyadic positive operators},
    journaltitle = {J. Anal. Math.},
    year = {2013},
    volume = {121},
    number = {},
    pages = {141-161}
}

@article{BernicotFreyPetermichl,
    author = {F. Bernicot and D. Frey and S. Petermichl},
    title = {Sharp weighted norm estimates beyond Calderón-Zygmund theory},
    journaltitle = {Anal. PDE},
    year = {2016},
    volume = {9},
    number = {5},
    pages = {1079-1113}
}

@article{sawyer,
    author = {E. Sawyer and R. L. Wheeden},
    title = {Weighted Inequalities for Fractional Integrals on Euclidean and Homogeneous Spaces},
    journaltitle = {Am. J. Math.},
    year = {1992},
    volume = {114},
    number = {4},
    pages = {813-874}
}

@article{sawyermaximal,
    author = {E. Sawyer},
    title = {A characterization of a two-weight norm inequality for maximal operators},
    journaltitle = {Studia Math.},
    year = {1982},
    volume = {75},
    number = {},
    pages = {1-11}
}

@article{uribe-moen,
    author = {D. Cruz-Uribe and K. Moen},
    title = {A Fractional Muckenhoupt-Wheeden Theorem and its Consequences},
    journaltitle = {Integr. Equ. Oper. Theory},
    year = {2013},
    volume = {76},
    number = {},
    pages = {421-446}
}

@article{fefferman-muckenhoupt,
    author = {C. Fefferman and B. Muckenhoupt},
    title = {Two Nonequivalent Conditions for Weight Functions},
    journaltitle = {Proc. Am. Math. Soc.},
    year = {1974},
    volume = {45},
    number = {1},
    pages = {99-104}
}

@article{conde-alonso,
    author = {J. Conde-Alonso and A. Culiuc and F. Di Plinio and Y. Ou},
    title = {A sparse domination principle for rough singular integrals},
    journaltitle = {Anal. PDE},
    year = {2017},
    volume = {10},
    number = {5},
    pages = {1255-1284}
}

@article{Petermichl,
    author = {S. Petermichl},
    title = {Dyadic shifts and a logarithmic estimate for Hankel operators with matrix symbol},
    journaltitle = {C. R. Acad. Sci.},
    year = {2000},
    volume = {330},
    number = {6},
    pages = {455-460}
}

@article{Hytonen,
    author = {T. Hytönen},
    title = {The sharp weighted bound for general Calderón-Zygmund Operators},
    journaltitle = {Ann. of Math.},
    year = {2012},
    volume = {175},
    number = {3},
    pages = {1473-1506}
}

@article{LernerA2,
    author = {A. Lerner},
    title = {A simple proof of the \(A_2\) conjecture},
    journaltitle = {Int. Math. Res. Not.},
    year = {2013},
    volume = {14},
    number = {},
    pages = {3159-3170}
}

@incollection{MariaSparseRevolution,
  author      = {M. C. Pereyra},
  title       = {Dyadic Harmonic Analysis and Weighted Inequalities: The Sparse Revolution},
  editor      = {A. Aldroubi, C. Cabrelli, S. Jaffard and U. Molter},
  booktitle   = {New Trends in Applied Harmonic Analysis},
  publisher   = {Birkhäuser},
  volume      = {2},
  address     = {},
  year        = 2019,
  pages       = {159-239},
  chapter     = {},
}

@article{LernerOmbrosi,
    author = {A. Lerner and S. Ombrosi},
    title = {Some Remarks on the Pointwise Sparse Domination},
    journaltitle = {J. Geom. Anal.},
    year = {2020},
    volume = {30},
    number = {},
    pages = {1011-1027}
}

@article{Carleson,
    author = {L. Carleson},
    title = {On convergence and growth of partial sums of Fourier series},
    journaltitle = {Acta Math.},
    year = {1966},
    volume = {116},
    number = {},
    pages = {135-157}
}

@article{Cabre1,
    author = {X. Cabré and X. Ros-Oton},
    title = {Regularity of Stable Solutions up to Dimension 7 in Domains of Double Revolution},
    journaltitle = {Commun. Partial Differ. Equ.},
    year = {2013},
    volume = {38},
    number = {1},
    pages = {135-154}
}

@article{Cabre2,
    author = {X. Cabré and X. Ros-Oton},
    title = {Sobolev and isoperimetric inequalities with monomial weights},
    journaltitle = {J. Differ. Equ.},
    year = {2013},
    volume = {255},
    number = {11},
    pages = {4312-4336}
}

@article{SteinWeiss,
    author = {E. Stein and G. Weiss},
    title = {Fractional Integrals on \(n\)-dimensional Euclidean Space},
    journaltitle = {J. math. mech.},
    year = {1958},
    volume = {7},
    number = {4},
    pages = {503-514}
}

@article{MuckenhouptWheeden,
    author = {B. Muckenhoupt and R. Wheeden},
    title = {Weighted Norm Inequalities for Fractional Integrals},
    journaltitle = {Trans. Am. Math. Soc.},
    year = {1974},
    volume = {192},
    number = {},
    pages = {261-274}
}

@article{MuckenhouptHardy,
    author = {B. Muckenhoupt},
    title = {Hardy's inequality with weights},
    journaltitle = {Stud. Math.},
    year = {1972},
    volume = {44},
    number = {1},
    pages = {31-38}
}

@article{Sawyer2,
    author = {E. Sawyer},
    title = {A two weight weak type inequality for fractional integrals},
    journaltitle = {Trans. Amer. Math. Soc.},
    year = {1984},
    volume = {281},
    number = {1},
    pages = {339-345}
}

@article{Wheeden,
    author = {R. Wheeden},
    title = {A characterization of some weighted norm inequalities for the fractional maximal function},
    journaltitle = {Stud. Math.},
    year = {1993},
    volume = {107},
    number = {3},
    pages = {257-272}
}

@article{UribeMoen2,
    author = {D. Cruz-Uribe and SFO and K. Moen},
    title = {One an Two Weight Norm Inequalities for Riesz Potentials},
    journaltitle = {Illinois J. Math.},
    year = {2013},
    volume = {57},
    number = {1},
    pages = {295-323}
}


\clearpage
\fancyhead[L]{}
\fancyhead[R]{Index}
\printindex

\end{document}